%% file: semicontinuity16.tex
\documentclass[11pt,letterpaper]{amsart}
  
\usepackage{amsmath,amssymb,amsthm,amsfonts,verbatim}
\usepackage{hyperref}
\usepackage{enumerate}
\usepackage{xcolor}
\usepackage{graphicx}

\theoremstyle{plain}
\newtheorem{theorem}{Theorem}[section]
\newtheorem*{theorem*}{Theorem}
\newtheorem{maintheorem}{Theorem}
\newtheorem{question}[theorem]{Question}
\newtheorem{proposition}[theorem]{Proposition}

\newtheorem{lemma}[theorem]{Lemma}
\newtheorem*{lemma*}{Lemma}
\newtheorem{claim}[theorem]{Claim}

\newtheorem{corollary}[theorem]{Corollary}

\theoremstyle{definition}
\newtheorem{definition}[theorem]{Definition}

\newtheorem*{convention*}{Convention}

\renewcommand{\epsilon}{\varepsilon}
\newcommand{\dmo}{\DeclareMathOperator}

\newcommand{\ZZ}{\mathbb{Z}}

\dmo{\Span}{span}
\dmo{\Diff}{Diff}
\dmo{\Homeo}{Homeo}
\dmo\im{im}
\dmo\tw{tw}
\dmo\id{id}
\dmo\Fix{Fix}
\dmo\Stab{Stab}
\dmo\Mcg{Mcg}
\dmo{\inte}{Inte}

\newcommand{\cd}{\mathcal{C}^\dagger}

\title[Rotation sets and axes in the fine curve graph]{Rotation sets and axes in the fine curve graph for torus homeomorphisms}

\author{Sebastian Hensel}
\author{Frédéric Le Roux}

\begin{document} 
\sloppy

\maketitle

\begin{abstract}  
  We expand the dictionary between the action of a torus homeomorphism
  on the fine curve graph and its rotation set.  More precisely, we
  show that the fixed points at infinity of a loxodromic element
  determine the rotation set up to scale.
  A key ingredient is a metric version of the classical WPD property
  from geometric group theory.

  As a consequence we find new stable criteria for positive scl, and
  for two homeomorphisms to generate a free group, and we provide a
  Tits alternative for groups of torus homeomorphisms.
\end{abstract}

\footnotesize
\tableofcontents
\normalsize

\section{Introduction}

\subsection{Rotation sets and algebra}
The aim of the present paper concerns the group $\mathrm{Homeo}_0(T^2)$ of homeomorphisms of the two-dimensional torus that are isotopic to the identity.
Our results relate the \emph{rotation set}, a classical invariant from dynamical systems, to the algebra of this group, via its action on the \emph{fine curve graph}. We will recall the definition of the rotation set in section~\ref{sec:rotation-set} below; for the moment, let us just say that it is a compact and convex subset of the plane, associated to a given element of our group $\mathrm{Homeo}_0(T^2)$, which is defined up to integer translations. 
We state three results relating the rotation set to the algebra of the group. The first theorem is a kind of Tits alternative; the second gives a criterion for two elements to generate a free group; the third one for positive stable commutator length.

\begin{maintheorem}
\label{thm:tits}
Let $G$ be a subgroup of $\mathrm{Homeo}_0(T^2)$ which contains some element $f$ whose rotation set $\mathrm{Rot}(f)$ has interior. Then one of the following possibilities occurs:

\begin{enumerate}
\item $G$ contains a free group, generated by some power $f^n$ of $f$ and some conjugate of $f^n$;
\item $G$ or a subgroup of index $2$ of $G$ fixes every backward shadowing equivalence class of $f$, or every forward shadowing class of $f$.
\end{enumerate}
\end{maintheorem}
The \emph{forward shadowing class} of a point $x$ is the set of points whose positive orbit stays a bounded distance from the positive orbit of $f$ in the universal cover. This is a non-trivial equivalence relation, in the sense that it contains uncountably many distinct equivalence classes, see section \ref{sec:tits} below for more details.

We mention that there are other Tits alternatives for homeomorphism
groups of surfaces \cite{HX, Dagger3}, which assume much stronger
conditions on an element of the group -- namely being a (pseudo-)Anosov --
but also find stronger constraints in the case where no free group is
found -- namely fixing a (singular) foliation up to index $2$.

\begin{maintheorem}
\label{thm:free-group}
  Suppose $f,g \in \mathrm{Homeo}_0(T^2)$ so that the rotation sets $\mathrm{Rot}(f), \mathrm{Rot}(g)$ 
  have interior, and are not homothetic. Then there exists $N>0$ such that $f^N$ and 
  $g^N$ generate a free group, all elements of which have a rotation set with non-empty interior.
\end{maintheorem}
Here two sets $C,C'$ are said to be homothetic if there exist $\lambda \in \mathbb{R} \setminus\{0\}$ and $v \in \mathbb{R}^2$ such that $C'=\lambda C + v$.
In the next theorem the set $C$ is said to have a point symmetry if there is some vector $v$ such that $C = -C + v$.
\begin{maintheorem}\label{thm:scl}
  Suppose $f \in \mathrm{Homeo}_0(T^2)$ so that the rotation set $\mathrm{Rot}(f)$ has
  interior, and $\mathrm{Rot}(f)$ has no point symmetry. Then $\mathrm{scl}(f) > 0$ in $\mathrm{Homeo}_0(T^2)$.
\end{maintheorem}

Here is the definition of the stable commutator length ($\mathrm{scl}$).
Given a group $G$ generated by a subset $\mathcal{S}$, the \emph{word length} $\ell_\mathcal{S}(g)$ of $g\in  G$ is the smallest number of elements of $\mathcal{S}$ whose product is $g$. The \emph{stable word length} is 
\[
s\ell_\mathcal{S}(g) = \lim \frac{1}{n} \ell_\mathcal{S}(g^n)
\]
where the limit always exists by sub-additivity. If the group $G$ is \emph{perfect}, \emph{i.e.} generated by the set $\mathcal{C}$ of commutators, then $\ell_\mathcal{C}(g)$ and $s\ell_\mathcal{C}(g)$ are denoted $\mathrm{cl}(g), \mathrm{scl}(g)$ and called the (stable) commutator length. By Bavard duality, positivity of $\mathrm{scl}(g)$ is equivalent to the existence of a \emph{homogeneous quasimorphism} $\Phi$ such that $\Phi(g) \neq 0$ (\cite{bavard}). The identity component of the group of homeomorphisms of any compact manifold is known to be perfect (\cite{Anderson}), in particular the definition applies to our group $\mathrm{Homeo}_0(T^2)$.

 By Kwapisz realization theorem \cite{kwapisz_realise}, every polygon with all vertices having rational coordinates (for short, \emph{rational polygon}) can be realized as the rotation set of 
some element of $\mathrm{Homeo}_0(T^2)$, and the proof gives an explicit construction as an axiom A diffeomorphism. This provides plenty of examples to illustrate both theorems. Note that a symmetric polygon has an even number of sides; in particular, triangles are never symmetric.
The easiest example satisfying Theorem~\ref{thm:scl} is obtained for a homeomorphism whose rotation set is the triangle $(0,0), (1,0), (0,1)$; we will give a construction close to that of Kwapisz in section~\ref{ssec:triangle}. On the other hand every symmetric quadrilateral $Q$ (parallelogram!) which is rational is the rotation set of some $f$ which has a power conjugate to its inverse, which easily implies $\mathrm{scl}(f) = 0$, see section~\ref{ssec:parallelograms}. This shows some kind of optimality for Theorem~\ref{thm:scl}.
Also note that by the continuity property of the rotation set, Theorem~\ref{thm:scl} describes an open set in $\mathrm{Homeo}_0(T^2)$, and Theorem~\ref{thm:free-group} an open set of its cartesian square.

The existence of elements with positive scl is not new: in~\cite{Dagger1}, Bowden, Hensel and Webb proved this in $\mathrm{Homeo}_0(\Sigma)$ for every surface $\Sigma$ of genus at least one. This was the original motivation for the introduction of the fine curve graph (see below), the elephant in the room in Theorem~\ref{thm:free-group} and~\ref{thm:scl}. The proof in~\cite{Dagger1} was rather constructive but did not provide any dynamical feature of the exhibited elements.
In~\cite{Dagger3}, the same authors provide more specific examples, namely (relative) pseudo-Anosov maps. We refer the reader to these previous papers for more context and motivation.
The following corollary is new.
\begin{corollary}\label{coro:positive_scl_id}
Positive scl happens arbitrarily close to the identity in $\mathrm{Homeo}_0(T)$.
\end{corollary}
\begin{proof}[Proof of Corollary~\ref{coro:positive_scl_id}]
Take any $f$ satisfying the hypotheses of Theorem~\ref{thm:scl}, e.g. whose rotation set is a rational triangle. For $m>0$, let $B_m = m\mathrm{Id}$, and consider the rescaling $f_m = B_m^{-1} f B_m$ (see Lemma~\ref{lem:rescaling} below). 
The sequence $(f_m)$ converges uniformly to the identity.
On the other hand the rotation set of $f_m$ is homothetic to that of $f$, thus it also satisfies the hypotheses, and the conclusion, of Theorem~\ref{thm:scl}, namely $\mathrm{scl}(f_m) >0$.
\end{proof}

Finally, note that Theorem~\ref{thm:scl}, together with Kwapisz' construction, also provides $C^\infty$-diffeomorphisms whose stable commutator length is positive within the group of $C^\infty$-diffeormophisms, since the commutator length in a subgroup is obviously not less than the commutator length in the group.

\subsection{Rotation sets and the fine curve graph}
The \emph{fine curve graph} $\cd(\mathbb{T}^2)$, introduced by Bowden, Hensel and Webb in~\cite{Dagger1}, is the graph whose vertices are essential simple closed curves, and edges join curves that are either disjoint or have only one intersection point, at which they cross transversely. This graph is Gromov-hyperbolic (\cite[Theorem 3.10]{Dagger1}). 
The group $\mathrm{Homeo}_0(T^2)$ acts naturally on this graph, and identifies with its automorphism group (\cite{LMPVY,LW}). The \emph{translation length} $\mathrm{tl}(f)$ of $f \in \mathrm{Homeo}_0(T^2)$ is defined as
\[
\mathrm{tl}(f) = \lim \frac{1}{n} d_{\cd(T^2)}(\alpha, f^n \alpha),
\]
where $d_{\cd(T^2)}$ denotes the distance in the graph, the limit exists by sub-additivity and does not depend on the choice of $\alpha$. An element $f$ is said to \emph{act loxodromically} if $\mathrm{tl}(f)>0$. The starting point for our work is the following result by Bowden, Hensel, Mann, Militon and Webb, who discovered that the action on the fine curve graph is intimately related to the rotation set.
\begin{theorem*}[{\cite[Theorem~1.3]{Dagger2}}]
	Let $f: T \to T$ be a torus homeomorphism which is isotopic to the identity. Then the following are equivalent:
	\begin{enumerate}[i)]
		\item $f$ acts loxodromically on the fine curve graph $\cd(T)$, and
		\item The rotation set $\mathrm{Rot}(f)$ has nonempty interior.
	\end{enumerate}
\end{theorem*}
We remark that the theorem is still true in higher genus if one replaces the rotation set by the ergodic rotation set (\cite{GM}), and that both conditions are equivalent to $f$ defining a pseudo-Anosov mapping class relative to a finite set of periodic points. Neither of these are relevant for the current work.

A Gromov hyperbolic space $X$ has a natural boundary $\partial_\infty X$, which roughly speaking represents the different ways to go to infinity (see section~\ref{sec:hyperbolic}). A loxodromic element $f$ acting on the hyperbolic space $X$ has exactly two fixed points at infinity, an attracting fixed point $\xi^+(f)$ that attracts all points of $X$ and a repulsive fixed point $\xi^-(f)$ that attract all points in the past.
We will prove that the attracting fixed point determines the rotation set up to homothety.
\begin{maintheorem}\label{thm:boundary-determines-rotation-set}
Let $f,g \in \mathrm{Homeo}_0(T^2)$ be two homeomorphisms acting loxodromically on the fine curve graph. Assume $\xi^+(f)= \xi^+(g)$. Then the \emph{scaled rotation sets}
\[ \frac{1}{\mathrm{tl}(f)}\mathrm{Rot}(f) \text{ and } \frac{1}{\mathrm{tl}(g)}\mathrm{Rot}(g) \]
are translates of each other.
\end{maintheorem}
The existence of free groups claimed by Theorem~\ref{thm:free-group} follows easily from theorem~\ref{thm:boundary-determines-rotation-set} by the classical ping-pong argument (see~\cite{delaHarpe}).
In section~\ref{sec:examples} we show how to determine the scaled rotation sets on some interesting examples.  In particular, we also get a free group generated by two elements having the same rotation set but distinct scaled rotation sets. By using a different invariant (the \emph{twist number}, see section~\ref{ssec:ingredients-WPD} below), we also get free groups generated by elements sharing the same scaled rotation sets.

To get positive stable commutator length, we need a stronger property.
In \cite{BF}, Bestvina and Fujiwara define an equivalence relation on loxodromic elements of the isometry group of a Gromov-hyperbolic space, namely $f \sim g$ whenever $g$ has conjugate elements whose repulsive and attracting fixed points in $\partial_\infty X$ are arbitrarily close respectively to the repulsive and attracting fixed points of $f$ (see section~\ref{ssec:Fujiwara} for more details).
The following statement generalizes Theorem~\ref{thm:boundary-determines-rotation-set}.
\begin{maintheorem}\label{thm:equivalence}
Let $f,g \in \mathrm{Homeo}_0(T^2)$ be two homeomorphisms acting loxodromically on the fine curve graph. Assume $f \sim g$. Then the scaled rotation sets of $f$ and $g$
are translates of each other.
\end{maintheorem}
We will actually get the stronger result that, roughly speaking, the rotation set depends semi-continuously on the attracting fixed point. For a precise statement, see Theorem~\ref{thm:semi-continuity} below. We can now get positive scl.
\begin{proof}[Theorem~\ref{thm:scl} from Theorem~\ref{thm:equivalence}]
Let $f \in \mathrm{Homeo}_0(T^2)$ whose rotation set has non-empty interior and no point symmetry. Since $\mathrm{Rot}(f^{-1}) = -\mathrm{Rot}(f)$, and $\mathrm{tl}(f^{-1}) = \mathrm{tl}(f)$, the absence of point symmetry yields that
the scaled rotation sets of $f$ and $f^{-1}$ are not translates of each other.
From Theorem~\ref{thm:equivalence} we get that $f \not \sim f^{-1}$. The result now follows from the work of Bestvina-Fujiwara \cite{BF}, see section~\ref{ssec:Fujiwara} below.
\end{proof}

\subsection{The metric WPD property}
The heart of our argument is the discovery that loxodromic elements of the fine curve graph share some kind of weak proper discontinuity property that we describe now.
This property is a natural extension of Bestvina-Fujiwara WPD property (\cite{BF}, see also \cite{osin}) to the setting of Rosendal's geometric group theory \cite{Rosendal}, which enlarges geometric group theory to enclose some non locally compact groups.

Let $G$ be a topological group, which for simplicity we assume to be connected. 
Call a subset $A$ of $G$ \emph{bounded} if for every neighborhood $V$ of the identity, there exists $n$ such that $A \subseteq V^n$, where $V^n$ denotes the set of products of at most $n$ elements of $V$.\footnote{These sets are called \emph{relatively (OB)} by Rosendal.}
Assume $G$ acts by isometries on a metric space $X$. For every $x,y \in X$ and $r>0$, we define the set $Z_r(x,y)$ of elements $h$ of $G$ such that $d(x, hx) \leq r$ and $d(y, hy) \leq r$.

Let $f\in G$. Assume $f$ acts loxodromically, that is, 
\[
\lim \frac{1}{n} d(x, f^nx) > 0
\]
for some (every) $x \in X$. 
\begin{definition}
We say that $f$ \emph{satisfies the metric WPD property} if for every $x\in X$ and every $r>0$, there exists $n_0$ such that the set
$Z_r(x, f^{n_0} x)$ is bounded in $G$.
\end{definition}

This definition is particularly relevant when the group $G$ is "locally (OB)", which means that there exists a neighborhood $O$ of the identity which is bounded. Then the word distance for the generating set $O$ is \emph{maximal} in the sense of Rosendal: a subset of $G$ is bounded (according to the above definition) if and only if it is bounded for this distance, if and only if it is bounded for every right-invariant distance which is compatible with the topology.
In this context, Rosendal's result is that these bounded neighborhoods of the identity can compensate for the absence of compact neighborhoods of the identity, and plays their role as "small" generating sets of $G$.
Note that when $G$ is a finitely generated group endowed with the discrete topology, the maximal metric is given by the usual word length, and bounded sets are just finite sets. Thus in this case our metric WPD property, with the appropriate definition of bounded sets, amounts to the classical WPD property.

Mann-Rosendal's theorem states that for every compact manifold, the identity component of the group of homeomorphisms is locally (OB), and the fragmentation norm with respect to any given cover by open balls is quasi-isometric to the maximal metric \cite{MR}. Back to our context, we propose to use the following concrete incarnation of the maximal metric.
For $f,g \in \mathrm{Homeo}_0(T^2)$, define
$$
\widetilde d(f, g) = \min \{d_{\mathcal{C}^0}(\widetilde{f}, \widetilde{g}) \}
$$
where the minimum is taken among all pairs $(\widetilde{f}, \widetilde{g})$ of lifts of $f$ and $g$ to the plane, and $d_{\mathcal{C}^0}$ denotes the usual sup distance in the plane.
It is easy to check that this defines a right-invariant distance on $\mathrm{Homeo}_0(S)$ which we call the \emph{universal cover distance}. This distance coincides with the usual $\mathcal{C}^0$ distance when $f$ and $g$ are close, but its large scale geometry is very different, since the usual distance on a compact manifold is obviously bounded, and $\widetilde d$ is not. Actually it follows from a theorem of Militon \cite{Militon} that $\widetilde d$ is quasi-isometric to the Mann-Rosendal maximal metric on $\mathrm{Homeo}_0(S)$ (see section~\ref{sec:marking} for more details).

\medskip
Our theorems will all be consequences of the following one. 
\begin{maintheorem}\label{thm:metric-WPD}
Every element of $\mathrm{Homeo}_0(S)$ acting loxodromically satisfies the metric WPD property.
\end{maintheorem}

Another consequence of this metric WPD property is
the following characterization of attracting fixed points of loxodromic elements.
\begin{corollary}[Metric characterization of fixed points of loxodromic]\label{cor:fixed-point-charac}
Let $f,g \in \mathrm{Homeo}_0(T^2)$ be acting loxodromically on the fine curve graph. Then $\xi^-(f) = \xi^-(g)$ if and only if  
there exists a sequence of positive numbers $(m_n)_{n \geq 0}$ such that the sequence 
$(f^{n} g^{-m_n} )_{n \geq 0}$ is bounded in the group $\mathrm{Homeo}_0(T^2)$.
\end{corollary}
Theorem~\ref{thm:boundary-determines-rotation-set} above is actually an easy consequence of this corollary.

The following corollaries provides dynamical interpretation; the last one is the key to our Tits alternative.
\begin{corollary}\label{cor:fixed-point-dyn}
Let $f,g \in \mathrm{Homeo}_0(T^2)$ be acting loxodromically on the fine curve graph.
 Assume $\xi^-(f) = \xi^-(g)$. Then the forward shadowing classes of $f$ and $g$ coincide.
\end{corollary}

\begin{corollary}\label{cor:fixed-point-stab}
Let $f \in \mathrm{Homeo}_0(T^2)$ be acting loxodromically on the fine curve graph.
Let $h \in \mathrm{Homeo}_0(T^2)$, and choose lifts $\widetilde f, \widetilde h$ of $h$ and $f$ to the universal cover. Then the following are equivalent:
\begin{enumerate}
\item $h(\xi^-(f)) = \xi^-(f)$,
\item the sequence of sup norms 
\[
\left(\|\widetilde f^n - \widetilde f^n \widetilde h \|_\infty\right)_{n \geq 0}
\]
is bounded.
\end{enumerate}
Under these conditions, $h$ preserves each forward shadowing class of $f$.
\end{corollary}

In the previous corollaries, it may seems paradoxical that the \emph{repulsive} fixed point on the fine curve graph is related to the \emph{forward} shadowing relation on the torus. This comes from the fact that the action on the fine curve graph is a left action, whereas the natural metric on the group is right invariant. Namely, if we have a bound, say, on the $C^0$ distance between $f^{n}$ and $g^{m}$ for large positive numbers $m$ and $n$, then we have the same bound on $f^{n}g^{-m}$, this will translate into a bound in the fine curve graph between a given curve $\alpha$ and $f^{n}g^{-m} \alpha$, and thus also between $f^{-n} \alpha$ and $g^{-m}\alpha$.

Inhyeok Choi has independently proved that relative pseudo-Anosov
homeomorphisms have metric WPD in every genus, see \cite{Choi}. His
proof uses very different methods than our approach, using an explicit
description of the fixed points at infinity as singular foliations
on the surface from \cite{Dagger3}.

\subsection{Ingredients for metric WPD}\label{ssec:ingredients-WPD}

The strategy for proving the metric WPD property goes through the \emph{fine marking graph}, whose vertices are fundamental domains of the torus homeomorphic to the unit square, and whose edges correspond to moving one side of the square in the fine curve graph. This graph is a combinatorial model for our group $\mathrm{Homeo}(T^2)$: indeed,  the orbit map of the natural group action is a quasi-isometry. Thus, bounding the distance in the group, as required by the WPD property, amounts to bounding the distance in the fine marking graph.

Inspired by the marking graph hierarchy formula of Masur-Minsky \cite{MM2} in the classical curve graph setting, we get an upper bound for the distance in the fine marking graph, in terms of (1) the distance in the fine curve graph between the curves corresponding to the sides of the fine markings, and (2) the sum of the \emph{twist numbers} in suitable annuli  (Lemma~\ref{lem:good-marking-bound}). Here the twist number in an annulus $A$ between two curves that cross $A$ is the distance in the fine arc graph of $A$ between the pieces of the curves that cross $A$. In other words, this is the fine version, in genus 1, of \emph{subsurface projections}. Another key ingredient is a fine version of the Bounded Geodesic Image Theorem (Lemmas~\ref{lem:bgit1} and~\ref{lem:bgit2}), from which we deduce that twist numbers between two curves along a given loxodromic axis are bounded independently of the annulus where we project (Lemma~\ref{lem:thick-axes}). The WPD property is then an easy consequence of these ingredients.

The twists numbers seem to be of independent interest: indeed, taking the supremum on all annuli provides a new conjugacy invariant on $\mathrm{Homeo}_0(T^2)$	that allows to distinguish elements sharing both the same rotation sets and translation length. We will derive the first properties of this invariant, including some semi-continuity, and show that it can take essentially all possible values by computing it on a classical family of examples.

\medskip A first proof of Theorem~\ref{thm:boundary-determines-rotation-set} involved \emph{train tracks} via Masur-Minsky nesting lemma and the fact that iterates of a given curve determines the rotation set (see~\cite{Videos}). We were not able to prove positive stable commutator length along the same lines.

\subsection{Organization of the paper}
Sections~\ref{sec:rotation-set},~\ref{sec:hyperbolic} and~\ref{sec:fine-cg} are preliminaries dealing respectively with the rotation set, hyperbolic geometry and the fine curve graph.
The fine marking graph and its relation with the group $\mathrm{Homeo}(T^2)$ is described in section~\ref{sec:marking}.
In section~\ref{sec.twist-numbers} we introduce the fine twist numbers, prove the bounded geodesic image theorems, and deduce that twist numbers are bounded along loxodromic axes.
In Section~\ref{sec:c-dagger-alpha} we discuss a variant of the fine curve graph and prove technical distance bounds which are the main ingredient for a distance formula in the fine marking graph that is proved in Section~\ref{sec:marking-distance-bound}.
We prove the metric WPD property in section~\ref{sec:WPD}.
The theorems relating the rotation set and the axis are proved in section~\ref{sec:main}.
The forward shadowing relation is studied in section~\ref{sec:tits}, which also contains the proof of the Tits alternative.
The last section~\ref{sec:examples} is dedicated to describing several examples, and also contains the definition of the twist number of a homeomorphism.

\section{The rotation set}\label{sec:rotation-set}

\subsection{Definition}\label{subsec:definition-rotation-set}
We recall the definition and main properties of the rotation set, following~\cite{MZ}.
Here the torus is $T^2 = \mathbb{R}^2 /\mathbb{Z}^2$.
Let $f\in \mathrm{Homeo}_0(T^2)$, and fix a lift $\widetilde f$ of $f$: this is an element of the group  $\mathrm{Homeo}(\mathbb{R}^2, \mathbb{Z}^2)$ of homeomorphisms of the plane that commutes with integer translations. Consider the function $D({\widetilde f})$: 
\[
\begin{array}{rcl}
\mathbb{R}^2 & \longrightarrow & \mathbb{R}^2 \\
 \widetilde x & \longmapsto & \widetilde{f}(\widetilde x) - \widetilde x.
\end{array}
\]
This function descends to a function $T^2 \to \mathbb{R}^2$ that we still denote $D(\widetilde f)$. 
\begin{definition}
The rotation set $\mathrm{Rot}(\widetilde f)$ is the set of vectors $v$ for which there exist a sequence $(x_k)$ of points in the torus, and a sequence $(n_k)$ of integers tending to $+\infty$, such that 
\[
v = \lim \frac{1}{n_k} D(\widetilde f^{n_k}) (x_k).
\]
\end{definition}
In this definition, whenever the sequence $(x_k)$ is constantly equal to some point $x$, and the sequence $(n_k)$ may be taken to be the sequence of all integers, then $v$ is also said to be the rotation vector of $x$, and we write $\mathrm{Rot}_{\widetilde f}(x) = v$.
In particular, a fixed point $x$ of $f$ has a rotation vector $\mathrm{Rot}_{\widetilde f}(x) = D(\widetilde f)(x)$ which is an integer.
Choosing another lift $\widetilde f + (p,q)$ of $f$ has the effect of translating the rotation set by $(p,q)$. Thus the rotation set $\mathrm{Rot}(f)$ is defined as a subset of the plane, defined modulo integer translations.

The rotation set has the following properties. 

\begin{theorem*}[Misiurewicz-Ziemian, \cite{MZ}]
The rotation set $\mathrm{Rot}(\widetilde f)$ is a convex and compact subset of the plane. Furthermore, for every compact $\Delta \subset \mathbb{R}^2$ that projects onto the torus, the sequence 
$$
\frac{1}{n}\widetilde{f}^n(\Delta)
$$
converges to $\mathrm{Rot}(\widetilde f)$ for the Hausdorff distance.
\end{theorem*}

 The rotation set  is \emph{homogeneous}, that is, the formula $\mathrm{Rot}(\widetilde f^n) = n \mathrm{Rot}(\widetilde f)$ holds for every integer $n$.
It is invariant under conjugacy in the group $\mathrm{Homeo}(\mathbb{R}^2, \mathbb{Z}^2)$. If $A \in \mathrm{SL}_2(\mathbb{Z})$, then
\[
\mathrm{Rot}(A \widetilde{f} A^{-1}) = A (\mathrm{Rot}(\widetilde{f})).
\]
Let $B$ be a matrix with integer coefficients which is invertible as a real matrix, but with determinant maybe not equal to $1$. The homeomorphism $B^{-1} \widetilde{f} B$ still belongs to the group $\mathrm{Homeo}(\mathbb{R}^2, \mathbb{Z}^2)$, and thus it induces an element of $\mathrm{Homeo}_0(T^2)$; beware that this element depends on the choice of the lift $\widetilde f$.
 When applied with $B = B_m = m\mathrm{Id}$ for some $m>1$, we call this element a \emph{rescaling} of $f$ of order $m$. The previous formula still holds for $A = B^{-1}$.
\begin{lemma}[\cite{kwapisz_realise}]\label{lem:rescaling}
For every matrix $B$ with integer coefficients and non-zero determinant,
\[\mathrm{Rot}(B^{-1} \widetilde{f} B) = B^{-1} (\mathrm{Rot}(\widetilde{f})).\]
\end{lemma}

The following lemma states that the set of realizable rotation sets is invariant under the rational affine group. It is probably well known, we will use it to construct examples in section~\ref{sec:examples}.
\begin{lemma}
\label{lem:GAQ}
The set 
\[
\{\mathrm{Rot}(\widetilde f), \widetilde f \text{ lift  of some } f\in \mathrm{Homeo}_0(T^2) \}
\]
is invariant under the action of the affine group $\mathrm{GA}_2(\mathbb{Q})$.

More precisely, given any lift $\widetilde f$ of some $f \in \mathrm{Homeo}_0(T^2)$, and $\Phi \in \mathrm{GA}_2(\mathbb{Q})$, we can find an invertible matrix $A$ with integer coefficients, a vector $v \in \mathbb{Q}^2$ and a positive integer $p$ such that the map $\widetilde g = A^{-1}\widetilde f^p A$ commutes with the translation by vector $v$, and
$\mathrm{Rot}(\widetilde g + v) = \Phi(\mathrm{Rot}(\widetilde f))$.
\end{lemma}
\begin{proof}
Fix a lift $\widetilde f$ of some element $f\in \mathrm{Homeo}_0(T^2)$.
Given a positive integer $m$, let $B = m\mathrm{Id}$. By Lemma~\ref{lem:rescaling} and the homogeneity, the map $\widetilde g = B^{-1} \widetilde f^{m} B$ has the same rotation set as $\widetilde f$, and if we pick some vector $v = \frac{1}{m}(p,q)$, then $\widetilde g$  commutes with the translation by $v$. We deduce that $\mathrm{Rot}(\widetilde g + v) = \mathrm{Rot}(\widetilde f) + v$. Thus the set of realizable rotation sets is invariant under the action of rational translations.

It remains to check invariance under the linear group $\mathrm{GL}_2(\mathbb{Q})$.
Let $A$ be a diagonal matrix with diagonal entries $\frac{p}{q}, \frac{p'}{q}$, and $B$ the diagonal matrix with diagonal entries $p'q, pq$. Then 
\[
\mathrm{Rot}(B^{-1} \widetilde f^{pp'}B) = A(\mathrm{Rot}(\widetilde f)).
\]
 Since $\mathrm{GL}_2(\mathbb{Q})$ is generated by $\mathrm{SL}_2(\mathbb{Z})$ together with such matrices $A$, the result follows.
\end{proof}

\subsection{Rotation set bounds}
We start with a well known lemma giving an easy upper bound for the rotation set. Here, and throughout, we denote by $\mathrm{Conv}(X)$ the convex hull of a set $X$ in the plane.
\begin{lemma}[Upper bound lemma]\label{lem:upper-bound-rotation-set}
For every lift $\widetilde f$ of $f \in \mathrm{Homeo}_0(T^2)$, for every $n>0$, we have
\[
\mathrm{Rot}(\widetilde f) \subset \mathrm{Conv}\left(\frac{1}{n} (\widetilde f^n- \mathrm{Id} )([0,1]^2) \right).
\]
\end{lemma}

\begin{proof}
 By homogeneity it suffices to prove this inclusion for $n=1$, and then apply it to $f^n$. 
We use the function $D(\widetilde f) = \widetilde f -\mathrm{Id}$ from section~\ref{subsec:definition-rotation-set}.
By definition, the rotation set $\mathrm{Rot}(\widetilde f)$ is made of limits of vectors of the form
\[
\frac{1}{n} D(\widetilde f^{n}) (x) = \frac{1}{n}\sum_{k=0}^n D(\widetilde f)(\widetilde f^k(x)).
\]
By $\mathbb{Z}^2$ invariance, each term of the sum belongs to $D(\widetilde f)([0,1]^2)$. Thus this vector is included in the convex hull of $D(\widetilde f)([0,1]^2)$, and so is $\mathrm{Rot}(\widetilde f)$.
\end{proof}

The following corollary will be useful for constructing examples.
\begin{corollary}
For every respective lifts $\widetilde f, \widetilde g$ of $f,g \in \mathrm{Homeo}_0(T^2)$, we have
\[
\limsup \frac{1}{n}\mathrm{Rot}(\widetilde f \widetilde g^n) \subseteq \mathrm{Rot}(\widetilde g).
\]
\end{corollary}
\begin{proof}
Let $\varepsilon>0$. Note that $\widetilde f$ is at bounded distance from the identity, and thus $\widetilde{f} \widetilde{g}^n$ is at bounded distance from $\widetilde g^n$.
By the Misiurewicz-Ziemian theorem above, for every $n$ large enough, the set
\[
E_n = \frac{1}{n}(\widetilde f \widetilde g^n - \mathrm{Id}) ([0,1]^2)
\]
is included in the $\varepsilon$-neighborhood $N_\varepsilon$ of $\mathrm{Rot}(g)$. 
Consider such an integer $n$. Since the rotation set is convex, $N_\varepsilon$ is also convex, and thus it also contains $\mathrm{Conv}(E_n)$.
We apply the lemma to $fg^n$ to get 
\[
\frac{1}{n} \mathrm{Rot}(\widetilde f \widetilde g^n) \subseteq \mathrm{Conv}(E_n) \subseteq N_\varepsilon.
\]
\end{proof}

Unfortunately we do not know if the equality holds in the corollary.
\begin{question}\label{ques:projective-cv}
Let $f,g \in \mathrm{Homeo}_0(T^2)$ and assume their rotation sets have non empty interior. Do we always have that 
\[
\lim \frac{1}{n} \mathrm{Rot}(fg^n) = \mathrm{Rot}(g) ?
\]
\end{question}
The difficulty to answer this question positively comes from the lack of reasonable lower bounds. This also explains why we could only prove semi-continuity, and not continuity, in our last theorem (see section~\ref{subsec:semi-continuity} and the question inside).

\bigskip

The next lemma will be useful to control the rotation sets in the main proofs.
Remember from the introduction that a sequence in $\mathrm{Homeo}_0(T^2)$ is said to be \emph{bounded} if it is bounded in the sense of Mann-Rosendal, or equivalently, bounded for the distance $\widetilde d$ defined as the minimal $C^0$ distance between lifts to the plane $\mathbb{R}^2$.
The next lemma shows how such boundedness may provide information on the rotation set up to translation. We emphasize that the translations in this statement are by any vector in $\mathbb{R}^2$, and not necessarily integer vectors. 
\begin{lemma}[Boundedness and the rotation set]~
\label{lem:uc-distance}

\begin{enumerate}[i)]
\item Let $f \in \mathrm{Homeo}_0(T^2)$.
Let $\varepsilon>0$, $D>0$.
Then there is $n_0>0$ such that for every $n \geq n_0$, for every $g \in \mathrm{Homeo}_0(T^2)$ such that $\widetilde d(f^n, g) \leq D$, the set $\frac{1}{n}\mathrm{Rot}(g)$ has a translate included in the $\varepsilon$-neighborhood of $\mathrm{Rot}(f)$. 
  \item Let $f,g \in \mathrm{Homeo}_0(T^2)$.
Assume that there is a sequence  $(m_n)_{n \geq 0}$ in $\mathbb{N}$ such that the sequence  $(f^ng^{-m_n})_{n \geq 0}$ is bounded. Assume furthermore that $\lim \frac{m_n}{n}$ exists and is a positive number, and denote it $\lambda$.
  Then $\lambda \mathrm{Rot}(g)$ and $\mathrm{Rot}(f)$ are translates of each other.
  \end{enumerate}
\end{lemma}

\begin{proof}[Proof of Lemma~\ref{lem:uc-distance}]
Let us prove i).
Choose some lifts $\widetilde f$ of $f$, and let $n_0$ be large enough so that $\frac{D}{n_0}< \varepsilon$ and for every $n \geq n_0$ the set
$$
\frac{1}{n} ({\widetilde f}^{n}- \mathrm{Id} )([0,1]^2)
$$
is included in the $\varepsilon$-neighborhood of $\mathrm{Rot}(\widetilde f)$. Let $n\geq n_0$, and consider some $g \in \mathrm{Homeo}_0(T^2)$ such that $\widetilde d(f^n, g) \leq D$: there is some lift $\widetilde g$ of $g$ which is within distance $D$ from $\widetilde f^n$. By the above upper bound lemma,
$$
\frac{1}{n}\mathrm{Rot}(\widetilde g) \subset \mathrm{Conv}\left(\frac{1}{n} (\widetilde g- \mathrm{Id} )([0,1]^2) \right).
$$
By the distance estimate, this last set is included in the $\frac{D}{n}$-neighborhood of 
$$
 \mathrm{Conv}\left(\frac{1}{n} (\widetilde f^n- \mathrm{Id} )([0,1]^2) \right)
$$
which is itself included in the $\varepsilon$-neighborhood of $\mathrm{Rot}(\widetilde f)$ (note that the $\varepsilon$-neighborhood of a convex set is convex).

Now let us prove ii). Denote $D$ the bound for the $\widetilde d$-distance from 
$(f^ng^{-m_n})_{n \geq 0}$ to the identity in $\mathrm{Homeo}_0(T^2)$. Consider some $\varepsilon>0$. Let $n_0$ be given by i): for every $n\geq n_0$, since $\widetilde d(f^n, g^{m_n}) \leq D$, we get that $\frac{1}{n}\mathrm{Rot}(g^{m_n}) = \frac{m_n}{n}\mathrm{Rot}(g)$ has a translate included in the $\varepsilon$-neighborhood of $\mathrm{Rot}(f)$. Letting $n$ tend to $+\infty$, and then $\varepsilon$ to $0$, we get that $\lambda \mathrm{Rot}(g)$ has a translate included in $\mathrm{Rot}(f)$.
For the reverse inclusion, we apply i) with the roles of $f$ and $g$ exchanged. We get some $n_0$ so that, for every $n \geq n_0$, $\frac{1}{m_n}\mathrm{Rot}(f^{n}) = \frac{n}{m_n}\mathrm{Rot}(f)$ has a translate included in the $\varepsilon$-neighborhood of $\mathrm{Rot}(g)$. We conclude as above that $\frac{1}{\lambda} \mathrm{Rot}(f)$ has a translate included in $\mathrm{Rot}(g)$.
\end{proof}

\section{Hyperbolic preliminaries}\label{sec:hyperbolic}
In this section, we collect some preliminaries on Gromov hyperbolic
spaces. We refer the reader to e.g. \cite{BH} for details on hyperbolic spaces;
we will only collect the required results. Most of them are
well-known, but especially in Subsections~\ref{sec:bds-conv}
or~\ref{sec:tl-estimates} we discuss results specifically adapted to our
situation.

\subsection{Basic definitions and the boundary}\label{ssec:hyperbolic-boundary}
Recall that a geodesic metric space $X$ is called
\emph{$\delta$--hyperbolic} if every geodesic triangle is
$\delta$-thin, i.e. if every side is contained in the $\delta$-neighbourhood of the other two sides.

For a basepoint $b \in X$, the \emph{Gromov product} is defined as
\[ (x\cdot y)_b = \frac{1}{2}\left( d(b,x) + d(b,y) - d(x,y)
  \right). \]
If $X$ is $\delta$--hyperbolic, and $\alpha$ is a geodesic joining $x$ to $y$, then
\[ (x \cdot y)_b \leq d(b, \alpha) \leq (x\cdot y)_b + 3\delta. \]
One can also think of the Gromov product as the distance where
geodesics from $b$ to $x,y$ start to diverge.

\smallskip Using the Gromov product, we can define the boundary at
infinity of a hyperbolic space. Namely, call a sequence $(x_i)$ \emph{admissible} if
\[ (x_i\cdot x_j)_b \to \infty, \]
and two sequences $(x_i), (y_i)$ equivalent if 
\[ (x_i\cdot y_j)_b \to \infty. \] Then the \emph{Gromov boundary}
$\partial_\infty X$ is the set of equivalence classes of admissible sequences.
An admissible sequence $(x_i)$ defining the boundary point $\xi$ is said to converge to $\xi$.

If $X$ is proper, then $\partial_\infty X$ can be identified with the
set of geodesic rays up to bounded distance -- but since the spaces we
deal with are very much not proper, it is not clear if every boundary
point is reached by an actual geodesic ray in our setting. See e.g. \cite[Section~2.A]{LongTan} for discussion.

\subsection{Geodesics and quasigeodesics}\label{sec:bds-conv}
A central property of geodesics in hyperbolic spaces is the following
stability (see e.g. \cite[Chapter III.H, Theorem 1.7]{BH})
\begin{lemma*}[Morse lemma]
  Suppose $X$ is a $\delta$--hyperbolic space. Then for any $K>0$ there is a
  $B=B(K,\delta)$ with the following property. Suppose that $\gamma$ a
  (finite) geodesic, and $q$ is a $K$--quasigeodesic with the same
  endpoints. Then the Hausdorff distance between $\gamma, q$ is at
  most $B$.
\end{lemma*}
Note that the Morse lemma implies that two quasigeodesics with the same endpoints are bounded Hausdorff distance apart (with a bound depending only on the hyperbolicity and quasigeodesic constants).

Another result we will use several times is the following
local-to-global property for quasigeodesics. Variants of this result
are well-known; see e.g. \cite[Lemma~4.2]{Min05} for a version
immediately implying our required result.
We adopt the usual convention that $[xy]$ denotes any geodesic joining two given points $x$ and $y$. 
\begin{lemma}[Local-to-global]\label{lem:local-to-global}
  Let $X$ be a $\delta$-hyperbolic space. For every $B$ there are
  constants $L,K$ so that the following is true: if
  \[ g = [x_1 x_2] \ast \cdots \ast [x_k x_{k+1}] \]
  is a concatenation of geodesics, so that for every $1\leq i\leq k$ we have
  \begin{enumerate}
  \item $d(x_i, x_{i+1}) \geq L$, and
  \item $(x_i, x_{i+2})_{x_{i+1}} \leq B$,
  \end{enumerate}
  then $g$ is a $K$--quasigeodesic.
\end{lemma}

\begin{lemma}[Fellow traveling geodesics]\label{lem:ftg}
  Let  $\alpha, \alpha',\beta,\beta'$ be four points in a $\delta$-hyperbolic space, and let $b$ be such that
$$
d(\alpha, \beta) \leq b, \quad d(\alpha', \beta') \leq b.
$$
Let $\ell>0$, and $\alpha_\ell, \beta_\ell$ respectively denote points of the geodesics $[\alpha, \alpha']$, $[\beta,\beta']$ that are distance $\ell$ from $\alpha, \beta$.
Then $d(\alpha_\ell, \beta_\ell) \leq 5b+4\delta$.
\end{lemma}
\begin{proof}
  \begin{figure}
  \centering
  \def\svgwidth{0.5\textwidth}
  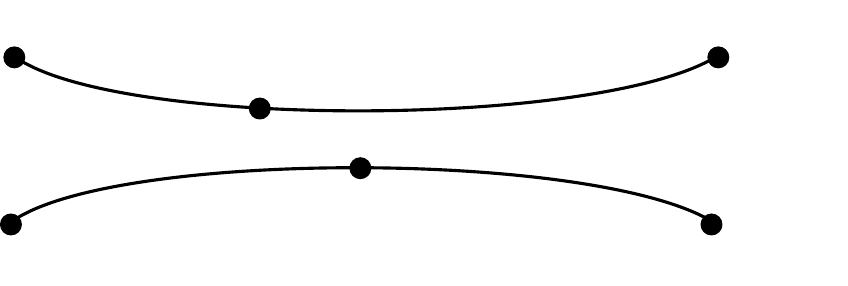
  \caption{Lemma~\ref{lem:ftg}}
\end{figure}
Note that the geodesics $[\alpha, \alpha']$, $[\beta,\beta']$ have $b$-close end-points, so their length differ by no more than $2b$.
If $\ell \leq 2\delta+b$, then the lemma follows at once from the triangular inequality. Similarly, we also get the result if $d(\alpha_\ell, \alpha') \leq 2\delta + b$, since in this case the above remark on the length of the geodesics gives $d(\beta_\ell, \beta') \leq 2\delta + 3b$.

From now on we assume that the distance from $\alpha_\ell$ to both endpoints $\alpha, \alpha'$ is more than $2\delta+b$.
Consider the geodesic quadrilateral with vertices $\alpha, \beta, \beta', \alpha'$.
Quadrilaterals in a $\delta$-hyperbolic space are $2\delta$ thin. Thus the point $\alpha_\ell \in [\alpha, \alpha']$ is $2\delta$ close to some point in the union of the three other sides. Since $\alpha_\ell$ is not too close to $\alpha$ neither to $\alpha'$, this point has to be on the geodesic $[\beta, \beta']$. Let us denote it $\beta_{\bar \ell}$, with $\bar \ell = d(\beta, \beta_{\bar \ell})$; and we remember that $d(\alpha_\ell,\beta_{\bar \ell}) \leq 2\delta$. From the bound on end-points of the geodesics $[\alpha, \alpha_\ell]$ and $[\beta, \beta_{\bar \ell}]$, we get that $\ell$ and $\bar \ell$ differ by no more than $b+2\delta$. Thus $d(\beta_{\bar \ell}, \beta_\ell) \leq b + 2 \delta$, and now the triangular inequality gives $d(\alpha_\ell, \beta_\ell) \leq b + 4\delta$.
\end{proof}

\begin{lemma}[Fellow traveling quasigeodesics]\label{lem:mid-quasigeodesics}
Let $b>0$.
Let $\alpha, \alpha',\beta,\beta'$ be four points in a $\delta$-hyperbolic space, and assume that
\[
d(\alpha, \beta) \leq b,\quad d(\alpha', \beta') \leq b.
\]
Consider a first $K$-quasigeodesic from $\alpha$ to $\alpha'$, and
a second one from $\beta$ to $\beta'$, and denote $M=M(K)$ the Morse constant.
Finally, let $\hat \alpha, \hat \beta$ be two points respectively on the first and second quasigeodesic, and denote
 $k = d(\alpha, \hat \alpha )$,  $\ell = d(\beta, \hat \beta)$.
Then $d(\hat \alpha, \hat \beta) \leq M' + |\ell-k|$, where $M'= 5b + 4\delta + 4M$.
\end{lemma}
\begin{proof}
  \begin{figure}
  \centering
  \def\svgwidth{0.5\textwidth}
  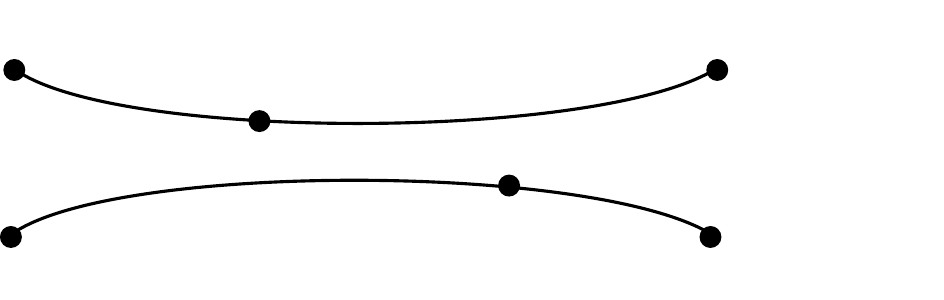
  \caption{Lemma~\ref{lem:mid-quasigeodesics}}
\end{figure}
Consider a geodesic $[\alpha, \alpha']$.
By the Morse lemma, the point $\hat \alpha$ is distance at most $M$ from some point $\alpha_{\bar k}$ on this geodesic, where $\bar k = d(\alpha, \alpha_{\bar k})$. Let $\beta_{\bar k}$ denote the point on a geodesic $[\beta, \beta']$ which is distance $\bar k$ from $\beta$. The previous lemma tells us that $d(\alpha_{\bar k}, \beta_{\bar k}) \leq 5b + 4\delta$.
On the other hand, we may also apply the Morse lemma to the point $\hat \beta$ and get some point $\beta_{\bar \ell}$ on the geodesic $[\beta, \beta']$, which is distance at most $M$ from $\hat \beta$, where $\bar \ell= d(\beta, \beta_{\bar \ell})$.
Note that both points $\beta_{\bar k}, \beta_{\bar \ell}$ lies on the geodesic $[\beta, \beta']$, and thus $d(\beta_{\bar k}, \beta_{\bar \ell}) = |\bar \ell - \bar k|$.
Furthermore $\bar k$ differ from $k$ by no more than $M$, and similarly for $\bar \ell$ and $\ell$, and thus $ |\bar \ell - \bar k| \leq |\ell -k| + 2M$.
Now the result follows from the triangular inequality by relating $\hat \alpha$ to $\hat \beta$ successively through the points $\alpha_{\bar k}, \beta_{\bar k}, \beta_{\bar \ell}$.
\end{proof}

\begin{corollary}\label{cor:midpoints}
For every $\delta, K, b$ there is $M'$ such that the following holds.
Let $f,g$ be two isometries having invariant quasi-axes which are $K$-quasigeodesic paths respectively through points $\alpha,\beta$.
Let $n, m$ be positive integers such that 
$$
d(\alpha,\beta) \leq b, \quad d(f^{2n} \alpha, g^{2m} \beta) \leq b.
$$
Then $d(f^{n} \alpha, g^{m} \beta) \leq M'$.
\end{corollary}

\begin{figure}
  \centering
  \def\svgwidth{0.5\textwidth}
  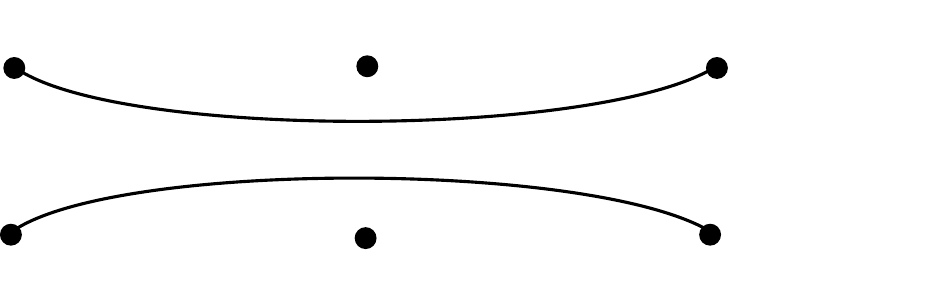
  \caption{Corollary~\ref{cor:midpoints}}
\end{figure}
\begin{proof}
Denote $M=M(K)$ the Morse constant.
Let $k = d(\alpha, f^{n} \alpha) = d(f^{n} \alpha, f^{2n} \alpha)$, $\ell = d(\beta, g^{m}\beta) = d(g^m\beta, g^{2m} \beta)$. By the distance hypotheses, the lengths of the two geodesics $[\alpha, f^{2n}\alpha]$, $[\beta, g^{2m}\beta]$ differ by at most $2b$.
Using this estimate and the closest point projection of $f^n\alpha$ on $[\alpha, f^{2n}\alpha]$, and likewise for $\beta$, we get $|\ell - k| \leq 2M+b$.
Now we may apply the previous lemma to get the wanted estimate, with $M' = 5b + 4\delta + 4M + (2M+b)$.
\end{proof}

We next exploit the fact that in hyperbolic space (quasi-)geodesics
fellow-travel uniformly close before diverging. The toy example is the
following situation.  Suppose that $g_1, g_2$ are two geodesics
starting in a common point $p$. Now suppose that the geodesic $\alpha$
between $g_1(s), g_2(t)$ stays distance at least $L$ from
$p$. Consider the first point $g_1(r)$ on $g_1$ which is \emph{not}
$\delta$--close to $g_2$. Such a point is (by $\delta$--hyperbolicity)
$\delta$--close to the geodesic $\alpha$ and thus $r \geq L -
\delta$. In other words, initial segments of $g_1, g_2$ of length
$L-\delta$ are $\delta$--fellow-traveling. In particular we have that
$d(g_1(s), g_2(s)) \leq 2 \delta$ for all $s \leq L-\delta$.

The next lemma is a somewhat technical fellow-traveling statement for
quasigeodesics generalizing this situation.
\begin{figure}
  \centering
  \def\svgwidth{0.5\textwidth}
  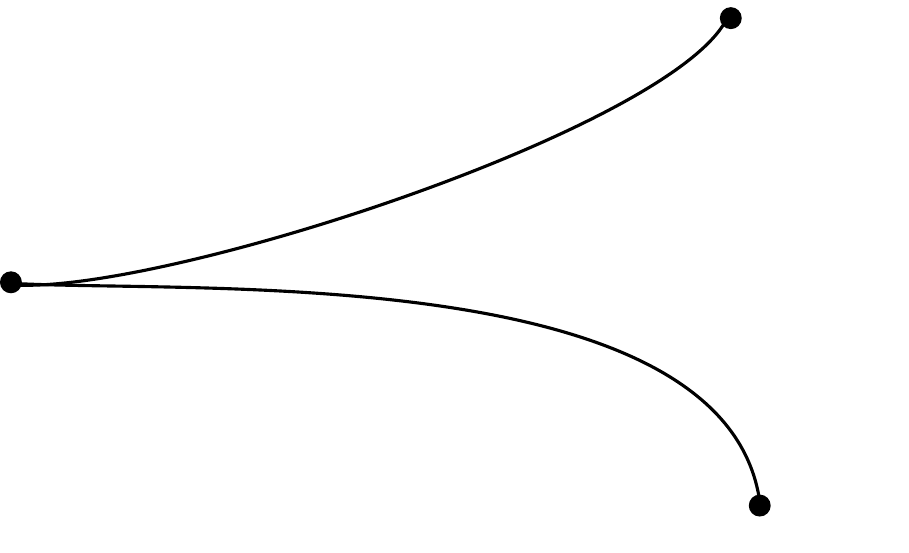
  \caption{The motivating example for Lemma~\ref{lem:fellow-travel-tech}}
\end{figure}
\begin{lemma}\label{lem:fellow-travel-tech}
Fix some $\delta, K, b_0$.
There exists $b=b(K, \delta, b_0)$ such that the following holds. For every $N$, there exists $L$ such that for every $K$-quasigeodesics $(\alpha_i)_{i=0, \dots, i_0}$, $(\beta_i)_{i=0, \dots, i_0}$, if 
\begin{enumerate}
\item $d(\alpha_0, \beta_0) \leq b_0$,
\item $d(\alpha_0, [\alpha_{i_0} \beta_{i_0}]) \geq L$
\end{enumerate}
then $N \leq i_0$, and there exists $N' \leq i_0$ such that $d(\alpha_N, \beta_{N'}) \leq b$.
\end{lemma}
\begin{proof} Let us fix $\delta, K, b_0, N$, and consider $K$-quasigeodesics $(\alpha_i)_{i=0, \dots, i_0}$, $(\beta_i)_{i=0, \dots, i_0}$ with $d(\alpha_0, \beta_0) \leq b_0$.
	Extend $\beta_i$ by setting $\beta_{-1} = \alpha_0$. The resulting path is a $K'$--quasigeodesic for a $K'$ depending only on $K$ and $b_0$. Let $M = M(\delta, K, b_0)$ be the Morse constant for $K'$-quasigeodesic. Hence, $(\alpha_i), (\beta_i), i \leq i_0$ are $M$-close to geodesics $g_1, g_2$ joining $\alpha_0$ to $\alpha_{i_0}, \beta_{i_0}$. There is a constant $\hat{N}$ so that any $\alpha_i, \beta_i$ with $i \leq N$ is $M$-close to a point on $g_j$ which has distance at most $\hat{N}$ from $\alpha_0$.
	
	Choosing $L$ large enough, we can guarantee (as in the discussion before the lemma) that initial segments of length $\hat{N}$ of $g_1, g_2$ stay $2\delta$--close. This also implies that any $\alpha_N$ is $2\delta+2M$ from some $\beta_{N'}$. Thus $b=2\delta+2M$ has the desired property.
\end{proof}

\subsection{Isometry types and translation length}\label{sec:tl-estimates}
We next turn to isometries of a $\delta$-hyperbolic space. The key quantity here is the \emph{(asymptotic) translation length} of an isometry $\phi$:
\[ \mathrm{tl}(\phi) = \lim_{n\to \infty}\frac{d(x_0, \phi^nx_0)}{n}\]
for some basepoint $x_0$. It is immediate that $\mathrm{tl}(\phi)$ does not depend on the choice of $x_0$. We then say that
\begin{enumerate}
	\item $\phi$ is \emph{loxodromic} (or \emph{hyperbolic}) if $\mathrm{tl}(\phi)>0$.
	\item $\phi$ is \emph{parabolic} if $\mathrm{tl}(\phi)=0$, but orbits $\{ \phi^n(x_0), n \in \mathbb{N} \}$ have infinite diameter, and
	\item $\phi$ is \emph{elliptic} if orbits $\{ \phi^n(x_0), n \in \mathbb{N} \}$ have finite diameter.
\end{enumerate}
If $\phi$ is loxodromic, then there is a bi-infinite quasigeodesic invariant under $\phi$ -- in fact, one can simply take the images $\phi^n(x_0), n \in \mathbb{Z}$. We call any invariant quasigeodesic a \emph{(quasi)axis} for $\phi$. 
Every isometry extends continuously to the Gromov boundary $\partial_\infty X$, and the extension of a loxodromic isometry $\phi$ has exactly two fixed points, denoted $\xi^\pm \phi$, such that every orbit $(\phi^n x)$ converges to $\xi^-\phi$ when $n$ tends to $-\infty$, and to $\xi^+\phi$ when $n$ tends to $+\infty$.

Observe that there is no reason to expect that $\phi$ has an invariant geodesic axis. In fact, if the space $X$ is a graph, this would imply that $\mathrm{tl}(\phi)$ would be an integer, since $\phi$ would translate along this axis by an integral amount.

The following lemma shows that we can, however, find a quasi-axis of definite quality if the translation length is not too small. It is likely known to experts, but we were not able to find a discussion in the literature, so we include a proof.

In our application of the following lemma, the space $X$ will be a graph, where the minimizing point $x$ needed in the first point always exists. The hypothesis on $x$ may be weakened by only requiring that $x$ is a $\delta$-almost minimizing point, and this gives a meaningful statement in the general case. Details are left to the reader.
\begin{lemma}\label{lem:K-for-axis}
Let $X$ be a $\delta$-hyperbolic metric space $X$. 

\begin{enumerate}
\item 
There are $K = K(\delta), L=L(\delta)$ such that the following holds. Let $f$ be a loxodromic element with $\mathrm{tl}(f) \geq L$. Choose $x \in X$ which  minimizes $d(x, f(x))$, and let $[x, f(x)]$ denote any geodesic path from $x$ to $f(x)$. Then
\[
\bigcup_{n \in \mathbb{Z}} f^n ([x, f(x)])
\]
is a $K$-quasigeodesic.
\item
Let $K>0$, and assume $x$ is on a $K$-quasigeodesic which is invariant under $f$.
Then for every positive $p,n$ we have
\[
n \left(d(x, f^p(x)) - 2M\right) \leq d(x, f^{np}(x) )
\]
where $M= M(\delta, K)$ is the Morse constant. In particular,
\[
\frac{1}{p} \left(d(x, f^p(x)\right) - 2M) \leq \mathrm{tl}(f).
\]
\end{enumerate}
\end{lemma}

Note that by sub-additivity we always have, for every $x\in X$ and $p>0$,
\[
\mathrm{tl}(f) \leq \frac{1}{p} (d(x, f^p(x))
\]
and together with the last part of the lemma, this gives the precise convergence speed towards the translation length, in terms of the quality of a quasigeodesic axis containing $x$.
\begin{proof}
\begin{enumerate}
\item Assuming that $d(x,f(x)) > 14\delta$, we first show that
  \[ (f^{-1}x\cdot f(x))_x < 4\delta. \] Namely, assume that this
  would be false. Consider a geodesic triangle with corners
  $f^{-1}(x),x,f(x)$. By the assumption that
  $(f^{-1}x\cdot f(x))_x \geq 4\delta$, the sides
  $g = [x,f^{-1}(x)], h = [x, f(x)]$ are $2\delta$--fellow-traveling
  for initial subsegments of length $4\delta$. Let $y \in g$ be the
  endpoint of such a segment. Note that it is $2\delta$--close to a
  point $z_1 \in h$ (which thus has $d(z_1, x) \leq 6\delta$).
  \begin{figure}[h!]
    \centering
    \def\svgwidth{0.8\textwidth}
    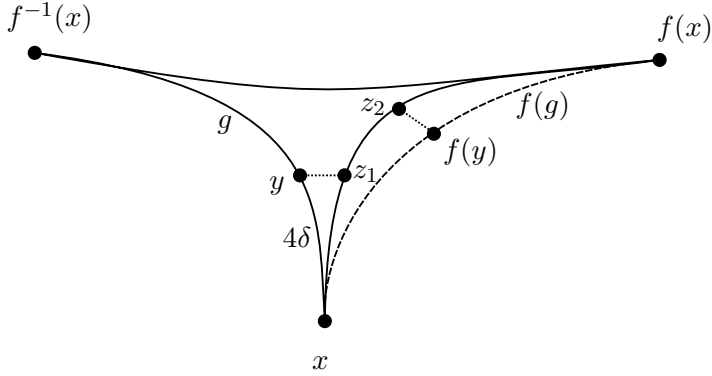
    \caption{The proof of Lemma~\ref{lem:K-for-axis}~(1). The geodesic $f(g)$ is drawn dashed.}
  \end{figure}
    
  The image $f(y)$ lies on the geodesic $f(g)$ which joins $f(x)$ to
  $x$, and is thus $\delta$--close to a point $z_2$ on $h$, and we
  have $d(z_2, f(x)) \leq 5\delta$. By our assumption that
  $d(x,f(x)) > 14\delta$, the point $z_1$ is between $x$ and $z_2$ on
  $h$, and since we have
  $d(z_1, x) \geq 2\delta, d(z_2, f(x)) \geq 3\delta)$ we get
  \[ d(z_1, z_2) \leq d(x, f(x)) - 5\delta, \]
  implying
  \[ d(y, f(y)) \leq d(x, f(x)) - 5 \delta + 3\delta < d(x,f(x)), \]
  which contradicts minimality of $x$. 
  
  Given the estimate on $(f^{-1}x\cdot f(x))_x$, we apply the
  local-to-global Lemma~\ref{lem:local-to-global} to $B=4\delta$, we
  obtain constants $K,L_0$. Setting $L = \max(L_0, 14\delta)$ this
  lemma then implies that
  \[
    \bigcup_{n \in \mathbb{Z}} f^n ([x, f(x)])
  \]
  is a $K$-quasigeodesic.
\item Let $n,p >0$, and let us prove that
  $d(x, f^{np}(x) ) \geq n (d(x, f^p(x)) - 2M)$. By the Morse Lemma,
  there is a point $\pi$ on the geodesics $[x, f^{np}x]$ which is
  distance less than $M$ from $f^p(x)$. We get
\begin{eqnarray*}
  d(x, f^{np}(x)) &=& d(x, \pi) + d(\pi, f^{np}(x))\\
                  &\geq &(d(x, f^{p}(x)) - M) + (d(f^p(x), f^{np}(x)) -M) \\
                  &=& d(x, f^{(n-1)p}(x)) + d(x, f^{p}(x)) - 2M
\end{eqnarray*}
and now the estimate follows by induction on $n$. 

Dividing by $np$ and letting $n$ tends to $+\infty$ yields the estimate on the translation length.
\end{enumerate}
\end{proof}

\begin{corollary}\label{cor:fellow-travel}
For every $\delta, K,T$ there is some $B=B(\delta, K,T)$ with the following property.
If $(\beta_i)$ is a $K$-quasigeodesic which is invariant under the isometry $g$, with $\mathrm{tl}(g) \leq T$, then every $\beta_i$ is distance at most $B$ from some iterate $g^n(\beta_0)$.
\end{corollary}
\begin{proof}
Applying point (2) of Lemma~\ref{lem:K-for-axis} with $p=1$, we get $d(\beta_0, g(\beta_0)) \leq T + 2M(\delta, K)$. Let $i_0$ be the index such that $\beta_{i_0} = g(\beta_0)$. The $K$-quasi-geodesicity gives a bound on $i_0$ by some function of $K$
and $d(\beta_0, g(\beta_0))$. Then the distance from any $\beta_i$ to the closest $\beta_{n i_0} = g^n(\beta_0)$ is also bounded.
\end{proof}

\subsection{Bestvina-Fujiwara equivalence relation}\label{ssec:Fujiwara}
In this section we (very) briefly summarize results which allow to use actions on hyperbolic spaces to construct unbounded quasimorphisms on groups, or alternatively certify that elements have positive stable commutator length (see the introduction for the definition).

\smallskip
Let $G$ be a group acting on a $\delta$-hyperbolic space, and $f,g \in G$. Assume they act loxodromically, and denote $\xi^\pm(f), \xi^\pm(g)$ their end-points.
Bestvina-Fujiwara (\cite{BF}) define a relation $f \sim g$ if there are copies of the axis of $g$ arbitrarily close to the axis of $f$. More precisely: if for every neighborhoods $V^-, V^+$ of $\xi^-(f), \xi^+(f)$ there exists $h \in G$ such that $h(\xi^-(g)) \in V^-, h(\xi^+(g)) \in V^+$. 


We can now state the criterion: 
\begin{theorem*}[{\cite[Proposition 5]{BF}}]\label{thm:bf}
  Suppose $G$ acts on a $\delta$--hyperbolic space by isometries.
  Assume $f$ acts loxodromically and $f \not \sim f^{-1}$.
Then  $\mathrm{scl}(f) > 0$.
\end{theorem*}

Let us end this section with an easy criterium for Bestvina-Fujiwara equivalence.
Let $f,g$ be acting loxodromically as above, and assume $\xi^+(f) = \xi^+(g)$.
Then for $n$ large enough $f^{-n}$ fixes $\xi^+(g)$ and sends $\xi^-(g)$ arbitrarily close to $\xi^-(f)$. We deduce:
\begin{lemma}\label{lem:bf}
If $\xi^+(f) = \xi^+(g)$ then $f \sim g$.
\end{lemma}

As an immediate consequence, we see that Theorem~\ref{thm:equivalence} (the Bestvina-Fujiwara class determines the scaled rotation set) entails Theorem~\ref{thm:boundary-determines-rotation-set} (the fixed point determines the scaled rotation set).

\section{Fine curve graphs}\label{sec:fine-cg}
In this section, we introduce the second main object of this article,
the \emph{fine curve graph} $\cd(S)$ of a closed, connected, oriented
surface $S$.  The vertices of $\cd(S)$ correspond to simple, essential
curves on $S$. If $S$ has genus $g>1$, we join two vertices by an edge
if the corresponding curves are disjoint. If $S$ is the torus, we also
join vertices corresponding to curves which intersect transversely, and at
most once.  Note that there are by now many variants of the fine curve
graph being studied which have slightly different edge relations (but
all of which are quasi-isometric).

The first key property is the following.
\begin{theorem}[{\cite[Theorem 3.10]{Dagger1}}]
  The fine curve graph is Gromov hyperbolic.
\end{theorem}

An important ingredient of the proof, that we will use again below,
are \emph{bicorn surgery paths}. Given simple closed curves
$\alpha,\beta$, a \emph{bicorn} defined by $\alpha$ and $\beta$ is a
simple, essential closed curve made of an arc $a$ included in $\alpha$
and an arc $b$ included in $\beta$. Observe that not every subarc of
$\alpha$ defines a bicorn.  However, for every pair of transverse curves
$\alpha, \beta$ one can find a quasi-path from $\alpha$ to $\beta$ in
$\cd(S)$ made of bicorns defined by $\alpha$ and $\beta$, where the
subarcs taken from $\alpha$ are strictly decreasing. We call such
quasi-paths \emph{bicorn surgery paths}.
We warn the reader that, strictly speaking, these are not paths in the fine curve graph (but they are paths in the augmented graphs, where edges are put between curves intersecting at most $4$ times).

The key point is that bicorn surgery paths are geometrically
efficient. This follows from work of Rasmussen \cite{Rasmussen} on
nonseparating curve graphs; see \cite[Section 4]{Dagger3} for details.
\begin{lemma}[{\cite[Lemma 4.2]{Dagger3}}]\label{lem:bicorns}
  There is a constant $K>0$ with the following property: suppose
  $\alpha, \beta \subset S$ are two curves intersecting
  transversely. Then if $(a_i)$ is a bicorn surgery path between
  $\alpha$ and $\beta$, there is a reparametrization $i(n)$ so that
  $n \mapsto a_{i(n)}$ is a $K$--quasigeodesic.
\end{lemma}


The following lemma will allow us to consider bicorns for curves that do not intersect transversely. Here we consider vertices of $\cd(S)$ as continuous maps $\mathbb{S}^1 \to S$ up to reparametrization, which allows to endow the set of curves with the quotient topology from the topology of uniform convergence of maps. (Note that this is stronger than the Hausdorff distance between sets; in the lemma the Hausdorff distance would be enough for us.)
\begin{lemma}[Bicorn Perturbation Lemma]\label{lem:non-transverse-bicorns}
Let $\alpha, \beta \in \cd(S)$, and subarcs $a \subset \alpha, b \subset \beta$ such that $a \cup b$ is a bicorn defined by $\alpha$ and $\beta$.
For every $\varepsilon>0$, there exists curves $\alpha', \beta' \in \cd(S)$ which intersect transversely, are respectively $\varepsilon$-close to $\alpha$ and $\beta$, and subarcs $a' \subset \alpha', b' \subset \beta'$ such that $a' \cup b'$ is a bicorn defined by $\alpha'$ and $\beta'$ which is $\varepsilon$-close to $a \cup b$. Furthermore we can assume that $\alpha'=\alpha$.
\end{lemma}
We postpone the proof of the perturbation lemma to the end of this section. We will need the following consequence.
\begin{corollary}\label{cor:morse-arbitrary-bicorn}
There is a constant $M>0$ such that any bicorn defined by any pair $\alpha, \beta \in \cd(S)$ is within distance $M$ of any geodesic $[\alpha, \beta]$.
\end{corollary}
\begin{proof}[Proof of the Corollary]
Let $M$ be the Morse constant associated to $K$-quasi-geodesics for the constant $K$ of Lemma~\ref{lem:bicorns}. Let $a \cup b$ be a bicorn defined by curves $\alpha$ and $\beta$, and apply the Bicorn Perturbation Lemma to get curves $\alpha', \beta'$ intersecting transversely, arbitrarily close to $\alpha, \beta$, defining a bicorn $a'\cup b'$ arbitrarily close to $a\cup b$. Curves that are close have a common disjoint curve, and thus are within distance $2$ in the fine curve graph, thus we get 
\[
d(\alpha, \alpha') \leq 2, \ \ d(\beta, \beta') \leq 2, \ \ d(a\cup b, a' \cup b') \leq 2.
\]
By Lemma~\ref{lem:bicorns} and the Morse Lemma, $a' \cup b'$ is within distance $M$ of any geodesic $[\alpha', \beta']$. On the other hand, by hyperbolicity, the geodesic $[\alpha', \beta']$ is included in the $(\delta+4)$-neighborhood of the geodesic $[\alpha, \beta]$. Thus the corollary holds with constant $M + \delta + 6$.
\end{proof}

\bigskip We will also make use of the \emph{fine arc graph} of the
annulus $A =[0,1] \times S^1$, denoted $\mathcal{A}^\dagger(A)$, whose
vertices are \emph{essential proper arcs} in $A$, that is, images of $[0,1]$
under an embedding that maps $(0,1)$ into the interior of $A$ and $0$
and $1$ to distinct boundary components of $A$, and edges correspond
to disjointness.  The distance between two vertices $a \neq b$ in this
graph may be computed as follows (see~\cite{Dagger2, LW}).

If $a, b$ are arcs in $A$, let $\widetilde{a}, \widetilde{b}$ be lifts
of the two arcs to the universal cover of $A$. Define the \emph{width}
$\mathrm{wd}_A(a, b)$ of $a$ and $b$ in $A$ as the number of distinct
lifts of $a$ which intersect one lift of $b$. We then have
\begin{lemma}\label{lem:distance-arc-graph}
  For any arcs $a,b \subset A$ we have 
  \[ d_{\mathcal{A}^\dagger(A)}(a,b) =  \mathrm{wd}_A(a, b) + 1. \]  
\end{lemma}
In particular, the distance between two arcs that intersect exactly
once is equal to $2$, which is just the observation that two essential
proper arcs that intersect once have a common disjoint arc.

\subsection{Proof of the perturbation lemma}
We start with a short discussion of intersection points for general $C^0$ curves.
 Let $\alpha, \beta$ be two simple closed curves on our surface $S$. Call a point $z \in \alpha \cap \beta$ an \emph{unstable intersection point} if 
there exist a neighborhood $V$ of $z$, and sequences of simple closed curves $(\alpha_n), (\beta_n)$ converging to $\alpha, \beta$ such that for each $n$, $\alpha_n \cap \beta_n \cap V= \emptyset$.
 Here is a more concrete description of unstable intersection points. Assume $\alpha$ and $\beta$ meet at $z$. Consider coordinates $(x,y)$ around the intersection point $z$ such that $\alpha$ locally coincides with the line $y=0$. Then $z$ is an unstable intersection point if and only if there is an open subarc of $\beta$ containing $z$ which is included in $y \geq 0$ or in $y \leq 0$. The following lemma will allow us to replace unstable intersection points with stable ones (where stable means not unstable).
\begin{lemma}
Let $\alpha, \beta$ be two simple closed curves on $S$, and $z\in \alpha \cap \beta$.
Let $\varepsilon>0$. Then there exists a simple closed curve $\beta'$, $\varepsilon$-close to $\beta$, equal to $\beta$ outside the $\varepsilon$-neighborhood of $z$ and such that $z$ is a stable intersection point of $\alpha$ and $\beta'$.
\end{lemma}
\begin{proof}
We may obviously assume that the intersection point $z$ is unstable.
We consider local coordinates as above, and assume that $\beta$ is locally contained in the upper half-plane $y \geq 0$. The new curve $\beta'$ is obtained as the image of $\beta$ under the composition $h_2 h_1$ of two homeomorphisms with support arbitrarily close to $x$. Namely, we first apply $h_1$ to push $\beta$ vertically downwards, so that $h_1 \beta$ meets $y<0$ while still having points in $y > 0$. It is easy to see that $h_1 \beta$ has a stable intersection point $z'$ with $\alpha$ in the local chart. Now push $z'$ horizontally by $h_2$ that preserves $\alpha$ and sends $z'$ to $z$, so that $z$ is a stable intersection point of $\alpha$ and $h_2 h_1 \beta$.
\end{proof}
\begin{proof}[Proof of Lemma~\ref{lem:non-transverse-bicorns}]
Let $\alpha, \beta, a,b, \varepsilon$ as in the lemma. Denote $x,y$ the common endpoints of $a$ and $b$. The first step is to get stable intersection points.
By applying the above lemma independently in some small neighborhoods of $x$ and $y$, we may perturb $\beta$ into a curve $\beta'$ which have stable intersection points with $\alpha$ at $x$ and $y$. Let $b'$ be the subarc of $\beta'$ with endpoints $x,y$ which is close to $b$. The closed curve $a \cup b'$ may be not simple, but $a \cap b'$ is included in the union of the neighborhoods of $x$ and $y$ on which the perturbation took place.
Thus we may shorten $a$ and $b'$ a little bit to get a simple closed curve $a' \cup b''$ which is a bicorn defined by $\alpha$ and $\beta'$, and which is still close to $a \cup b$. Note that the closeness ensures that $a' \cup b''$ is homotopic to $a \cup b$, and thus is still an essential closed curve.
The conclusion of this first step is that it suffices to prove the lemma in the case when the intersection points $x$ and $y$ are stable, which we assume now.
We endow $S$ with a smooth structure. Thanks to the Schoenflies theorem, we may assume that $\alpha$ is smooth with respect to this structure.
It is well known that smooth curves are $C^0$-dense in the set of continuous simple closed curves. Furthermore, among smooth curves, transversality to a given one is a generic property. Thus we may  find $\beta'$, arbitrarily $C^0$-close to $\beta$, which is smooth and transverse to $\alpha$. Since $x$ and $y$ were stable intersection points, $\alpha \cap \beta'$ contains points $x',y'$ respectively arbitrarily close to $x$ and $y$. As above, we can now find a bicorn $a' \cup b'$ defined by $\alpha$ and $\beta'$ and arbitrarily close to $a \cup b$. This completes the proof.
\end{proof}

\section{Fine markings}\label{sec:marking}
In this section we restrict to the torus for simplicity. There are
likely no serious issues generalising the notions and results
discussed here to higher genus; however, we do not want to commit to
a precise definition at this point, and we will explore the geometry
of  marking graphs in higher genus in future work.

\begin{definition}
  \begin{enumerate}[i)]
  \item A \emph{(fine) marking (of the torus)} is a set $\{\alpha, \beta\}$ of
    simple closed curves which intersect once, and transversely.
    We say that $\beta$ is a \emph{dual to $\alpha$}.
  \item We say that a marking $\{\alpha', \beta'\}$ \emph{is obtained
      from $\{\alpha,\beta\}$ by an elementary move} if one of the
    curves is unchanged, say $\alpha = \alpha'$, and $\beta'$ intersects $\beta$
     at most once, and transversely.   
  \item The \emph{fine marking graph} $\mathcal{M}^\dagger$ is the
    graph whose vertices are fine markings, and whose edges correspond
    to elementary moves.
  \item If $m$ is a fine marking, we denote by
    $\mathcal{M}^\dagger(m)$ the subgraph spanned by all markings
    isotopic to $m$.
  \end{enumerate}
\end{definition}
\begin{figure}
  \centering
  \def\svgwidth{0.8\textwidth}
  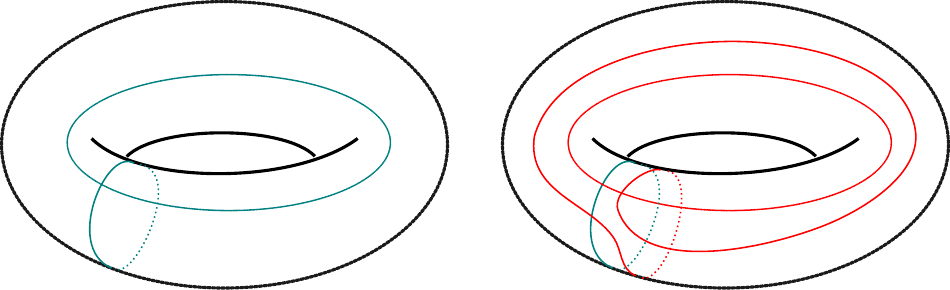
  \caption{An elementary move of fine markings $\{\alpha, \beta\}$ to $\{\alpha, \beta'\}$}
\end{figure}

Thus the vertices of the fine marking graph are edges of the fine curve graphs, and the 
edges corresponds to triangles in the fine curve graph. Note that, by classification of surfaces, the group $\mathrm{Homeo}(T^2)$ acts transitively on the fine marking graph, and that the subgraph $\mathcal{M}^\dagger(m)$ is the orbit of $m$ under the subgroup $\mathrm{Homeo}_0(T^2)$. Edges of this subgraph correspond to disjoint moves, that is, there is an edge from $\{\alpha, \beta\}$ to $\{\alpha, \beta'\}$ if and only if $\beta$ and $\beta'$ are disjoint. The geometry of the subgraph $\mathcal{M}^\dagger(m)$ captures the geometry of the homeomorphism group
in the following sense.
\begin{lemma}\label{lem:geometric-model}
  For any $m_0$, the graph $\mathcal{M}^\dagger(m_0)$ is connected and 
  the orbit map
  \[ \mathrm{Homeo}_0(T^2) \to \mathcal{M}^\dagger(m_0), \quad f
    \mapsto f^{-1}(m_0) \] is a quasi-isometry for the (unique)
  right-invariant quasi-isometric structure on $\mathrm{Homeo}_0(S)$
  in the sense of Mann-Rosendal, and the (intrinsic) metric on
  $\mathcal{M}^\dagger(m_0)$.
\end{lemma}
\begin{proof}[Proof of Lemma~\ref{lem:geometric-model}]
  Let $m_0 = (\alpha_0, \beta_0)$.  We choose a covering of the
  surface by disks $U_i$ which are: the complement $U_1$ of
  $\alpha_0\cup\beta_0$, a small neighbourhood $U_2$ of
  $\alpha_0\cap\beta_0$, and neighbourhoods $U_3, U_4$ of the arcs
  $\alpha_0\cup\beta_0\setminus U_2$.  By the work of Mann-Rosendal
  (\cite{MR}), the geometry of $\mathrm{Homeo}_0(T^2)$ is
  quasi-isometric to the fragmentation norm induced by the cover
  $\{U_i\}$. We denote this norm by $\|\cdot\|$. We denote the metric
  on $\mathcal{M}^\dagger(m_0)$ for brevity by $d$.  By
  right-invariance of the metric on $\mathrm{Homeo}_0(T^2)$ and of the
  action by inverses of maps on $\mathcal{M}^\dagger(m_0)$, it
  suffices to estimate $d(m_0, \phi^{-1} m_0)$ in terms of $\|\phi\|$,
  which we will do now.

  \smallskip First, observe that if $f$ is supported in some $U_i$, then
  \[ d(m_0, fm_0) \leq 4, \] since for each possible support $U_i$ a
  marking $m_i$ can be obtained from $m_0$ by at most two elementary
  moves, which is disjoint from $U_i$ and hence fixed by $f$.
  If now $\phi \in \mathrm{Homeo}_0(T^2)$ is arbitrary, with $\|\phi\| = n$, we
  can write
  \[ \phi = f_1 \cdots f_n, \] with each $f_i$ supported in some
  $U_{m_i}$. We have by the triangle inequality
  \[ d(m_0, \phi m_0) \leq \sum_{i=1}^n d(f_1\cdots f_{i-1}m_0,
    f_1\cdots f_i m_0). \] Since the action on $\mathcal{M}^\dagger(m_0)$
  is isometric, and by the previous remark, we then have
  \[ \sum_{i=1}^n d(f_1\cdots f_{i-1}m_0, f_1\cdots f_i m_0)
    = \sum_{i=1}^n d(m_0, f_i m_0)
    \leq 4n. \]
  Hence, we have
  \[ d(m_0, \phi^{-1}(m_0)) = d(m_0, \phi(m_0)) \leq 4 \|\phi\|. \] 
  In particular, this shows
  that the graph is connected, since the action of
  $\mathrm{Homeo}_0(T^2)$ on $\mathcal{M}^\dagger(m_0)$ is transitive.
  \begin{figure}
    \centering
    \def\svgwidth{0.8\textwidth}
    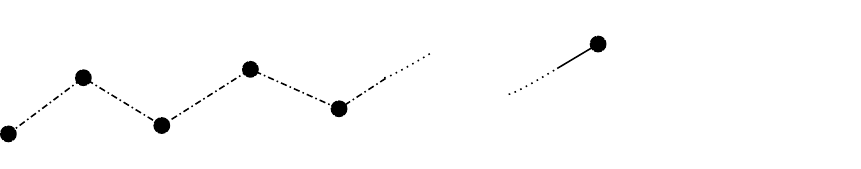
    \caption{Moving in the marking graph. The lines between
      consecutive $f_1\cdots f_im_0$ indicate
      distance at most $4$.}
  \end{figure}

  \smallskip Now, on the other hand, suppose that $m_1, m_0$ differ by
  a single elementary move. Then there is an element
  $f \in \mathrm{Homeo}_0(T^2)$ with $f(m_0) = m_1$ and
  $\|f\| \leq 4$; in fact $f$ can be written as a product of
  homeomorphisms supported on each $U_i$ at most once. Namely, assume
  that $\alpha_0$ is unchanged and $\beta_0$ changes to $\beta_1$,
  with $\beta_0 \cap \beta_1 = \emptyset$. By 
  a modification in the neighbourhood of the intersection point followed by another one 
  in the neighourhood of $\beta_0$, we get a curve
  $f_1 f_0(\beta_0)$ which 
  is disjoint from $\beta_0$. By applying a homeomorphism $f_2$ supported in 
  the neighborhood of $\alpha_0$, we get $f_2 f_1 f_0(\beta_0)$ which additionaly 
  coincides with $\beta_1$ near $\alpha_0$.
  Now a modification on the complement
  of $\alpha_0\cup\beta_0$ (i.e. a homeomorphism $f_3$ of
  fragmentation norm $1$) can send $f_2 f_1 f_0(\beta_0)$ to $\beta_1$. The
  existence of $f_2$ and $f_3$ is an application of a Schoenflies-like
  theorem (see for instance \cite{Siebenmann} for an account of the
  classical case, and \cite{Homma} for a very general case.
  \begin{figure}
    \centering
    \def\svgwidth{0.8\textwidth}
    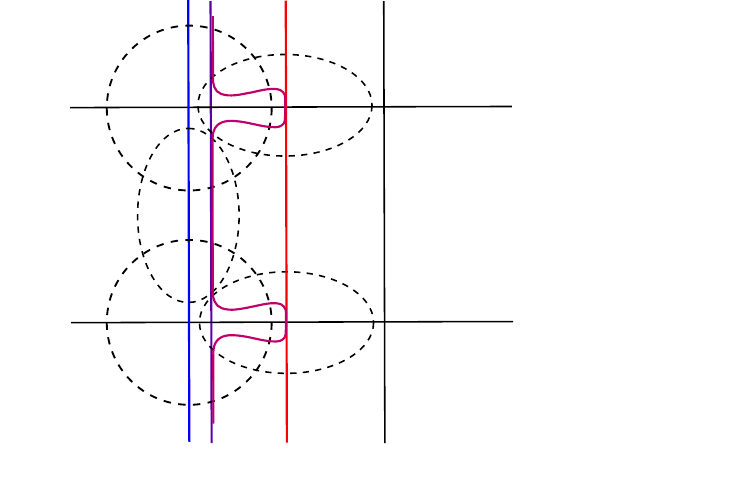
    \caption{Getting one move in the marking graph with a homeomorphism of 
    fragmentation norm 4.}
  \end{figure}

  Let now, $m' = \Psi^{-1}(m_0)$ be a marking obtained by applying an
  arbitrary element $\Psi \in \mathrm{Homeo}_0(T^2)$.
  Denote $n = d(m_0, m') = d(m_0, \Psi (m_0))$, and
  choose a geodesic
  \[ m_0, m_1, \ldots, m_n = \Psi(m_0) \] in the fine marking graph.  By the
  above, there is a homeomorphism $\phi_1$ of fragmentation norm at
  most $4$ with $\phi_1(m_0) = m_1$. Assume inductively that we have
  an element $\phi_i$ of fragmentation norm at most $4i$ so that
  $\phi_i(m_0) = m_i$. Then, $\phi^{-1}_i(m_{i+1})$ is adjacent to
  $m_0$, and so there is an element $f_{i+1}$ of fragmentation norm at
  most $4$ so that $f_{i+1}(m_0) = \phi^{-1}_i(m_{i+1})$. Hence, we have
  \[ \phi_i f_{i+1}(m_0) = m_{i+1}, \] Hence,
  $\phi_{i+1} = \phi_i f_{i+1}$ satisfies the inductive hypothesis for
  $i+1$. We then have that
  \[ \Psi^{-1}\phi_n (m_0) = m_0. \]
  We leave it to the reader to check that a homeomorphism that preserves 
  $m_0$ (but maybe does not fixes it pointwise) has fragmentation norm at most $5$. Therefore 
  $\| \Psi^{-1}\phi_n \| \leq 5$. We thus have
  \[ \|\Psi\| \leq 4n + 5 = 4d(m_0, \Psi^{-1}(m_0)) + 5. \]  
\end{proof}

From now on, we identify the torus $T^2$ with $\mathbb{R}^2/\mathbb{Z}^2$,
and we let $m_0$ be the marking defined by the horizontal and vertical
curve through $0$. For brevity, we write
$\mathcal{M}^\dagger(m_0) =\mathcal{M}^\dagger_0$, since the basepoint marking 
$m_0$ will be fixed throughout.

We then have the following useful consequences, connecting marking
distance to $\mathcal{C}^0$--distance. Remember from the introduction that $\widetilde d$ denotes the distance on $\mathrm{Homeo}_0(T^2)$ obtained by minimizing the usual $C^0$-distance on $\mathrm{Homeo}(\mathbb{R}^2)$ over all pairs of lifts.
\begin{corollary}[$C^0$ bound from markings]
\label{cor:c0-bound}
  For any $D>0$ there is a number $R>0$ with the following property.
  Suppose that $h \in \mathrm{Homeo}_0(T^2)$ satisfies
  \[ d_{\mathcal{M}_0^\dagger}(m_0, hm_0) \leq D. \] Then 
  \[ \widetilde d(\mathrm{id}, {h}) \leq R. \]
\end{corollary}
\begin{proof}
  Cover the torus by finitely many disks $U_i$ which admit finite
  diameter lifts to $\mathbb{R}^2$, and let $r$ be the maximal
  diameter of such lifts. In particular, any homeomorphism supported
  on a disk $U_i$ admits a lift to $\mathbb{R}^2$ which has
  $\mathcal{C}^0$--distance at most $r$ from the identity. Then, if
  $h$ has fragmentation norm at most $n$ (with respect to $U_i$), it
  admits a lift $\widetilde{h}$ of distance at most $rn$ to the
  identity. Now, the corollary follows immediately from
  Lemma~\ref{lem:geometric-model}.
\end{proof}
We note that this corollary proves the easy direction of the
following result due to Emmanuel Militon \cite{Militon}. Here $\mathrm{diam}$ denotes the diameter in the plane.
\begin{theorem}[{\cite[Theorem 2.4]{Militon}}]\label{thm:militon}
  For any choice of covering $\mathcal{U}$ of $T$ there are constants
  $C, C'$ so that for any $f \in \mathrm{Homeo}_0(T)$ we have
  \[ \frac{1}{C}\mathrm{diam}(\widetilde{f}[0,1]^2) - C' \leq \|f\|_\mathcal{U} \leq C\mathrm{diam}(\widetilde{f}[0,1]^2) + C'. \]
\end{theorem}

  It will be convenient for later use to interpret the
  distance in $\mathcal{M}_0^\dagger$ 
  in a slightly different way. Given two homotopic simple
  closed curves $\delta, \delta'$ we define the \emph{relative width}
  $\mathrm{wd}_\delta(\delta')$ as the number of distinct lifts of
  $\delta$ in $\mathbb{R}^2$ which a lift of $\delta'$ intersects.

  Given two markings $m = \{\alpha, \beta\}, m'=\{\alpha', \beta'\}$ in
  the isotopy class of $m_0$, we then define
  \[ d_w(m, m') = \mathrm{wd}_\alpha(\alpha') + \mathrm{wd}_\beta(\beta'). \]
  \begin{lemma}\label{lem:relative-width-distance}
    The metrics $d_{\mathcal{M}_0^\dagger}$ and $d_w$ on
    $\mathcal{M}_0^\dagger$ are equivalent.
  \end{lemma}
  \begin{proof}
    Note that both metrics are invariant under the action of
    $\mathrm{Homeo}_0(T)$, and thus it suffices to show that
    $d(m_0, m), d_w(m_0, m)$ are comparable for all $m$. To this end
    we consider the orbit map
    \[ \mathrm{Homeo}_0(T) \to \mathcal{M}_0^\dagger(T), \] noting
    that we already know that this is a quasi-isometry for the metric
    $d$. Thus, it suffices to show that it also is a quasi-isometry
    for the metric $d_w$. 
    Let $f\in \mathrm{Homeo}_0(T)$, denote $m_0 = \{\alpha, \beta\}$
    and $f(m_0) = \{\alpha', \beta'\}$.
    From definition of relative width, it is
    clear that
    \[ \mathrm{diam}(\widetilde{f}[0,1]^2) \geq \max\{
      \mathrm{wd}_\alpha(\alpha'), \mathrm{wd}_\beta(\beta') \} \]
    Next, note that $\widetilde{f}[0,1]^2$ is contained in a
    (topological) strip bounded by two adjacent lifts of
    $\widetilde{f}(\alpha)$. By definition of the width, this strip is contained in
    a straight strip of width $\mathrm{wd}_\alpha(\alpha')+2$ bounded by
    lifts of $\alpha$, and analogous for $\beta$. This shows
    \[ \mathrm{diam}(\widetilde{f}[0,1]^2) \leq
      \mathrm{wd}_\alpha(\alpha') + \mathrm{wd}_\beta(\beta') + 4. \] 
      Now, the desired claim follows from Theorem~\ref{thm:militon}.
  \end{proof}

\bigskip Finally, we discuss the difference between the geometry of the subcomplex
$\mathcal{M}^\dagger_0$ and the whole complex $\mathcal{M}^\dagger$.
\begin{lemma}\label{lem:distorted-markings}
  The inclusion $\mathcal{M}^\dagger_0 \to \mathcal{M}^\dagger$ is
  exponentially distorted; that is: there is a function $F$ of exponential growth type, so that
  \[ d_{\mathcal{M}^\dagger_0} (m,m') \leq F(d_{\mathcal{M}^\dagger} (m,m'))\] 
  for all $m, m' \in \mathcal{M}^\dagger_0$, and every function with this property has at least exponential growth type.
\end{lemma}
\begin{proof}
  Denote by $R_0 = [0,1]^2$ the fundamental domain corresponding to
  our basepoint marking $m_0$. Observe that there is a $k>0$ so that
  if $m$ is an arbitrary marking homotopic to $m_0$, and $R$ the
  corresponding fundamental domain, then we have
  \[ \mathrm{diam}(R) \leq 3^{d_{\mathcal{M}^\dagger}(m,m_0)}. \] This
  follows since each edge in $ \mathcal{M}^\dagger$ corresponds to a
  move that can only change the corresponding fundamental domain to
  intersect three adjacent copies. This implies the ``at most exponential''
  direction by Lemma~\ref{lem:geometric-model} and Theorem~\ref{thm:militon}.

  To see the other direction, let $A$ be an Anosov map of the torus,
  and $\varphi$ a point push around a curve $\gamma$. Then the conjugate
  $A^n\varphi A^{-n}$ is a point push around $A^n\gamma$, and therefore
  \[ \mathrm{diam}(A^n\varphi A^{-n})(R_0) \] grows exponentially in
  $n$. Hence, by Lemma~\ref{lem:geometric-model}, the distance
  $d_{\mathcal{M}^\dagger_0}(m_0, A^n\varphi A^{-n}m_0)$ also grows
  exponentially. On the other hand, by the triangle inequality, the distance $d_{\mathcal{M}^\dagger}(m_0, A^n\varphi A^{-n}m_0)$ grows linearly. Hence, no subexponential bound of $d_{\mathcal{M}^\dagger_0}$ in terms of $d_{\mathcal{M}^\dagger}$ is possible. 
\end{proof}

As an easy consequence we get the following properness, which says that a set of vertices in the subgraph $\mathcal{M}^\dagger_0$ is bounded if and only if it is bounded in the ambient graph $\mathcal{M}^\dagger$ .
This will be crucial later.
\begin{corollary}\label{cor:properness-markings}
  For any $r$ there is an $R$ so that: for all 
  $m,m' \in \mathcal{M}^\dagger_0$
  \[ d_{\mathcal{M}^\dagger}(m, m') \leq r \quad\Rightarrow\quad
    d_{\mathcal{M}^\dagger_0}(m, m') \leq R. \]
\end{corollary}

The phenomenon we exhibited here is completely analogous to behavior
seen in mapping class groups. Here, Broaddus--Farb--Putman show that
the kernel of the Birman exact sequence is exponentially distorted
\cite{BFP}.  


\subsection{Outlook}
To finish this section, we describe an outlook over the content of the
next three sections. The eventual goal will be to prove a distance
formula bounding the distance between two markings $m, m'$ from above.

Our formula is inspired by Masur-Minsky's hierarchy formula. We begin
by describing the simplest case of this formula, in the case of the
torus. In this case, the (non fine) curve graph is the classical \emph{Farey
  graph}, whose vertices correspond to alternatively rational numbers,
or isotopy classes of simple closed curves on the torus. Edges
correspond to intersection $1$ up to isotopy, and thus a marking
for the torus corresponds exactly to an edge in the Farey graph.

Suppose we are given a geodesic $\alpha_1, \ldots, \alpha_n$ in the
Farey graph. Then $\{\alpha_1, \alpha_2\}$ and
$\{\alpha_2, \alpha_3\}$ are markings which share a curve $\alpha_2$. The duals
$\alpha_1, \alpha_3$ differ by a power of a Dehn twist about
$\alpha_2$. Moving from $\{\alpha_1, \alpha_2\}$ to
$\{\alpha_2, \alpha_3\}$ can be seen as ``moving across a fan of
triangles'' in the Farey graph. The length of the fan is the power of
the Dehn twist.

Continuing on, moving from $\{\alpha_2, \alpha_3\}$ to
$\{\alpha_3, \alpha_4\}$ corresponds to applying a suitable Dehn twist
power about $\alpha_3$; the power is the length of the fan around
$\alpha_3$.

Hence, one obtains an upper (and, in fact, also a lower bound) for the
distance between $\{\alpha_1, \alpha_2\}$ and
$\{\alpha_{n-1},\alpha_n\}$ as the sum of the lengths of all the fans.

What makes this approach powerful is that the ``length of the fan''
between, say, $\{\alpha_{i-1}, \alpha_i\}$ and
$\{\alpha_i,\alpha_{i+1}\}$, can be computed, up to a bounded error, as a twist of $\alpha_1$
and $\alpha_n$ at $\alpha_i$ \emph{without needing to know the
  intermediate curves $\alpha_k$ in the geodesic}. This is the
celebrated ``bounded geodesic image theorem'' in this context: the
amount of twisting of $\alpha_1$ and $\alpha_{i-1}$ about $\alpha_i$
is uniformly bounded (and similar for $\alpha_n, \alpha_{i+1}$). In
this setting, a reason this works is that all $\alpha_i$ are in
different isotopy classes.

Combining everything one obtains a distance formula in the marking
graph in terms of \emph{curve graph distance} and a sum of
\emph{twists}.

\begin{figure}[h!]
  \centering
  \def\svgwidth{0.8\textwidth}
  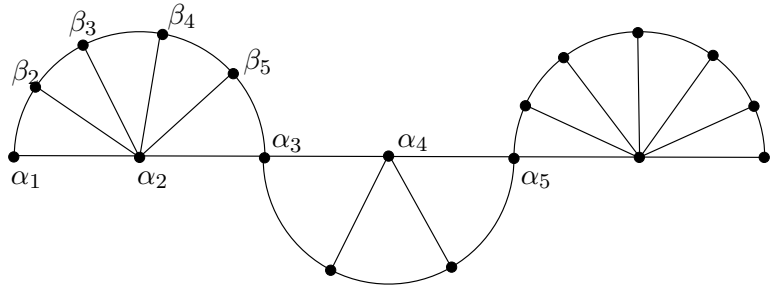
  \caption{From geodesics in the Farey graph to paths in the marking graph}
\end{figure}

We aim to prove a similar upper bound in our fine setting. There are
three obstacles which will be addressed in the next sections. First,
we need to define fine twist numbers, and prove fine bounded geodesic
image theorems. This is done in Section~\ref{sec.twist-numbers}.  A
major technical difficulty arises since geodesics in the fine curve
graph of the torus can contain many curves in the same isotopy classes
-- which means that the bounded geodesic image theorem cannot directly
be used to relate the twists of $\alpha_1$ and $\alpha_{i-1}$ about
$\alpha_i$. To overcome this problem, in Section
\ref{sec:c-dagger-alpha} we discuss paths in the fine curve graph
which stay in the same isotopy class, and how to extract paths in the
marking graphs from them.  Finally, in
Section~\ref{sec:marking-distance-bound}, we put the pieces together
to obtain the distance formula.

\section{Fine twist numbers and bounded geodesic image theorems}\label{sec.twist-numbers}
In the beginning of this section and the next one we discuss a general (orientable) surface $S$ of genus $\geq 1$. We will need only the case of the torus, but the tools developed here are likely useful in the general setting. In higher genus, edges of the curve graph corresponds to disjointness, see~\cite{Dagger1}.

We study the metric properties of the natural map that projects the fine curve graph to the arc graph of an annulus embedded in $S$. The main result, stated here in the torus, is that the projection of the axis of a loxodromic element is bounded (Lemma~\ref{lem:thick-axes}).

\subsection{Fine twist numbers}

\begin{definition}
Let us call \emph{almost embedded annulus} a map $A: [0,1]\times S^1 \to S$ which is an embedding on the interior of the annulus, and on each boundary component separately.
The almost embedded annulus $A$ is called \emph{essential} if the images of the boundary curves are essential. 
\end{definition}
A typical example on the torus is the almost embedded annulus $A_\alpha$ obtained by cutting along an essential simple closed curve $\alpha$. At some point below we will need to deform such annuli, which will lead to more complicated almost embedded annuli. We will often abusively identify $A$ and its image in $S$.

Let $A$ be an almost embedded essential annulus on $S$. The \emph{interior} of $A$ is the image of the interior $(0,1)\times S^1$ under the almost embedding map, the \emph{boundary} $\partial A$ is the image of the boundary of $[0,1]\times S^1$. Likewise, we will call an \emph{essential proper arc} of $A$ the image of an element of the fine arc graph $\mathcal{A}^\dagger([0,1]\times S^1)$. We denote $\mathcal{A}^\dagger(A)$ the graph of essential proper arcs of $A$, with an edge between two essential arcs if they are the images of disjoint essential arcs of $[0,1]\times S^1$: thus $A$ induces a graph isomorphism between $\mathcal{A}^\dagger([0,1]\times S^1)$ and $\mathcal{A}^\dagger(A)$. As usual, we see the graph $\mathcal{A}^\dagger(A)$ as a metric space, and in particular we can consider the Hausdorff distance $d_\mathrm{Hdff}$ between finite subsets of $\mathcal{A}^\dagger(A)$.
We then have a projection
\[ \pi_A: \mathcal{C}^\dagger(S) \to
  \mathbb{P}(\mathcal{A}^\dagger(A)) \] where $\pi_A(\beta)$ is
the set of all essential proper arcs of $A$ included in $\beta$. 
This is a finite set that could be empty; its diameter is at most $2$ (and $2$ may indeed happen when $\pi_A(\beta)$ contains two essential arcs sharing a boundary point that comes from an "inner tangency" of $\beta$ with the boundary of $A$).

\begin{figure}
  \centering
  \def\svgwidth{0.8\textwidth}
  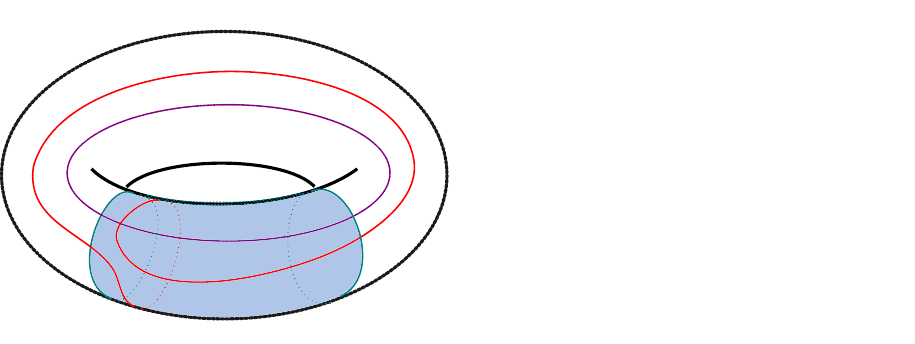
  \caption{Projection into annuli}
\end{figure}

\begin{definition}
If $\beta, \beta'$ both have nonempty projection in the almost embedded annulus $A$, we can define a
\emph{twist number} as the Hausdorff distance between the projections in the arc graph:
\[ \mathrm{tw}_A(\beta, \beta') =  d_{\mathrm{Hdff}}(\pi_A(\beta),\pi_A(\beta')). \]
\end{definition}
Since both sets $\pi_A(\beta),\pi_A(\beta')$ have diameter at most $2$, this quantity is coarsely equal to the distance between any two points of both sets, and to the diameter of the set $\pi_A(\beta) \cup \pi_A(\beta')$ in the arc graph: the difference between any two of these quantities is bounded by $4$. More precisely, we have the following inequalities in any metric space:
denoting $m$ and $M$ respectively the minimum and maximum distance between points of $A$ and $B$,
\[
m \leq d_\mathrm{Hdff}(A,B)
\leq M
\leq \mathrm{diam}(A \cup B) 
\leq m + \mathrm{diam}(A) + \mathrm{diam}(B).
\]
In particular, $\mathrm{tw}_A(\beta, \beta') \leq \mathrm{diam}(\pi_A
(\beta) \cup \pi_A(\beta'))$.

We begin by observing some basic properties of twist numbers.
\begin{lemma}\label{lem:basic-twists-curves}
  Let $A \subset S$ be an almost embedded essential annulus.
  \begin{enumerate}[i)]
  \item The projection $\pi_A(\alpha)$ is empty if and only if
    there is a curve $\gamma$ which is included in the interior of $A$ and disjoint from $\alpha$. In particular, in this case $d_{\cd}(\delta, \alpha) \leq
    2$, where $\delta$ denote one of the boundary curves of $A$,
    $\pi_{A_\alpha}(\delta)$ is also empty, and $\alpha$ is homotopic to $\delta$.
  \item If $\pi_A(\alpha_i) \neq \emptyset$ for $i=1,2$ and
    $d_{\cd}(\alpha_1, \alpha_2) \leq 1$, then
    \[ \mathrm{tw}_A(\alpha_1, \alpha_2) \leq 2. \]
  \item If $(g_0, \dots, g_\ell)$ is an edge path in $\cd(S)$
    so that $\pi_A(g_i) \neq \emptyset$ for all $i$, then
    \[ \mathrm{tw}_A(g_0, g_\ell) \leq 2\ell. \]
  \item If $\pi_A(\alpha_i) \neq \emptyset$ for $i=1,2,3$ then there is a triangle inequality
    \[ \mathrm{tw}_A(\alpha_1, \alpha_3) \leq \mathrm{tw}_A(\alpha_1, \alpha_2)+\mathrm{tw}_A(\alpha_2, \alpha_3). \]
  \item The twist number is a topological invariant, in the sense that for every $h \in \mathrm{Homeo}(S)$ and all curves $\alpha,\alpha'$ we have
  \[
  \mathrm{tw}_A(\alpha, \alpha') =  \mathrm{tw}_{h(A)}(h \alpha, h\alpha')
  \]
  if one (hence both) sides are defined.
  \end{enumerate}
\end{lemma}

\begin{proof}
  If $\pi_A(\alpha)$ is empty, then every arc of $A$ included in $\alpha$ joins the  same side of $A$ to itself, and defines a bigon with the boundary of $A$. Thus we may find a curve $\gamma$ in the interior of $A$ by pushing one side of $A$ off 
  $\alpha$ along all bigons. This proves the direct implication in the first sentence of i). The reverse implication is obvious:
  if there is a common disjoint curve, then $\alpha$ cannot have an
  essential arc with respect to $\delta$.  The remaining of i) is also obvious.
  Statements ii) and iii) follow from the observation that the projections
  of curves intersecting at most once also intersect at most
  once, and thus are distance at most two in the fine arc graph, by Lemma~\ref{lem:distance-arc-graph}. Property iv) is simply the triangle inequality for the Hausdorff distance for finite subsets of the fine arc graph of $A$. The topological invariance v) is straightforward.
\end{proof}

Concerning point iii), we emphasize that the projection $\pi_{A}$ is \emph{not} Lipschitz on its domain of definition. Indeed, given a curve $\alpha$, there are curves $\alpha', \alpha''$ which both intersect $\alpha$ once and transversely, and with arbitrarily high twist $\mathrm{tw}_{A_\alpha}(\alpha', \alpha'')$. Then $(\alpha', \alpha, \alpha'')$ is a geodesic in $\cd(S)$ and $\pi_{A_\alpha}(\alpha'), \pi_{A_\alpha}(\alpha'')$ are both non-empty and arbitrarily far away. Nevertheless, it turns out that geodesics with non trivial projections have bounded images.

\subsection{Bounded geodesic image theorems}
We have the following bounded geodesic image theorems, optimizing
either for minimizing the projection bound, or the distance to the
projection annulus. Both of these are inspired by Masur-Minsky's bounded geodesic image theorem, originally proved in \cite{MM2}.
Also compare~\cite{LongTan} for a bounded geodesic image theorem
in higher genus fine curve graphs projecting to non-annular subsurfaces. 
In the following lemmas, $\mathrm{diam}$ denotes the diameter of sets of vertices in ${\mathcal{A}^\dagger}(A)$, and $B_r(\alpha)$ denotes a ball in the fine curve graph.
\begin{lemma}[First bounded geodesic image theorem] \label{lem:bgit1}
  For any $K>0$ there is a number $r>0$ so that the following is
  true. Let $A \subset S$ be an almost embedded essential annulus, and $\alpha$ be any 
  of the boundary curves of $A$.
  Suppose that $(c_i)$ is a $K$--quasigeodesic in
  $\mathcal{C}^\dagger(S)$ which is disjoint from $B_r(\alpha)$. Then
  \[ \mathrm{diam}(\cup_i\pi_A(c_i)) \leq 2. \]
\end{lemma}
\begin{proof}[Proof of Lemma~\ref{lem:bgit1}]
  We argue by contraposition.
  Suppose that 
  \[\mathrm{diam}(\cup_i\pi_A(c_i)) \geq 3.\] 
  Since each $\pi_A(c_i)$ has diameter at most $2$, 
  there are $i \neq j$ and arcs $a \subset c_i, b \subset c_j$ so that
  \begin{enumerate}[i)]
  \item $a,b$ are essential proper arcs in $A$,
  \item in the universal cover of $A$, any lift of $b$ meets two lifts of $a$, that is, $a\cup b$ fills the annulus $A$.
  \end{enumerate}
  Hence, there are subarcs $a'\subset a, b'\subset b$ so that
  $a' \cup b'$ is 
  a bicorn defined by $c_i$ and $c_j$
  If $a' \cup b'$ is actually included in the interior of $A$ then it is distance $1$ 
  from $\alpha$. If $a' \cup b'$ meets the 
  boundary of $A$, then we can still find a curve which intersects exactly once, and   
  transversely, both $\alpha$ and $a' \cup b'$, and thus the distance between 
  the two curves is equal to $2$. 

By Corollary~\ref{cor:morse-arbitrary-bicorn}, there a universal constant $M$ such that a geodesic $[c_i c_j]$ enters the $M$-neighborhood of $a' \cup b'$, and thus also the $M+2$-neighborhood of $\alpha$.
  Since  $(c_i)$ is a $K$--quasigeodesic, this implies (by the
  Morse lemma) that there is a constant $r=r(K,M ,\delta)$ depending
  only on the quasigeodesic constants and the hyperbolicity constant
  of the fine curve graph, so that some $c_k \in B_r(\alpha)$.
\end{proof}
The second version below provides a non-explicit bound on the diameter of the projection, for any quasigeodesic \emph{path} under the only condition that every element has a nonempty projection. Note that this hypothesis is satisfied as soon as the path does not enter the ball of radius 2 around the boundary curve of the annulus (by point i) of lemma~\ref{lem:basic-twists-curves}).
\begin{lemma}[Second bounded geodesic image theorem]\label{lem:bgit2}
  For any $K>0$ there is a $B>0$ so that the following is true.  Let
  $A \subset S$ be an almost embedded essential annulus, and suppose that $(c_0, \dots, c_N)$ is
  an edge-path which is a $K$--quasigeodesic in
  $\mathcal{C}^\dagger(T^2)$ so that $\pi_A(c_i)$ is never empty. Then
  \[ \mathrm{diam}(\cup_i\pi_A(c_i)) \leq B. \] 
\end{lemma}

We thank Katie Mann for pointing out to us that this second version follows easily from the first one; our initial proof was an adaptation of Richard Webb's argument in the usual curve graph (\cite{Webb}).
\begin{proof}
Let $r$ be given by the previous lemma, and $\alpha$ denotes one of the boundary curve of $\partial A$.
The case when the path $(c_i)$ does not enter the ball $B_r(\alpha)$ follows at once from the previous lemma (with $B=2$), so let us assume there is some $i_0$ such that $c_{i_0} \in B_r(\alpha)$. 
Note that since $(c_i)$ is a K-quasigeodesics, there exists some $L = L(K,r)$ such that whenever $|i-j| \geq L$, the distance between $c_i$ and $c_j$ in the fine curve graph is more than $2r$. Thus the path decomposes as the concatenation of the three sub-paths
$\gamma_1 \cup \gamma_2 \cup \gamma_3$ with $\gamma_1, \gamma_3$ (maybe empty) entirely outside $B_r(\alpha)$, and $\gamma_2$ containing $c_{i_0}$ of length no more than $2L$.
By Lemma~\ref{lem:basic-twists-curves}, iii), the projection of $\gamma_2$ has diameter no more than $4L$, and by the first Bounded Geodesic Image Theorem the projections of $\gamma_1, \gamma_3$ have diameter at most 2. Thus we get the lemma with $B = 4L+4$.
\end{proof}

In Section~\ref{sec:c-dagger-alpha} we will prove a third bounded
geodesic image theorem for a (slightly more technical) class of paths.

\subsection{Twist bounds along axes}
Here, we restrict to the torus again. First, we define projections and twist numbers of markings in the obvious way. 

Namely, if $m=(\alpha,\beta)$ is a fine marking, we define
\[ \pi_A(m) = \pi_A(\alpha) \cup \pi_A(\beta). \] 
The key advantage is that since $\alpha$ and $\beta$ are not homotopic, at least one of them has non-empty projection to any annulus $A$, so this set is never empty. Furthermore it has diameter at most $2$, since $\alpha$ and $\beta$ intersect only once. The twist number between two markings $m,m'$ is then defined as 
\[ \mathrm{tw}_A(m, m') = d_\mathrm{Hdff}(\pi_A(m), \pi_A(m')). \]
Note that since $\pi_A(m)$ and $\pi_A(m')$ both have diameter at most
$2$, for every $\alpha \in m, \alpha' \in m'$ we have
$\mathrm{tw}_A(\alpha, \alpha') = \mathrm{tw}_A (m,m') \pm 4$.  We
then have the following lemma which includes a version of
Lemma~\ref{lem:basic-twists-curves}, iii) for markings, which removes
the requirement of nonempty projection.
\begin{lemma}\label{lem:basic-twists-markings}
  Let $A$ be an almost embedded essential annulus in $T^2$.
  \begin{enumerate}[i)]
  \item The projection
    \[ \pi_A: \mathcal{M}^\dagger \to
      (\mathbb{P}\mathcal{A}^\dagger(A), d_\mathrm{Hdff}) \] is a $2$--Lipschitz map.
  \item 
  If the two markings $m=(\alpha, \beta)$ and $m' = (\alpha, \beta')$ share the curve $\alpha$, then 
    \[ d_\mathcal{M^\dagger}(m,m') \leq \mathrm{tw}_{A_\alpha}(\beta, \beta'). \]
  \item There is a triangle inequality
    \[ \mathrm{tw}_A(m_1, m_3) \leq \mathrm{tw}_A(m_1, m_2)+\mathrm{tw}_A(m_2, m_3), \]
  \item The twist number is a topological invariant, in the sense that for every $h \in \mathrm{Homeo}(T^2)$ and every markings $m,m'$ we have
  \[
  \mathrm{tw}_A(m, m') =  \mathrm{tw}_{h(A)}(h m, h m').
  \]
  \end{enumerate}
\end{lemma}
\begin{proof}
To prove point i), it suffices to check that the twist number between any two adjacent markings $m=\{\alpha, \beta\}, m'=\{\alpha, \beta'\}$ is at most 2. First consider the case when $\pi_A(\alpha) \neq \emptyset$.
Then the sets $\pi_A(m), \pi_A(m')$ have some commun element; since their diameter is at most $2$, the Hausdorff distance between them is bounded by $2$, as wanted.
It remains to consider the case when $\pi_A(\alpha) = \emptyset$. Then $\pi_A(m) = \pi_A(\beta), \pi_A(m') = \pi_A(\beta')$, both sets are non-empty, and since $\beta$ and $\beta'$ are distance $1$ in the fine curve graph, the wanted inequality is exactly Lemma~\ref{lem:basic-twists-curves}, point ii).

Let us prove ii). Denote $A= A_\alpha$ the almost embedded annulus determined by $\alpha$. Here $\pi_A(\beta)$, $\pi_A(\beta')$ are singletons and $r = \mathrm{tw}_A(\beta, \beta')$ is by definition equal to the distance between $\pi_A(\beta)$ and $\pi_A(\beta')$ in the fine arc graph $\mathcal{A}^\dagger(A)$. Let $(b_0, \dots, b_r)$ be a geodesic in $\mathcal{A}^\dagger(A)$ from $\pi_A(\beta)$ to $\pi_A(\beta')$. One can modify $b_i$ inductively so that the almost embedding map $A$ maps both endpoints of $b_i$ to the same point, so that the image $\beta_i$ of $b_i$ under the map $A$ is a simple closed curve. Now the sequence $(\alpha, \beta_0), \dots, (\alpha, \beta_r)$ is a path in the fine marking graph from $m$ to $m'$, showing that $d_\mathcal{M^\dagger}(m,m') \leq r$.

Points iii) and iv) are immediate.
\end{proof}

As a consequence of the Bounded Geodesic Image theorem, we obtain the
following lemma which shows that hyperbolic elements have uniformly
bounded twists along their axes.

First, we observe that given any base marking $m_0$, and any
$f\in \mathrm{Homeo}_0(T^2)$ a homeomorphism acting hyperbolically on
the fine curve graph, the marking $m_0$ can be extended to an
$f$--invariant path $(m_i)$ in the marking graph. Also observe that
for any such invariant path there is a constant $K$, depending on $m_0$ and $f$, so that $(m_i)$ defines a (parametrized) bi-infinite $K$--quasigeodesic in the fine
curve graph.

\begin{lemma}[Twist bounds along axes]
\label{lem:thick-axes}
  Suppose $f\in \mathrm{Homeo}_0(T^2)$ acts
  hyperbolically on the fine curve graph, and let $(m_i)$ be an
  $f$--invariant path in the marking graph.

  Then there is a number $T$ so that for every almost embedded essential annulus $A$
  and every $i,j$ we have
  \[  \mathrm{tw}_A(m_i, m_j) \leq T. \]
\end{lemma}
\begin{proof}
  Let $A$ be any almost embedded annulus.
  Let $N$ be so that $m_{i+N} = fm_i$ for all $i$. 
  We can choose some $L= L(K)$ and indices $i_0 < i_1$ with
   the following properties:
  \begin{enumerate}
  \item for every $i$, if either $i \leq i_0$ or $i \geq i_1$, then $\pi_A(c_i) \neq 
  \emptyset$, and
  \item $i_1-i_0 \leq L$.
  \end{enumerate}
  This is due to the fact that $\pi_A(\beta) = \emptyset$ implies that
  the curve graph distance of $\beta$ to a boundary curve of $A$ is at
  most $2$.

  Now, the second Bounded Geodesic Image Theorem (Lemma~\ref{lem:bgit2}) implies that
  \[ \mathrm{diam}(\cup \{ \pi_A(m_i), i \leq i_0 \}) \leq B, \quad
    \mathrm{diam}(\cup \{ \pi_A(m_i), i \geq i_1 \}) \leq B, \] while the fact
  that twists are Lipschitz in the marking graph implies that
  \[ \mathrm{diam}(\cup \{ \pi_A(m_i), i_0 \leq i \leq i_1 \}) \leq 2L+2. \]
  Together this shows that $T = 2B+2L+2 $ satisfies the desired bound, since 
  $\mathrm{tw}_A(m_i, m_j) \leq \mathrm{diam}(\pi_A(m_i) \cup \pi_A(m_j)) \leq \mathrm{diam}(\cup \{ \pi_A(m_k), i \leq k \leq j \})$.
\end{proof}

\subsection{Width in annuli}\label{ssec:width}	
In this section we discuss the notion of \emph{width} of a curve
relative to some annulus $A$. Roughly speaking, the width generalizes
the twist number for arcs and curves that do not project in $A$. This
will be useful in the next section when constructing duals in order
make efficient moves in the fine marking graph.

In this section we also consider any orientable surface $S$ of genus $\geq 1$.
\begin{definition}
Consider an almost embedded essential annulus $A$ in the surface $S$. 
If $a, b$ are any embedded (compact) arcs in $A$, then the \emph{relative width} $\mathrm{wd}_A(a, b)$ is defined as the number of lifts of $a$ met by one lift of $b$ in the universal cover of $A$. 
If $\alpha, \beta$ are any sets in $S$, e.g. simple closed curves, then $\mathrm{wd}_A(\alpha, \beta)$ is defined as the maximum of $\mathrm{wd}_A(a, b)$ where $a, b$ are  arcs included in $A$ and respectively in $\alpha$ and $\beta$.
Note that here we do not assume any kind of transversality between  our curves and the boundary of $A$. Furthermore the definition is compatible with the definition already given in $\mathcal{A}^\dagger (A)$ (see~Lemma~\ref{lem:distance-arc-graph}).
 This is clearly a symmetric quantity, namely $\mathrm{wd}_A(\alpha, \beta) = \mathrm{wd}_A(\beta, \alpha)$ for every $\alpha, \beta$. 
\end{definition}

The notions of the previous section readily extend to arcs. Namely, given an arc $a$ and an almost embedded essential annulus $A$ in the surface $S$, we define the projection $\pi_A(a)$ to be the set of all essential proper arcs of $A$ included in $a$. Now the twist $\mathrm{tw}_A(\alpha, \beta) = d_{\mathrm{Hdff}}(\pi_A(\alpha),\pi_A(\beta))$ makes sense whenever $\alpha, \beta$ are arcs or curves with non-empty projection in $A$.
\begin{lemma}[Width is a twist]\label{lem:width-is-a-twist}
  Suppose $A$ is an almost embedded essential annulus in a surface $S$, $b_0\subset A$ an essential arc, and
  $\alpha$ a simple closed curve.
\begin{enumerate}[i)]
\item if $\pi_A(\alpha) \neq \emptyset$, then
\[ \mathrm{wd}_A(b_0, \alpha)-1 \leq \mathrm{tw}_A(b_0, \alpha) \leq \mathrm{wd}_A(b_0, \alpha)+1.\]
\item if $\pi_A(\alpha) = \emptyset$ but $\alpha$ meets the interior of $A$, then
there is an almost embedded essential annulus $A_0 \subset A$ so that
  $\pi_{A_0}(\alpha) \neq \emptyset$ and
  \[ \mathrm{wd}_{A}(b_0, \alpha) - 1 \leq \mathrm{tw}_{A_0}		
      (b_0, \alpha) \leq \mathrm{wd}_{A}(b_0, \alpha)+1. \]
\end{enumerate}  
\end{lemma}
\begin{proof} 
  \begin{figure}
    \centering
    \def\svgwidth{0.9\textwidth}
    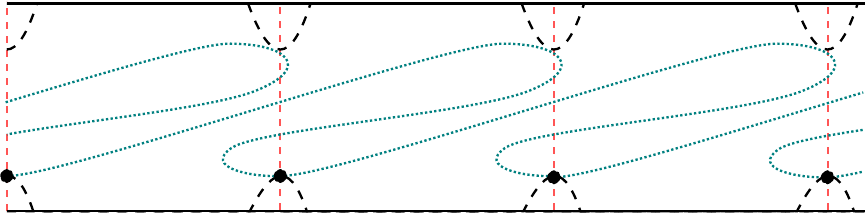
    \caption{Width is a twist}
  \end{figure}
Let us prove i). Since $b_0$ is an essential arc, 
$\pi_A(b_0)=\{b_0\}$. Given the definition of the twist and the formula relating the distance in $\mathcal{A}^\dagger(A)$ with the width (lemma~\ref{lem:distance-arc-graph}), we have 
\[
\mathrm{tw}_A(b_0, \alpha) = \max_a d_{\mathcal{A}^\dagger(A)}(b_0,a) = \max_a \mathrm{wd}_A(b_0, a) + 1
\]
 where $a$ runs over the set $\pi_A(\alpha)$ made of all essential proper arcs of $A$ included in $\alpha$. 
On the other hand, $\mathrm{wd}_A(b_0, \alpha)+1$ is defined by the same rightmost formula, except that now $a$ runs over all arcs $a \subseteq \alpha$ of $A$, not necessarily essential. Thus the upper bound in i) is clear. For the lower bound, consider an essential arc $a_0$ of $A$ included in $\alpha$, and denote $\Delta_0$ the closed fundamental domain cut by $a_0$ in the universal cover of the annulus. If
$a$ is any other arc included in $\alpha$, any lift of $a$ is included in the union of two consecutive copies of $\Delta_0$ (the worst case happens when $\alpha$ has a "tangency" at both ends of $a_0$, $a$ contains $a_0$ and continues at both ends of $a_0$ on opposite sides of $a_0$). Thus it meets at most two more lifts of $b_0$ than $a_0$ does. Thus $\mathrm{wd}_A(b_0, a) \leq \mathrm{wd}_A(b_0, a_0)+2$.
By choosing $a_0$ such that $\mathrm{tw}_A(b_0, \alpha) = \mathrm{wd}_A(b_0, a_0) + 1$, this gives the lower bound.

Let us prove ii). If every connected component of $\alpha \cap A$ meets $b_0$ at most once, then any essential embedded annulus $A_0 \subset A$ in which $\alpha$ has non-empty projection satisfies the conclusion of the lemma.
 

From now on we assume that some component of $\alpha \cap A$ meets $b_0$ at least twice. We will obtain $A_0$ by pushing $A$ along $b_0$, taking care that the width does not decrease too much.
By definition of width, there is a connected component $\bar \alpha$ of $\alpha \cap A$ such that 
$\mathrm{wd}_{A}(b_0, \alpha)= \mathrm{wd}_{A}(b_0, \bar \alpha)$, and we can choose $\bar \alpha$ that meets $b_0$ at least twice.
We build $A_0 \subseteq A$ by pushing both boundary components of $A$ along $b_0$ until they both touch $\bar \alpha$, each at a single point on $b_0$. By construction we have $\bar \alpha \subset A_0$, and $b_0' = b_0 \cap A_0$ is connected.
The universal cover $\widetilde{A_0}$ of $A_0$ naturally sits inside the universal cover $\widetilde A$ of $A$. Unrolling the definitions we see that $\mathrm{wd}_{A_0}(b_0, \alpha) = \mathrm{wd}_{A_0}(b_0, \bar \alpha) = \mathrm{wd}_{A}(b_0, \bar \alpha) = \mathrm{wd}_{A}(b_0, \alpha)$. Now the estimate follows from point i) applied to $A_0$ and $b_0'$.
\end{proof}

\section{The isotopy-constrained fine curve graph}
\label{sec:c-dagger-alpha}
This section is concerned with a subgraph of the fine curve graph,
which will play a central role in our construction of paths in the
marking graph. We fix an essential simple closed curve $\alpha_0$ in
the torus throughout. Let
\[ \cd_{[\alpha_0]}(T^2) = \cd_0(T^2) \subset \cd(T^2) \] be the full subgraph of the fine
curve graph spanned by those curves in the isotopy class of
$\alpha_0$. We denote the induced path-distance on this graph by
$d_{[\alpha]}$, or $d_0$ if the base curve $\alpha_0$ is understood.

Note that in general $d_0$ is much bigger than $d$: if $\beta$ is a
curve not in the isotopy class of $\alpha_0$, then the projection of
any curve isotopic to $\alpha_0$ to the complementary annulus $A_\beta$ of $\beta$
is nonempty, and therefore
\[ \cd_0(T^2) \to \ZZ, \quad \alpha \mapsto
  \mathrm{tw}_{A_\beta}(\alpha_0, \alpha) \] is a Lipschitz
function. Hence, there is a curve $\alpha$ at distance two of $\alpha_0$ in
$\cd(T^2)$ whose distance to $\alpha_0$ in $\cd_0(T^2)$ can be
arbitrarily large. For a similar reason $\cd_0(T^2)$ is not hyperbolic (one can find a quasi-flat by considering disjoint annuli).

\smallskip The graph $\cd_0(T^2)$ is maybe more similar to the fine
arc graph of an annulus than the full fine curve graph of the torus.
Namely, let us choose coordinates so that
$T^2 \simeq \mathbb{R}^2/\mathbb{Z}^2$ and
$\alpha_0 = \{0\} \times \mathbb{R}/\mathbb{Z}$. Consider the infinite
cyclic cover $A \simeq \mathbb{R} \times \mathbb{S}^1 \to T^2$, where
curves isotopic to $\alpha_0$ lift to simple closed curves. We denote
the deck transformation of this cover by $T$ (which, in our
coordinates is simply the shift by $1$ in the $\mathbb{R}$-direction).
\begin{lemma}[{see~\cite{LW}}, Section 4.3]\label{lem:cd0-distance}
  Let $\alpha, \alpha' \in \cd_0(T^2)$ be two distinct curves, and let
  $\widehat{\alpha}, \widehat{\alpha}'$ be lifts to $A$. Then
  \[ d_0(\alpha, \alpha') = \#\{n \in \ZZ, \widehat{\alpha} \cap
    T^n\widehat{\alpha}' \neq \emptyset\} + 1, \] that is: the distance
  in the graph $\cd_0(T^2)$ is equal to the number of lifts to $A$ of
  $\alpha$ met by one lift to $A$ of $\alpha'$, plus one.
\end{lemma}
We remark that $d_0$ is also the minimal degree of a finite cover of
the torus in which the curves have disjoint lifts.

\subsection{Statement of the main proposition}
The rest of this section is concerned with stating and proving
Proposition~\ref{prop:move-markings-all}, which is a key tool we need
in the next section in order to estimate fine marking graph
distance. The proposition will construct markings which in some sense
``follow'' an efficient path in $\cd_0$. At first reading, it might be
useful to skip the proof to see how this proposition is used later.

The statement of the central proposition relies on the following
notion.
\begin{definition}
  Let $\sigma, \tau \in \cd_{0}(T^2)$, and $n>0$. We say that $\tau$
  is \emph{almost distance $n$ from $\sigma$} if it belongs to the
  $C^0$-closure of the ball of radius $n$ centered at $\sigma$ in
  $\cd_0(T^2)$.  A \emph{quasi-path} from $\alpha$ to $\alpha'$ is a
  sequence $(\sigma_0, \dots, \sigma_n)$ in $\cd_{0}(T^2)$, 
  such that consecutive curves $\sigma_i, \sigma_{i+1}$ are almost distance $1$, for each $i = 0, \dots, n-1$. 
\end{definition}

Remember that we say that $\beta$ is a \emph{dual} to $\alpha$ if the
curves intersect transversely in one point (i.e. $\{\alpha, \beta\}$
is a fine marking).  The following proposition provides an efficient
sequence of fine markings connecting two given markings
$\{\alpha, \beta\}$ and $\{\alpha,', \beta'\}$ where $\alpha$ and
$\alpha'$ are in the same isotopy class. All distances between two
consecutive markings are bounded, up to constants, by twists in some
annuli between the duals $\beta$ and $\beta'$.

\begin{proposition}[Efficient isotopic markings]\label{prop:move-markings-all} There exists constants $B,C$ such that the following holds.
Let $\alpha, \alpha' \in \cd_0(T^2)$, and denote $n=d_0(\alpha, \alpha')$ the distance in the graph. Let $\beta, \beta'$ be duals respectively to $\alpha, \alpha'$.
There exist 
\begin{itemize}
\item a quasi-path $\sigma_{0} = \alpha, \dots, \sigma_{n} = \alpha'$ in  $\cd_0(T^2)$,
\item curves $\tau_0^-= \beta$, $\tau_0^+ , \tau_1^-, \tau_1^+, \dots, \tau_{n-1}^-, \tau_{n-1}^+, \tau_n^-$, $\tau_n^+= \beta'$,
\item and annuli $A_i \subseteq A_{\sigma_i}$
\end{itemize}
 such that: 
\begin{enumerate}
\item each $\tau_i^\pm$ is dual to $\sigma_i$, we denote $m_i^- = \{\sigma_i, \tau_i^- \}$ and $m_i^+ = \{\sigma_i, \tau_i^+ \}$,
\item $d(m_i^+, m_{i+1}^-) \leq C$ for all $i = 0, \dots, n-1$,
\item $d(m_0^-, m_0^+) \leq \mathrm{tw}_{A_0}(\beta, \beta') {+B+1}$,
\item $d(m_1^-, m_1^+) \leq \mathrm{tw}_{A_1}(\beta, \beta') + \mathrm{tw}_{A_0}(\beta, \beta') + {4 + 2B}$,
\item $d(m_i^-, m_i^+) \leq \mathrm{tw}_{A_i}(\beta, \beta') + 2B+6$, for all $i=2, \dots, n-2$,
\item $d(m_{n-1}^-, m_{n-1}^+) \leq \mathrm{tw}_{A_{n-1}}(\beta, \beta') + \mathrm{tw}_{A_n}(\beta, \beta') + {4+2B}$,
\item $d(m_n^-, m_n^+) \leq \mathrm{tw}_{A_n}(\beta, \beta')+ {B+1}$,
\end{enumerate}
where all the distances are in the fine marking graph. Furthermore, (8) for each $i = 0, \dots, n$, the curve $\sigma_i$ is almost distance $i$ from $\alpha$, and almost distance $n-i$ from $\alpha'$.
\end{proposition}

  \begin{figure}
    \centering
    \def\svgwidth{\textwidth}
    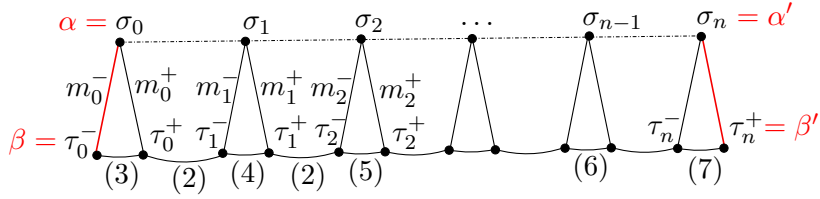
    \caption{Efficient sequence of fine markings connecting two curves in the same isotopy class. Numbers refer to upper bounds in Proposition~\ref{prop:move-markings-all}.}
  \end{figure}

\subsection{Quasi-paths}

Let us first give a more concrete, geometrical characterization of
being at almost distance at most $1$, which will show in particular that this
is a symmetric notion.  From now on we fix an orientation on our base
curve $\alpha_0$, and orient each curve isotopic to $\alpha_0$
accordingly, so that the oriented curves are isotopic. Consider the infinite cyclic cover $A = \mathbb{R} \times \mathbb{S}^1$ as above where $\alpha_0$ lifts to a closed curve, and let
$T: A \to A$ denote the covering transformation
$(x,y ) \mapsto (x+1, y)$.  We may choose the orientation so that for
each lift $\widetilde \alpha$ of $\alpha \in \cd_0(T^2)$, the lift
$T\widetilde \alpha$ is on the right-hand side of $\widetilde \alpha$.
Given two essential simple closed curves $a,b$ in $A$, we write
$a < b$ if $b$ is included in the open domain bounded by $a$ on its
right-hand side, and $a \leq b$ if $b$ is included in the closure of
this domain.

\begin{lemma}\label{lem:sandwich}
  Let $\alpha, \alpha' \in \cd_0(T^2)$. Fix a lift $\widetilde \alpha$
  of $\alpha$ in $A$.
  \begin{enumerate}[i)]
  \item $\alpha'$ and $\alpha$ are disjoint exactly if $\alpha'$ has a
    lift $\widetilde \alpha'$ such that
    $\widetilde \alpha < \widetilde \alpha' < T \widetilde \alpha$.
  \item The curve $\alpha'$ is almost distance $1$ from $\alpha$ if
    and only if $\alpha'$ has a lift $\widetilde \alpha'$ such that
    $\widetilde \alpha \leq \widetilde \alpha' \leq T \widetilde
    \alpha$.
  \end{enumerate}
\end{lemma}
\begin{proof}
  Part i) is obvious. It also immediately implies that being almost
  distance $1$ implies the criterion~ii) in the lemma. Conversely, if
  $\widetilde \alpha'$ satisfies condition~ii), there are arbitrarily
  small deformations so that
  $\widetilde \alpha < \widetilde \alpha' < T \widetilde \alpha$, and
  thus $\alpha'$ has almost distance $1$.
\end{proof}
Assuming the situation of the lemma, the curves are distance one if
and only if
$\widetilde \alpha < \widetilde \alpha' < T \widetilde \alpha$. In the
opposite case, we will say that \emph{$\alpha'$ touches $\alpha$ only
  from the right} if $\widetilde \alpha' < T \widetilde \alpha$,
\emph{$\alpha'$ touches $\alpha$ only from the left} if
$\widetilde \alpha < \widetilde \alpha'$, and \emph{$\alpha'$ touches
  $\alpha$ only from both sides} if none of the two inequalities
hold. Rephrasing the lemma, a curve is almost distance $1$ from
$\alpha$ if and only if it either touches $\alpha$ from the right, or
from the left, or from both sides, or is disjoint from it.

Note that $\alpha'$ touches $\alpha$ from one side if and only if
$\alpha$ touches $\alpha'$ from the other side; and "touching from both
sides" is a symmetric relation.  Also note that if $\alpha'$ touches
$\alpha$ from both sides then the projection $\pi_\alpha(\alpha')$ in the annulus $A_{\alpha}$ is
not empty. 

\medskip We also need a version of the bounded geodesic image theorem
for quasi-paths, which uses the following lemma.
\begin{lemma}
  Suppose $A$ is an almost embedded essential annulus in $T^2$, and $(\alpha_i),i=0,\ldots,n$ a
  quasi-path. If $\pi_A(\alpha_i) \neq \emptyset$ for all
  $i$, then
  \[ \mathrm{diam}(\cup_i\pi_A(\alpha_i)) \leq 2n. \]
\end{lemma}
\begin{proof}
  Observe that by Lemma~\ref{lem:sandwich}, consecutive
  $\alpha_i, \alpha_{i+1}$ do not intersect transversely, and thus the
  same is true for the projections
  $\pi_A(\alpha_i), \pi_A(\alpha_{i+1})$. This implies that
  \[ \mathrm{tw}_A(\alpha_i, \alpha_{i+1}) \leq 2 \] for all $n$, and
  thus the lemma by induction.
\end{proof}
Given the previous lemma, the proof of the second Bounded Geodesic
Image Theorem applies verbatim to give the following.
\begin{lemma}[Quasi-path bounded geodesic image theorem]\label{lem:bgit3}
  For any $K>0$ there is a $B>0$ so that the following is true.  Let
  $A \subset S$ be an almost embedded essential annulus, and suppose
  that $(c_0, \dots, c_N)$ is a quasi-path in $\cd_{0}(T^2)$ which is also a
  $K$--quasigeodesic in $\mathcal{C}^\dagger(T^2)$, so that
  $\pi_A(c_i)$ is never empty. Then
  \[ \mathrm{diam}(\cup_i\pi_A(c_i)) \leq B. \]
\end{lemma}

\subsection{Proof of Proposition~\ref{prop:move-markings-all}, Strategy}

We will split our main proposition above into three statements. The
first is concerned with constructing the quasi-path which will form
the $\sigma_i$, and the other two provide the nice duals.

\begin{proposition}\label{prop:quasi-path}
Let $\alpha, \alpha' \in \cd_0(T^2)$, and denote $n=d_0(\alpha, \alpha')$ the distance in the subgraph.
Then there exists a quasi-path $\sigma_{0} = \alpha, \dots, \sigma_{n} = \alpha'$ in  $\cd_0(T^2)$ with the following properties.
\begin{enumerate}
\item $\sigma_1$ touches $\sigma_0$ from the left, $\sigma_{n-1}$ touches $\sigma_n$ from the left,
\item for each $i=1, \dots, n-1$, the curves $\sigma_{i+1}$ touches $\sigma_i$ from both sides,
\item for each $i = 0, \dots, n$, the curve $\sigma_i$ is almost distance $i$ from $\alpha$, and almost distance $n-i$ from $\alpha'$.
\end{enumerate}
In particular, note that $\pi_{A_{\sigma_i}}(\sigma_j) \neq \emptyset$ if $|i-j| \geq 2$.
Furthermore, all these curves are included in $\alpha \cup \alpha'$.
\end{proposition}

\begin{proposition}\label{prop:move-markings}
Let $\sigma, \sigma'\in \cd(T^2)$ be isotopic curves that touch each other from one side. Let $\tau$ be a dual to $\sigma$.
Then there exists a curve $\tau'$ which is a common dual to $\sigma$ and $\sigma'$, and an almost embedded annulus $A \subset A_\sigma$, in which $\sigma'$ has non-empty projection, and  such that
$\mathrm{tw}_{\sigma}(\tau, \tau') \leq  \mathrm{tw}_{A}(\tau, \sigma')$.
Furthermore we can choose $\tau'$ isotopic to $\tau$.
\end{proposition}

\begin{proposition}\label{prop:move-markings2}There exists a constant $C$ such that the following holds.
Let $\sigma, \sigma' \in \cd(T^2)$ be isotopic curves that touch each other from both sides. Then:
\begin{enumerate}
\item There exist curves $\tau^+, \tau'^-$ respectively dual to $\sigma, \sigma'$ such that 
\[
\mathrm{tw}_\sigma(\tau^+, \sigma') \leq 2, \ \ \mathrm{tw}_{\sigma'}(\sigma, \tau'^-) \leq 2.
\]
\item Denoting $m= \{\sigma, \tau^+\}$, $m'= \{\sigma', \sigma'^-\}$, we have
$d_{\mathcal{M}_0^\dagger}(m, m') \leq C$.
\end{enumerate} 
\end{proposition}

The proofs are contained in the next subsections. We now deduce our main proposition.

\begin{proof}[Proof of Proposition~\ref{prop:move-markings-all}]
Let $\alpha, \alpha' \in \cd_0(T^2)$, 
$n=d(\alpha, \alpha')$, $\beta$ a dual to $\alpha'$. Proposition~\ref{prop:quasi-path} provides a quasi-path $(\sigma_0, \dots, \sigma_n)$ from $\alpha$ to $\alpha'$, with distance to the end-points satisfying the wanted property (8).

Let us construct the duals. We set $\tau_0^- = \beta$, and apply Proposition~\ref{prop:move-markings} with $\sigma=\sigma_0, \sigma'=\sigma_1$ which touch each other from one side, and dual $\tau_0^-$. We get an annulus $A_0 \subseteq A_{\sigma_0}$ and a common dual $\tau'$ for $\sigma_0$ and $\sigma_1$, and set $\tau_0^+ = \tau_1^- = \tau'$. 
Now we start estimating distances in the fine marking graph, using silently the basic properties of twists given by Lemma~\ref{lem:basic-twists-markings}.
Since $\sigma_0$ and $\sigma_1$ touch each other from one side and $\tau'$ is a common dual, we can find a curve $\sigma'$ which is disjoint from both $\sigma_0, \sigma_1$ and still dual to $\tau'$. This entails $d(m_0^+, m_1^-) \leq 2$, which is point (2) for $i=0$.
According to Proposition~\ref{prop:move-markings} we have \[ d(m_0^-, m_0^+) \leq \mathrm{tw}_{\sigma_0}(\tau_0^-, \tau_0^+) \leq \mathrm{tw}_{A_0}(\tau_0^-, \sigma_1) = \mathrm{tw}_{A_0}(\beta, \sigma_1). \]
 Now, $\sigma_2, \ldots,
  \sigma_n$ is a quasi-path which is also a
  $2$-quasigeodesic in
  $\cd(T^2)$, and by Proposition~\ref{prop:quasi-path} the projections
  $\pi_{A_0}(\sigma_i), i\geq 2$ are non empty.
  Thus the Quasi-path bounded geodesic image theorem~\ref{lem:bgit3}, together with the 
  triangular inequality for twists, implies that
  \[ \mathrm{tw}_{A_0}(\beta, \sigma_1) \leq
    \mathrm{tw}_{A_0}(\beta, \sigma_n) + B. \] Since $\sigma_n,
  \beta'$ intersect once, and both project to
  $A_0$, the definition of twists gives $\mathrm{tw}_{A_0}(\sigma_{n}, \beta') \leq 1$ and thus
  \[ d(m_0^-, m_0^+) \leq \mathrm{tw}_{A_0}(\beta, \beta') + B+1, \]
which is point (3).

Next we apply Proposition~\ref{prop:move-markings2} with $\sigma, \sigma' = \sigma_i, \sigma_{i+1}$ for each $i=1, \dots, n-2$. We get curves $\tau_i^+, \tau_{i+1}^-$ which are duals respectively to $\sigma_i, \sigma_{i+1}$.
By the proposition we have $d(m_i^+, m_{i+1}^-) \leq C$, which gives point (2) for all $i$ except $n-1$.

Let us check (4), setting $A_1 = A_{\sigma_1}$. By the proposition we have $\mathrm{tw}_{\sigma_1}(\tau_1^+, \sigma_2) \leq 2$,
and thus
\[
\begin{array}{rcl}
d(m_1^-, m_1^+) &\leq & \mathrm{tw}_{\sigma_1}(\tau_1^-, \tau_1^+) \\
&\leq & \mathrm{tw}_{\sigma_1}( \tau_1^-, \beta) + \mathrm{tw}_{\sigma_1}(\beta, \sigma_2) + 2.
\end{array}
\]
The first term $\mathrm{tw}_{\sigma_1}( \tau_1^-, \beta) = \mathrm{tw}_{\sigma_1}( \tau_0^+, \tau_0^-)$ is bounded by the "universal twist" between the isotopic curves $\tau_0^+, \tau_0^-$, namely the number of lifts of $\tau_0^+$ met by one lift of $\tau_0^-$ in the universal cover of the torus. Since both curves are dual to $\sigma_0$, this coincides with $\mathrm{tw}_{\sigma_0}( \tau_0^+, \tau_0^-)$. We have proved above that this is bounded by $\mathrm{tw}_{A_0}(\beta, \sigma_1)$. Finally we get
\[  
d(m_1^-, m_1^+) \leq  \mathrm{tw}_{A_0}(\beta, \sigma_1) + \mathrm{tw}_{\sigma_1}(\beta, \sigma_2) + 2.
\]
As before, we note that the projections of $\sigma_i$ to $A_{\sigma_0}, A_{\sigma_1}$ are nonempty for all $i > 2$ by Proposition~\ref{prop:quasi-path}, and thus the Quasi-path bounded geodesic image theorem~\ref{lem:bgit3} again implies that we can replace $\sigma_1, \sigma_2$ in the above formula by $\beta'$ for an additional cost of $B+1$ each, yielding
\[  
d(m_1^-, m_1^+) \leq  \mathrm{tw}_{A_0}(\beta, \beta') + \mathrm{tw}_{\sigma_1}(\beta, \beta') + 4 + 2B,
\]
which is (4).
  
Let $i \in \{2, \dots, n-2\}$, and let $A_i = A_{\sigma_i}$. 
To prove (5), we first note that the two markings $m_i^-, m_i^+$ have the curve $\sigma_i$ in common, thus their distance is bounded by $\mathrm{tw}_{A_i}(\tau_i^-, \tau_i^+)$. Now the triangular
inequality for twists and the bounds in point (1) of
Proposition~\ref{prop:move-markings2} gives
\[ d(m_i^-, m_i^+) \leq \mathrm{tw}_{A_i}(\sigma_{i-1}, \sigma_{i+1})
  + 4.\]
   The same argument with the
bounded geodesic image theorem allows to replace the twist term by
$\mathrm{tw}_{A_i}(\beta, \beta')$ at the cost of $2B+2$, showing (5).
We apply again Proposition~\ref{prop:move-markings} symmetrically with
$\sigma = \alpha', \sigma' = \sigma_{n-1}, \tau' = \beta'$, and get a
curve $\tau_{n-1}^+ = \tau_n^-$ which is a common dual to
$\sigma_{n-1}$ and $\sigma_n$.
Finally we get property (6) exactly as (4), and (7) exactly as (3). 
\end{proof}

\subsection{Operations on curves in the infinite annulus}
In order to construct our quasi-path, we will mainly work, as before, in the infinite annulus $A \simeq \mathbb{R} \times \mathbb{S}^1$, remembering the covering transformation $T(x,y) = (x+1, y)$. This annulus has a natural compactification by adding two points denoted $\pm \infty$. Remark that $A \cup\{+\infty\}$ is homeomorphic to the plane, and a sequence tends to $+\infty$ if and only if its first coordinate tends to $+\infty$. Given a set $K$ whose first projection is bounded from above, we denote $R(K)$ the unique connected component of $A \setminus K$ whose closure contains $+\infty$.

Denote $\cd(A)$ the set of essential simple closed curve in $A$, which corresponds to curves in the plane that surrounds $+\infty$.
Given two curves $\alpha, \beta \in \cd(A)$, we denote $\alpha \wedge \beta$ the boundary of $R(\alpha \cup \beta)$. A theorem of Kerekjarto (see \cite{kerekjarto}) asserts that $\alpha \wedge \beta$ is an element of $\cd(A)$, the remarkable fact being that this holds without any regularity nor transversality assumption. Thus we have an operation $\wedge : \cd(A) \times \cd(A) \to \cd(A)$. Note the following caracterization, which is an immediate  consequence of the Schoenflies theorem: a point $x$ belongs to $\alpha \wedge \beta$ if and only if (1) it belongs to $\alpha \cup \beta$, and (2)  it is accessible from $+\infty$ in $R(\alpha) \cap R(\beta)$, that is, there is an embedding 
 with $\rho(0)= x$, $\rho(+\infty) = +\infty$,
 and $\rho((0,1)) \subset R(\alpha) \cap R(\beta)$. We call $\rho$ a witness ray for $x$ in $\alpha \wedge \beta$. 

Let us recall the following relations on $\cd(A)$:

\begin{itemize}
\item $\alpha < \beta$ if $\beta \subset R(\alpha)$, 
\item $\alpha \leq \beta$ if $\beta \subset \mathrm{Clos}(R(\alpha))$. 
\end{itemize}

The following lemma collects some classical properties of the wedge operation. The proof is left to the reader.
\begin{lemma}[Algebra of curves]
For every $\alpha, \alpha', \beta, \gamma$ the following holds:
\begin{enumerate}
\item $\alpha \wedge \beta = \beta \wedge \alpha$.
\item The relations $<, \leq$ are transitive, and $\alpha < \beta \leq \gamma \Rightarrow \alpha < \gamma$.
\item $\alpha \leq \alpha' \Rightarrow \alpha \wedge \beta \leq \alpha' \wedge \beta$.
\item $\alpha < \gamma, \beta < \gamma \Rightarrow \alpha \wedge \beta < \gamma$.
\item $\alpha  \leq \alpha \wedge \beta$, and $\alpha < \alpha\wedge \beta \iff \alpha < \beta$, and in this case $\alpha \wedge \beta = \beta$.
\item $\alpha \cap T \alpha = \emptyset \Rightarrow \alpha < T \alpha$.
\item $T ( \alpha \wedge \beta) = (T \alpha) \wedge (T \beta)$.
\item $\alpha < \beta \iff T \alpha < T \beta$, and likewise for $\leq$.
\end{enumerate}
\end{lemma}

\begin{lemma}\label{lem:wedge2}
If $\alpha \cap T \alpha = \emptyset$ and $\beta \cap T \beta = \emptyset$ then 
$(\alpha \wedge \beta) \cap T ( \alpha \wedge \beta) = \emptyset$.
\end{lemma}
\begin{proof}
By transitivity (2) and (5) of the lemma we have $\alpha < T \alpha \leq (T \alpha) \wedge (T \beta)$, and likewise $\beta < (T \alpha) \wedge (T \beta)$, thus by (4) we get
$(\alpha \wedge \beta) \cap T ( \alpha \wedge \beta) = \emptyset$.
\end{proof}

\subsection{Construction of the quasi-path}
Coming back to the torus, let $\alpha, \alpha' \in \cd_{0}(T^2)$, and denote $n= d_0(\alpha, \alpha')$ as above. We want to prove Proposition~\ref{prop:quasi-path}, and we may assume $n \geq 2$, otherwise there is nothing to prove.
Choose lifts $\widetilde \alpha, \widetilde \alpha'$ such that the lifts met by $\widetilde \alpha'$ are $T^{-(n-1)}\alpha, \dots, T^{-1} \alpha$.
Consider
\[
\tilde \sigma_i = \tilde \alpha \wedge T^i\tilde \alpha', \ \  i = 0, \dots, n.
\]
\begin{lemma}[Properties of the quasi-path]
\label{lem:quasi-path}
The following holds:
\begin{enumerate}
\item[(i)] $\widetilde \sigma_0=\widetilde \alpha$, $\widetilde \sigma_n = T^n\widetilde \alpha'$, and each $\widetilde \sigma_i$ is an essential simple closed curve isotopic to $\widetilde \alpha$ and disjoint from its image under $T$.
\item[(ii)] $\widetilde \sigma_0 \leq \widetilde \sigma_1 < T \widetilde \sigma_0$, and $\widetilde \sigma_0$ meets $\widetilde \sigma_1$; 
\item[(iii)] $\widetilde \sigma_{i} \leq \widetilde \sigma_{i+1} \leq T\widetilde \sigma_i$, and $\widetilde \sigma_{i+1}$ meets  $\widetilde \sigma_i$ and $T\widetilde \sigma_i$, for each $1 \leq i \leq n-2$;
\item[(iv)] $\widetilde \sigma_{n-1} < \widetilde \sigma_n \leq T \widetilde \sigma_{n-1}$, and $\widetilde \sigma_n$ meets $T\widetilde \sigma_{n-1}$.
\item[(v)] $\widetilde \alpha \leq \widetilde \sigma_i < T^i \widetilde \alpha$, and $T^i \widetilde \alpha' \leq \widetilde \sigma_i < T^n \widetilde \alpha'$, for $1 \leq i \leq n-1$.
\end{enumerate}
\end{lemma}

More geometrically, one can think of the $\sigma_i$'s the following way. Let $0 < i < n$.
Call $(a,b)$ an $i$-bigon if $a \cup b$ is a simple closed curve with
$a$ a sub-arc of $\widetilde \alpha$, and $b$ a sub-arc of $T^i\widetilde \alpha'$ on the right of $\widetilde \alpha$. Call an $i$-bigon maximal if it is not included in the disk bounded by any other $i$-bigon. Note that for each $i$ there are maximal $i$-bigons meeting $T^{i-1}\widetilde \alpha$.
Then $\sigma_i$ is obtained from $\alpha$ by changing $a$ into $b$ for every maximal $i$-bigon $(a, b)$. Before proving the lemma, let us first check that it entails the wanted properties for the quasi-path.

\begin{proof}[Proof of Proposition~\ref{prop:quasi-path}]
Let $\alpha, \alpha' \in \cd_0(T^2)$, denote $n=d_0(\alpha, \alpha')$. Consider the curves $\widetilde \sigma_i$, and let $\sigma_i$ denote their projection in the torus. By (i), these are simple closed curves in the right isotopy class. Each curve touches the next one from one or two sides, by points (ii), (iii), (iv) of the lemma, as required by points (1) and (2) of the proposition.
 For point (3), note that due to property (v) of the lemma, pushing $\sigma_i$ slightly to the right yields a curve $\alpha'_i$ arbitrary close to $\sigma_i$ with a lift $\widetilde \alpha'_i$ such that 
\[
\widetilde \alpha < \widetilde \alpha'_i < T^i \widetilde \alpha \text{ and } T^i \widetilde \alpha' < \widetilde \alpha'_i < T^n \widetilde \alpha'
\]
and thus $d_0(\alpha, \alpha'_i) \leq i$ and $d_0(\alpha'_i, \alpha') \leq n-i$.
\end{proof}

\bigskip

\begin{proof}[Proof of Lemma~\ref{lem:quasi-path}]
We will refer to the algebraic properties of curves stated in the Lemma~\ref{lem:wedge2} about the wedge operation by their number.
First note that $\widetilde \alpha'$ and $T^n\widetilde \alpha'$ are disjoint from $\widetilde \alpha$, and thus
$\widetilde \sigma_0=\alpha$ and $\widetilde \sigma_n = T^n(\widetilde \alpha')$.
Each curve $\widetilde \sigma_i$ is disjoint from its image under $T$, thanks to Lemma~\ref{lem:wedge2}.

Let us prove (i). We have $\widetilde \alpha \leq \widetilde \sigma_1$, and the two curves meet, by (5).
Since $\widetilde \alpha' < \widetilde \alpha$ we have $T\widetilde \alpha' < T\widetilde \alpha$, and since $\widetilde \alpha < T \widetilde \alpha$ we get $\widetilde \sigma_1 < T \widetilde \alpha$ by (4). Point (4) of the lemma is similar.

Let us prove (ii). Since $T^i \widetilde \alpha' < T^{i+1} \widetilde \alpha'$ we get 
$\widetilde \sigma_i \leq \widetilde \sigma_{i+1}$ by (3). Likewise, from $\widetilde \alpha < T \widetilde \alpha$ we get $\widetilde \sigma_{i+1} \leq T \widetilde \sigma_{i}$. It remains to prove that the curve $\widetilde \sigma_{i+1}$ meet $\widetilde \sigma_i$ and $T \widetilde \sigma_i$.
For this we need a claim.
\begin{claim}
Let $0 < i < n$, and consider a connected component $b$ of $\widetilde{\alpha}_i \setminus \widetilde \alpha$. Then $Tb$ is included in some connected component $B'$ of 
$\widetilde{\alpha}_{i+1} \setminus \widetilde \alpha$.
\end{claim}
The claim entails immediately that $\widetilde \sigma_{i+1}$ meets $T \widetilde \sigma_i$, as asserted in point (iii) of the lemma.
\begin{proof}
Let $x \in b$ with $b$ as in the claim.
By the characterization of $\tilde \sigma_i = \tilde \alpha \wedge T^i\tilde \alpha'$, 
there is a witness ray $\rho$ for $x$. The curve $T\rho$ is clearly on the right-hand side of $\widetilde \alpha$ and $T^{i+1}\widetilde \alpha'$, thus it is a witness ray for $Tx$.
Thus $Tx $ belongs to $\widetilde \sigma_{i+1}$, and since it does not belong to $\widetilde \alpha$, it belongs to  $\widetilde{\alpha}_{i+1} \setminus \widetilde \alpha$. This proves that $Tb$ is included in $\widetilde{\alpha}_{i+1} \setminus \widetilde \alpha$, and thus in some connected component $b'$ of this set.
\end{proof}
The fact (iv) that $\widetilde \sigma_{i+1}$ meets  $\widetilde \sigma_i$ has a similar proof, here is a sketch. Consider some point $x$ in $\widetilde \sigma_{i+1} \setminus T^{i+1} \widetilde \alpha'$. A witness ray $\rho$ for $x$ in $\widetilde \sigma_{i+1} = \tilde \alpha \wedge T^{i+1}\tilde \alpha'$ is on the right-hand side of both $\widetilde \alpha$ and $T^{i+1} \widetilde \alpha'$, and thus also on the righ-hand side of $T^{i} \widetilde \alpha'$, thus it is also a witness ray for $x$ in $\widetilde \sigma_i$. This proves that $\widetilde \sigma_{i+1} \setminus T^{i+1} \widetilde \alpha' \subset \widetilde \sigma_i$, and thus both curves meet.

Let us prove (v). Let $1 \leq i \leq n-1$. First $\widetilde \alpha \leq \widetilde \alpha \wedge T^i \widetilde \alpha' = \widetilde \sigma_i$ by (5).
For the second inequality, we have $\widetilde \alpha' < \widetilde \alpha$, so
$T^i \widetilde \alpha' < T^i \widetilde \alpha$, and obviously $\widetilde \alpha < T^i \widetilde \alpha$, so $\widetilde \sigma_i < T^i \widetilde \alpha$ by (4).
Likewise, we have $T^i \widetilde \alpha' < \widetilde \sigma_i$ by (5), and from $T^{-n} \widetilde \alpha < \widetilde \alpha'$ we get $\widetilde \sigma_i < T^n \widetilde \alpha'$.
\end{proof}

\subsection{Constructing duals}

\begin{proof}[Proof of Proposition~\ref{prop:move-markings}]
  We work in the annulus $A_\sigma$ cut by $\sigma$.  Let
  $c$ be an arc inside $A$ that starts on the side which is not
  touched by $\sigma'$, is disjoint from the dual $\tau$, and ends at
  a point $x$ of $\sigma'$; we may choose $c$ such that it is disjoint
  from $\sigma'$ except at its endpoint (see Figure~\ref{fig:construct-duals}).
  \begin{figure}[h!]
    \centering
    \def\svgwidth{0.8\textwidth}
    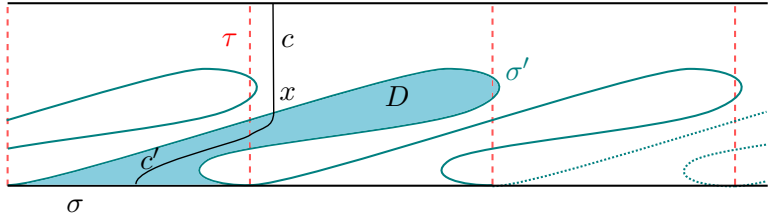
    \caption{Constructing a dual in Proposition~\ref{prop:move-markings}}
    \label{fig:construct-duals}
  \end{figure}
  Let $t \subset \sigma'$ be the connected component of
  $\sigma' \setminus \sigma$ containing $x$. Let $s \subset \sigma$ be
  the arc such that $s \cup t$ bounds an open disk $D$ in the annulus
  $A_\sigma$. Note that since $\sigma'$ touches $\sigma$ on one side, it
  is disjoint from $D$. Let $c'$ be any arc that starts at $x$ and
  ends at a point $y$ on $s$, and is included in $D$ but at its
  endpoints. Now we may modify $c$ close to $\sigma$ so that its
  starting point, seen in the torus, coincides with $y$; note that
  this can be done with $c$ still not meeting $\tau$.  Now the curve
  $c \cup c'$ defines a simple closed curve $\tau'$ in the torus,
  which meets $\sigma$ only at $y$ and $\sigma'$ only at $x$, and this is
  a common dual. (This is not our final curve $\tau'$ since it is not
  isotopic to $\tau$; we will remedy for this at the very end.)

Let us check the inequality about twists. Remember that $c$ is disjoint from $\tau$. The arc $c'$ is included in a bigon $D$ of $\sigma'$ and $\sigma$ bounded by $s \cup t$. In the universal cover of the annulus $A_\sigma$, consider lifts $\widetilde s$, $\widetilde t$ of $s$ and $t$ whose union bounds a lift $\widetilde D$ of $D$, and the lift $\widetilde c'$ of $c'$ included in $\widetilde D$. Each lift of $\tau$ that meets $\widetilde c'$ must also meet $\widetilde t$, since it is a connected set that meets both $\widetilde D$ and its complement. Thus we have
\[
\mathrm{tw}_\sigma(\tau, \tau') = \mathrm{wd}_\sigma(\tau, c') \leq \mathrm{wd}_\sigma(\tau, \sigma').
\]
Now we apply the "width is a twist" Lemma~\ref{lem:width-is-a-twist} to find an annulus $A$ included in $A_\sigma$ such that $\mathrm{wd}_{\sigma}(\tau, \sigma') = \mathrm{tw}_{A}(\tau, \sigma')$, and thus satisfies the wanted inequality on twists. 

It remains to modify $\tau'$ to get a curve in the same isotopy class as $\tau$.
Since $\tau$ is disjoint from a small enough tubular neighborhood of $\sigma$ on one side of $\sigma$, we can modify it anyhow in this neighborhood while keeping a common dual to $\sigma, \tau$. We do such a modification in order to change $\tau'$ into a curve isotopic to $\tau$. This can be done, in the universal cover of $A_\sigma$, by modifying only the $c$ part of $\tau'$, without crossing any lift of $\tau$ which is not already crossed by $c'$. Thus the twist number $\mathrm{tw}_\sigma(\tau, \tau')$ does not increase under the modification. This completes the proof.
\end{proof}

\begin{proof}[Proof of Proposition~\ref{prop:move-markings2}]
  Let $\sigma, \sigma'$ be as in
  Proposition~\ref{prop:move-markings2}. In order to evaluate the marking distance, we will use the distance $d_w$, which was defined by
  \[ d_w(\{\alpha, \beta\}, \{\alpha', \beta'\}) = \mathrm{wd}_\alpha(\alpha') + \mathrm{wd}_\beta(\beta') \]
(see Lemma~\ref{lem:relative-width-distance}).
 Note that 
  $\mathrm{wd}_\sigma(\sigma') \leq 3$ since they touch from both
  sides. 
  
  We claim that there is a curve $\gamma$ which is an ``approximate
  dual'' to both $\sigma, \sigma'$. By this we mean that there are
  duals $\tau^+, \tau^-$ to $\sigma, \sigma'$ which are
  arbitrarily $C^0$-close to $\gamma$. To find such a curve, consider
  a dual $\tau$ to $\sigma$ which intersects $\sigma'$
  with no triple points. Then replace the segment after the first
  intersection with $\sigma'$ until the last intersection with
  $\sigma'$ with a subarc of $\sigma'$. 
  \begin{figure}[h!]
    \centering
    \def\svgwidth{0.8\textwidth}
    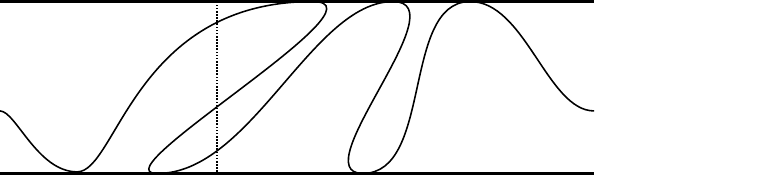
    \caption{Constructing an approximate dual in
      Proposition~\ref{prop:move-markings2}}
    \label{fig:construct-approx-duals}
  \end{figure}
  Now, let $m = \{\sigma, \tau^+\}, m'=\{\sigma', \tau^-\}$ be the two
  markings for $\tau^+, \tau^-$ close enough to $\gamma$. We then have
  that the relative width of $\tau^-, \tau^+$ is at most $1$, and thus
  $d_w(m, m') \leq 3+1 = 4$. By
  Lemma~\ref{lem:relative-width-distance} this implies that a constant
  $C$ as in (2) exists.

  For the twist bound (1) we note
  \[ \mathrm{tw}_\sigma(\tau^+, \sigma') \leq
    \mathrm{tw}_\sigma(\tau^+, \gamma) + \mathrm{tw}_\sigma(\gamma,
    \sigma') \leq 2, \]
  and similar for the other case.
\end{proof}

\section{The marking graph estimate}\label{sec:marking-distance-bound}

The aim of this section is to prove an upper bound for the distance
$d(m,m')$ in the marking graph
inspired by foundational Masur-Minsky distance formula \cite{MM2} in
the ``usual'' marking graph.

This bound is in terms of (1) the distance between $m$ and $m'$ in the
fine curve graph, and (2) the sum over the twist numbers between $m$
and $m'$ in suitable annuli. 

\smallskip Before we phrase the result, note that given two markings
$m=\{\alpha, \beta\}, m'=\{\alpha', \beta'\}$, we have
$d_{\cd}(\alpha, \beta) = d_{\cd}(\alpha', \beta')=1$ , and thus we
may define $d_{\cd}(m,m') = d_{\cd}(\alpha,\alpha')$, up to an error
of $2$.

\begin{lemma}[Fine marking distance estimate] \label{lem:good-marking-bound}
  There is a constant $K>0$ so that the following holds:
  Suppose that $m, m' \in \mathcal{M}^\dagger$ are two markings. Then
  there is a number $N \leq d_{\mathcal{C}^\dagger}(m, m')$ and curves
  $\delta_i, i=0,\ldots, N$ so that 
  \[ \frac{1}{K} d_{\mathcal{M}^\dagger}(m, m') - K \leq d_{\cd}(m,m')
    + \sum_{i=0}^N \mathrm{tw}_{\delta_i}(m, m'). \]
\end{lemma}

The proof of the marking distance bound uses the following class of
geodesics, which seem to be an analog of Masur-Minsky's \emph{tight
  geodesics} in our setting.
\begin{definition}
  A geodesic $(\alpha_i), i=0, \ldots, n$ in $\cd(T)$ is called
  \emph{steady} if the number
  of indices $i$ so that $\alpha_i \cap \alpha_{i+1} = \emptyset$ is
  maximal among all geodesics joining $\alpha_1$ to $\alpha_n$.

  For such a geodesic we define the \emph{disjointness intervals} as
  the maximal intervals $I_k = [i,j]\cap\ZZ$ so that
  $\alpha_r, \alpha_{r+1}$ are disjoint for all $r, r+1 \in I_k$. 
\end{definition}
Note that every curve on a steady geodesic lies in a disjointness
interval, which might have length $0$.

The reason for the usefulness of these geodesics is the next lemma,
which shows that these geodesics behave like tight geodesics in the
usual curve graph: they "don't loose projection". Remember that a curve $_alpha$ is almost distance $k$ from a curve $\beta$ if it may be $C^0$-approximated by curves at distance at most $k$ from $\beta$. Note that if $\alpha$ is almost distance $0$ from $\beta$ then $\alpha=\beta$.
\begin{lemma}\label{lem:tightness}
  Suppose that $(\alpha_i), i=0,\ldots,n$ is steady, $\beta_0$ is a
  dual to $\alpha_0$, $\beta_n$ is a dual to $\alpha_n$ and
  $I_k = [i,j]\cap\ZZ$ is a disjointness interval for $i>1, j<n$.

  Suppose that $A$ is an annulus one of whose boundary curves $\delta$
  has almost distance $k$ from $\alpha_i$ and almost distance
  $(j-i)-k$ from $\alpha_j$.

  Then $\pi_A(\alpha_r) \neq \emptyset$ for all $r$ with $r<i$ or $r>j$.
\end{lemma}
\begin{proof}
  We prove that $\pi_A(\alpha_r) \neq \emptyset$ for $r < i$; the
  other case is similar. Thus, suppose that
  $\pi_A(\alpha_r) = \emptyset$, and thus
  $d_0(\alpha_r, \delta) \leq 2$. Since $\delta$ has almost distance
  $(j-i)-k$ from $\alpha_j$, and upper distance bounds are an open
  condition, we thus find that
  \[ j-r = d(\alpha_r, \alpha_j) \leq d(\alpha_r, \delta') +
    d(\delta', \alpha_j) \leq 2 + (j-i)-k \leq 2+j-i, \] where
  $\delta'$ is a curve close enough to $\delta$.
  Hence, we obtain
  \[ r \geq i-2+k. \]
  Since we assume $r < i$, this implies $r = i-1$ or $r = i-2$ (and, in this latter case, $k=0$).

  Note that by the definition of disjointess
  intervals, $\alpha_{i-1}$ is in a different isotopy class than any
  curve inside the disjointess interval, and therefore 
  $\pi_A(\alpha_{i-1}) \neq \emptyset$.

  Thus, the only remaining case is $r = i-2, k=0$, and thus
  $\pi_A(\alpha_{i-2}) = \emptyset$ where a boundary of $A$ has almost
  distance $0$ from $\alpha_i$, which means that $A = A_{\alpha_i}$.  However, in this case, there is a
  curve $\gamma$ which is disjoint from $\alpha_i, \alpha_{i-2}$.
  Thus, replacing $\alpha_{i-1}$ by $\gamma$ yields a geodesic between
  $\alpha_1, \alpha_n$ with more edges corresponding to disjointness
  -- contradicting the definition of steadyness.
\end{proof}
Combining this lemma with the main result of the previous section, we
obtain the following crucial estimate.
\begin{corollary}\label{cor:move-over-disjointness-interval}
  There is a constant $K>0$ so that the following holds.
  Suppose that $\alpha_k, k=0,\ldots,N$ is a steady geodesic, and
  $\alpha_{-1}$ is dual to $\alpha_0$, and $\alpha_{N+1}$ is dual to
  $\alpha_N$.

  Suppose that $I = [i,j]$ is a disjointness interval of the steady
  geodesic. Then
  $m_I^- = (\alpha_i, \alpha_{i-1}), m_I^+=(\alpha_j, \alpha_{j+1})$
  are markings and there are annuli $A_k, k=0, \ldots, j-i$ so that 
  \[ d(m_I^-, m_I^+) \leq K(j-i) + 2 \sum_{k=0}^{j-i} \mathrm{tw}_{A_k}(\alpha_{-1}, \alpha_{N+1}) \]
\end{corollary}
\begin{proof}
  The fact that $m_I^-, m_I^+$ are markings is immediate from the
  definition of disjointness interval and the choice of
  $\alpha_{-1}, \alpha_{N+1}$. To obtain the estimate, apply
  Proposition~\ref{prop:move-markings-all} to
  $\alpha = \alpha_i, \beta = \alpha_{i-1}, \alpha' = \alpha_j, \beta'
  = \alpha_{j+1}, n = j-i$ to obtain markings $m_k^\pm$ and
  annuli $A_k$ whose boundary have almost distance $k$ from $\alpha$
  and almost distance $n-k$ from $\alpha'$ so that
  \begin{eqnarray*}
    d(m_I^-, m_I^+) &=& d(m_0^-, m^+_n) \\
                    &\leq& \sum_{k=0}^n d(m_k^-, m_k^+) + \sum_{k=0}^{n-1} d(m_k^+, m_{k+1}^-)
  \end{eqnarray*}
We now use the estimates (2) to (7) from Proposition~\ref{prop:move-markings-all}. By (2) the second sum is bounded by $nC$. The estimates (3) to (7) provide a bound for each term of the first sum by one or two twist numbers and a constant. Putting this together, we get            
  \begin{eqnarray*}
    d(m_I^-, m_I^+) &\leq& 2\mathrm{tw}_{A_0}(\beta, \beta') +
                           \mathrm{tw}_{A_1}(\beta, \beta')                          +  \sum_{k=2}^{n-2}(\mathrm{tw}_{A_k}(\beta, \beta')+ 2B+ 6)\\
                    && +
                       \mathrm{tw}_{A_{n-1}}(\beta, \beta') + 2\mathrm{tw}_{A_n}(\beta, \beta') + 6B + 10 + nC\\
    & \leq& (C+3B+6)n + 2 \sum_{i=0}^n \mathrm{tw}_{A_k}(\beta, \beta').
  \end{eqnarray*}
  Now, by the previous Lemma~\ref{lem:tightness}, we have that
  $\pi_{A_k}(\alpha_{k'}) \neq 0$ for all $k'<i, k'>j$ and $k=0, \dots, n$. Thus, by the
  bounded geodesic image theorem and the triangular inequality for
  twists we have for every $i$ that
  \[ \mathrm{tw}_{A_k}(\beta, \beta') \leq \mathrm{tw}_{A_k}(\alpha_{-1}, \alpha_{N+1}) +2B. \]
  Together, these imply the lemma.
\end{proof}

We are now ready for the proof of the main estimate.
\begin{proof}[Proof of Lemma~\ref{lem:good-marking-bound}]
  Let $m=(\alpha, \beta)$ and $m'=(\alpha', \beta')$ be given. We
  begin by choosing a steady geodesic $(\alpha_i), i=1,\ldots, N$
  between $\alpha = \alpha_1$ and $\alpha' = \alpha_n$. Denote by
  $I_1, \ldots, I_K$ the disjointness intervals, and put
  $I_k = [i_k, j_k]$. Let $m_k^\pm$ be the markings obtained from
  Corollary~\ref{cor:move-over-disjointness-interval} applied to
  $I_k$. Note that by construction for any $k$ we have
  $m_k^+ = m_{k+1}^-$, and thus
  \[ d(m, m') \leq \sum d(m_k^-, m_k^+). \] Summing the estimate from
  Corollary~\ref{cor:move-over-disjointness-interval} then gives the
  desired formula, noting that the sum of the lengths of the
  disjointness intervals is the length of the steady geodesic.
\end{proof}

\section{Metric WPD property, Mann-Rosendal distance estimates}\label{sec:WPD}
In this short section we prove the metric WPD property, by mixing all our ingredients: the twist bound along axes (Lemma~\ref{lem:thick-axes}), the marking distance estimate in terms of twists (Lemma~\ref{lem:good-marking-bound}), and the bound on $\tilde d$ from marking distance (Corollary~\ref{cor:c0-bound}).

\subsection{A strong version}
For clarity we prove that the metric WPD property is equivalent to another, seemingly stronger, property. We do not really need this lemma here since in our context the proof of the stronger version will not be more difficult. We consider a topological group $G$ acting by isometries on a Gromov-hyperbolic metric space $X$.
\begin{lemma}\label{lem:strong-WPD}
A loxodromic  element $f$ satisfies the metric WPD property if and only if
for every $x\in X$ and every $r>0$, there exists $n_0$ and $D$ such that for every $n \geq n_0$, the set
$Z_r(x, f^n x)$ is bounded by $D$ for the maximal metric in $G$.
\end{lemma}
\begin{proof}
The reverse implication is clear, so let us prove the direct one. Assume $f$ satisfies the metric WPD property. Let $x \in X$ and $r>0$, and let us choose an invariant quasigeodesic path $A$ for $f$ containing $x$. Let $M$ be the associated Morse constant, and let $M'=5r + 4\delta + 4M$ be the constant of Lemma~\ref{lem:mid-quasigeodesics}. The metric WPD property gives a number $n_0 \geq 0$ such that 
the set $Z_{M'}(x, f^{n_0} x)$ is bounded by some $D$ for the maximal metric. Let $n \geq n_0$. Now the proof will be complete if we show that $Z_r(x, f^n x) \subseteq Z_{M'}(x, f^{n_0} x)$.

To see this, let $h \in Z_r(x, f^n x)$, that is, we have $d(x, hx) \leq r$ and $d(f^{n}x, h f^n x) \leq r$. Denote $\ell = d(x, f^{n_0}x)$ and observe that $\ell = d(h x, h f^{n_0}x)$. We are in the situation of Lemma~\ref{lem:mid-quasigeodesics}, with $\alpha, \alpha', \beta, \beta', \hat \alpha, \hat \beta$ respectively set to $x, f^nx, hx, hf^nx, f^{n_0}x, h f^{n_0}x$, and $k=\ell$. We get that $d(f^{n_0}x, h f^{n_0}x) \leq M'$, that is, $h \in Z_{M'}(x, f^{n_0} x)$.
\end{proof}

\subsection{Proof of the metric WPD property}

This subsection is dedicated to the proof of the metric WPD property (Theorem~\ref{thm:metric-WPD}).
\begin{lemma}[Twist bound]
\label{lem:WPD-1}
  Suppose $f\in \mathrm{Homeo}_0(T^2)$ acts hyperbolically on the fine curve graph,and 
  let $m_0= (\alpha, \beta)$ be a marking.

Then for every $r>0$, there exists $n_0$ and  $T$ such that  for every $n \geq n_0$, 
for every $h\in Z_r(\alpha, f^n\alpha)$ and every almost embedded annulus $A$ we have
\[
\mathrm{tw}_A(m_0,  h m_0) \leq T.
\]
\end{lemma}
\begin{proof}
  \begin{figure}
    \centering
    \def\svgwidth{0.7\textwidth}
    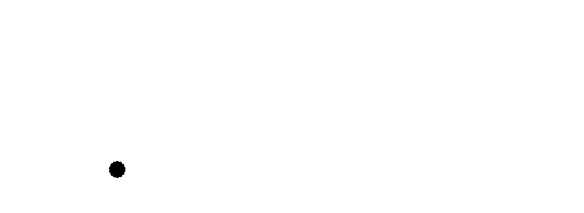
    \caption{Lemma~\ref{lem:WPD-1}}
  \end{figure}

Let $m_0$ as in the statement, and let $r>0$.
Complete $m_0$ into a path of markings which is invariant under $f$. Lemma~\ref{lem:thick-axes} provides a bound $T$ for twists numbers along this path.
Choose 
\[
n_0> \frac{2r+2}{\mathrm{tl}(f)}
\]
so that for every $n\geq n_0$, $d_\cd (\alpha, f^{n} \alpha ) > 2r + 2$.
Let $h\in Z_r(\alpha, f^n\alpha)$, and $A$ be an annulus.
Let $B$ be given by the Bounded Geodesic Image Theorem~\ref{lem:bgit2}. Let us prove that 
$\mathrm{tw}_A(m_0,  h m_0) \leq B + 2T$.
Assume $\mathrm{tw}_A(m_0,  h m_0) > B$, otherwise there is nothing to do.
Then the Bounded Geodesic Image Theorem says that a geodesic $[\alpha, h\alpha]$
enters the $2$-neighborhood of a boundary curve of $A$. By the choice of $n_0$ and triangular inequality we get that any geodesic $[f^n\alpha, hf^n\alpha]$ is disjoint from the $2$-neighborhood of the boundary curve of $A$, and so the twist number
$\mathrm{tw}_A( f^n m_0,  h f^n m_0)$ is bounded by $B$. We now apply the triangular inequality for twist numbers (Lemma~\ref{lem:basic-twists-markings}, iii) to get
\[
\mathrm{tw}_A(m_0,  h m_0) \leq
\mathrm{tw}_A(m_0,  f^n m_0) + \mathrm{tw}_A( f^n m_0,  h f^n m_0) + \mathrm{tw}_A(h f^n m_0,  h m_0).
\]
The first term is bounded by $T$, the second by $B$. By topological invariance of the twists numbers, the last one equals $\mathrm{tw}_{h^{-1}A}(f^n m_0,  m_0)$ and thus is also bounded by $T$. Thus we get the wanted bound.
\end{proof}

\begin{proof}[Proof of Theorem~\ref{thm:metric-WPD}]
Let $f$ be an element of $\mathrm{Homeo}_0(T^2)$ acting loxodromically, $\alpha \in \cd(S)$, and $r>0$. We will prove the strong version of the metric WPD expressed in Lemma~\ref{lem:strong-WPD}, that is, we look for some $n_0$ such that for every $n \geq n_0$, the set $Z_r(\alpha, f^{n} \alpha)$ is bounded.

Complete $\alpha$ to get a marking $m_0 = (\alpha, \beta)$. Apply lemma~\ref{lem:WPD-1} to get some $n_0$ and the twist bound $T$. Let $n \geq n_0$ and $h \in Z_r(\alpha, f^{n} \alpha)$. We want to bound the $\tilde d$ distance from  the identity to $h$. We know that $d(\alpha, h\alpha) \leq r$, and for every annulus $A$ we have $\mathrm{tw}_A(m_0,  h m_0) \leq T$.
Lemma~\ref{lem:good-marking-bound} provides a number {$D = D(r,T)$} such that $d_{\mathcal{M}^\dagger}(m_0, h m_0 ) \leq D$.
By properness of the inclusion $\mathcal{M}_0^\dagger \subset \mathcal{M}^\dagger$ (Corollary~\ref{cor:properness-markings}), and $C^0$ bound from markings (Corollary~\ref{cor:c0-bound}), 
 we get some $R'$, depending only on $D$, such that $\tilde d (h, \mathrm{Id}) \leq R'$.
\end{proof}

The proof of the semi-continuity theorem in the next section will make use of the following form of the metric WPD property.
\begin{corollary}\label{cor:metric-WPD}
Let $b>0$, let $f$ be an element of $\mathrm{Homeo}_0(T^2)$ acting loxodromically, and let $\alpha \in \cd(S)$. Then there exists $n_0>0$, $D>0$ such that for every $g \in \mathrm{Homeo}_0(S)$, for every $\beta \in \cd(S)$, for every $n \geq n_0$ and $m\in \mathbb{Z}$, if
\[
d(\alpha, \beta)\leq b, \quad d(f^n\alpha, g^m \beta) \leq b, \quad d(f^{2n}\alpha, g^{2m} \beta) \leq b
\]
then $\tilde d (f^{-n}, g^{-m}) \leq D$.
\end{corollary}
Note that here we do not assume that $g$ is loxodromic.
\begin{proof}   
In short, we get the result by showing that $h = g^{-m}f^{n}$  belongs to $Z_{2b}(\alpha, f^n\alpha)$, and then applying the WPD.
Here are the details. Let $b,f,\alpha$ be as in the statement. By Theorem~\ref{thm:metric-WPD}, $f$ satisfies the metric WPD property. Applying this to $\alpha$ and $b$ provides a positive integer $n_0$ and a positive number $D$. Let $g \in \mathrm{Homeo}_0(S)$, $\beta \in \cd(S)$, $n \geq n_0$, $m\in \mathbb{Z}$, and consider the map $h = g^{-m}f^{n}$.
By the triangular inequality and the fact that $f$ and $g$ act by isometries, we deduce that $h \in Z_{2b}(\alpha, f^n\alpha)$: indeed,
\[
d(\alpha, h \alpha) \leq d(\alpha, \beta) + d(\beta, h \alpha) = d(\alpha, \beta) + d(g^m \beta, f^n \alpha) \leq 2b,
\]
\[
d(f^n\alpha, h f^n \alpha) \leq d(f^n \alpha, g^m \beta) + d(g^{2m} \beta, f^{2n} \alpha) \leq 2b.
\]
Then the metric WPD property in the form of Lemma~\ref{lem:strong-WPD} tells us that $\tilde d (f^{-n}, g^{-m}) = \tilde d (h, \mathrm{Id}) \leq D$.
\end{proof}

\section{Axes and rotation sets}\label{sec:main}

In this section we prove our main results relating the axis of a loxodromic element and the rotation sets, namely (1) Each end-point of the axis determines the scaled rotation set (Theorem~\ref{thm:boundary-determines-rotation-set}), (2) Two elements that are equivalent in the sense of Fujiwara have the same scaled rotation set (Theorem~\ref{thm:equivalence}), and (3) The rotation set depends semi-continuously on each end-point (Theorem~\ref{thm:semi-continuity} below).
Each result follows easily from the next one: we have already explained how to deduce Theorem~\ref{thm:boundary-determines-rotation-set} from Theorem~\ref{thm:equivalence} at the end of Section~\ref{sec:hyperbolic}, and we explain below how to deduce Theorem~\ref{thm:equivalence} from the semi-continuity. We nevertheless include a direct proof of Theorem~\ref{thm:boundary-determines-rotation-set}, firstly because we will need the main lemma of the proof to get the metric characterization of the axis end-points (in section~\ref{subsec:metric-axis}) below), and secondly as a warm-up for the more difficult proof of the semi-continuity. For the convenience of the reader we also include a sketch of a direct proof for Theorem~\ref{thm:equivalence}. Then we will prove the full semi-continuity. We end the section by proving the metric characterization of the attracting fixed points for loxodromic elements.

\subsection{Each axis end-point determines the rotation set}\label{subsec:axis-rotation-set}
Let us recall Theorem \ref{thm:boundary-determines-rotation-set} from the introduction.
\begin{theorem*}
Consider $f,g \in \mathrm{Homeo}_0(T^2)$ acting loxodromically on the fine curve graph, and assume $\xi^+(f)= \xi^+(g)$.
Then the scaled rotation sets
\[ \frac{1}{\mathrm{tl}(f)}\mathrm{Rot}(f) \text{ and } \frac{1}{\mathrm{tl}(g)}\mathrm{Rot}(g) \]
are translates of each other.
\end{theorem*}

For $f$ as in the theorem, the scaled rotation set $\mathrm{SRot}(f)$ of $f$
is defined as an element of the set $\mathcal{K}(\mathbb{R}^2)$ of compact convex subsets of the plane up to translations.
The theorem says that the map $\mathrm{SRot}$ factorizes through $\xi^+$ and a map $\mathrm{srot} : X \to\mathcal{K}(\mathbb{R}^2)$, where $X$ is the subset of the points  in $\partial_\infty \cd (T^2)$ that are fixed points of some loxodromic element. In sub-section~\ref{subsec:semi-continuity} we will explore the continuity of the map $\mathrm{srot}$ with respect to the distance induced by the Hausdorff distance on the set $\mathcal{K}(\mathbb{R}^2)$, for the usual semi-continuity associated with the Hausdorff distance.

\begin{proof}
Let $f,g \in \mathrm{Homeo}_0(T^2)$ as in the theorem, with $\xi^+(f)= \xi^+(g)$.
We begin by proving the following boundedness property in $\mathrm{Homeo}_0(T^2)$.
\begin{lemma}\label{lem:metric-from-axis}
There exists a sequence $(m_n)$ tending to $+\infty$, such that the sequences 
\[
(d_{\cd}(f^{n}\alpha_0, g^{m_n}\alpha_0))_{n \geq 0}, \ \
(\widetilde d(f^{-n}, g^{-m_n} ) )_{n \geq 0}
\]
are bounded.
\end{lemma}

\begin{proof}
Fix some $\alpha_0 \in \cd(T^2)$. The hypothesis $\xi^+(f)= \xi^+(g)$ exactly means that the Gromov product
\[
(f^n(\alpha_0), g^m(\alpha_0))_{\alpha_0}
\]
tends to $+\infty$ when $n,m$ tends to $+\infty$.
We apply Lemma~\ref{lem:fellow-travel-tech} to the quasigeodesic axes $(f^{2n}(\alpha_0))_{n \in \mathbb{Z}}$, $(g^{2n}(\alpha_0))_{n \in \mathbb{Z}}$. This provides some sequence $(m_n)_{n \geq 0}$ of positive numbers such that the sequence $(d_{\cd}(f^{2n}(\alpha_0), g^{2m_n}(\beta_0)))_{n \geq 0}$ is bounded by some number $b$.

Corollary~\ref{cor:midpoints} provides some $M'$ so that for every $n$, $d_{\cd}(f^{n}\alpha_0, g^{m_n}\alpha_0) \leq M'$, which provides our first bound. We now apply the metric WPD property as expressed by Corollary~\ref{cor:metric-WPD}, with $b''=\max(b,M')$. We get some $D$ and some $n_0$, depending on $b, \alpha_0, f$, so that, assuming $n \geq n_0$, we get $\tilde d(f^{-n}, g^{-m_n})\leq D$.
\end{proof}

We can now complete the proof of the theorem.
We first note that since $\mathrm{Rot}(f^{-1}) = -\mathrm{Rot}(f)$ and $\mathrm{tl}(f^{-1}) = \mathrm{tl}(f)$, it suffices to prove that the scaled rotation sets of $f^{-1}$ and $g^{-1}$ coincides. Let $(m_n)$ be the sequence given by the lemma: the sequence 
$(d_{\cd}(f^{n}\alpha_0, g^{m_n}\alpha_0))_{n \geq 0}$
is bounded. From the definition of translation lengths it follows easily that
\[
\lim \frac{m_n}{n} = \frac{\mathrm{tl}(f)}{\mathrm{tl}(g)}.
\]
Denote this limit by $\lambda$. Since, by the above lemma, the sequence $(f^{-n} g^{m_n})$ is bounded in the topological group $\mathrm{Homeo}_0(T^2)$, we may apply the "boundedness and the rotation" Lemma~\ref{lem:uc-distance}, ii) to the maps $f^{-1}$ and $g^{-1}$. We conclude that the sets $\mathrm{Rot}(f^{-1})$ and $\lambda \mathrm{Rot}(g^{-1})$ are translates of each other. This completes the proof.
\end{proof}

\subsection{The Bestvina-Fujiwara class determines the rotation set}
The above proof of Theorem~\ref{thm:boundary-determines-rotation-set} may be adapted to prove Theorem~\ref{thm:equivalence}. We now sketch such a proof, for a full indirect proof see the end of the next subsection.

\begin{proof}[Sketch of proof of Theorem~\ref{thm:equivalence}]
Let $f,g \in \mathrm{Homeo}_0(T^2)$ acting loxodromically on the fine curve graph.
Assume $f \sim g$: there exists a sequence $(g_k) = (h_k g h_k^{-1})$ of maps conjugate to $g$ such that
$\lim \xi^-(g_k) = \xi^-(f), \lim \xi^+(g_k) = \xi^+(f)$. Choose some axes $A=(f^n\alpha_0)$ for $f$, and $B=(g^n\beta_0)$ for $g$. Let $n \geq 0$.
If $k$ is large enough, then the end-points of $h_k(B)$ are close to that of $A$, and thus $h_k(B)$ fellow travels with the part of $A$ from $\alpha_0$ to $f^{2n}\alpha_0$.
As before, using Corollary~\ref{cor:midpoints} and~\ref{cor:metric-WPD}, we get a bound $D$ (not dependent on $n$) on the distance $\widetilde d(f^{-n}, g_{k_n}^{-m_n})$. Thus we get the following property, which plays the role of Lemma~\ref{lem:metric-from-axis} in the previous proof:
  \begin{figure}
    \centering
    \def\svgwidth{0.7\textwidth}
    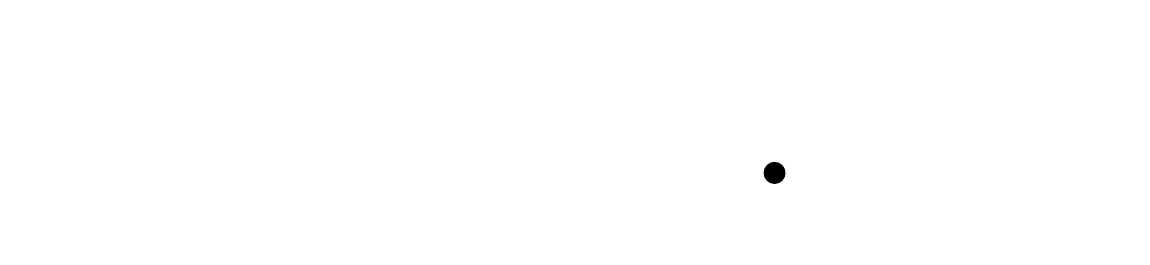
    \caption{Theorem~\ref{thm:equivalence}}
  \end{figure}

\emph{There are sequences of positive integers $(k_n)_{n \geq 0}$, $(m_n)_{n \geq 0}$ such that the sequence $\widetilde d((f^{-n}, g_{k_n}^{-m_n})_{n \geq 0}$ is bounded by some $D$. Furthermore, $\lim \frac{m_n}{n} = \frac{\mathrm{tl}(f)}{\mathrm{tl}(g)}$.}

Now let $\varepsilon>0$. With the emphasized property we can apply Lemma~\ref{lem:uc-distance}, i), and we get some $n_0$ so that for every $n \geq n_0$, the set 
\[
\frac{1}{n}\mathrm{Rot}(g_{k_n}^{-m_n}) = \frac{m_n}{n}\mathrm{Rot}(g^{-1})
\]
has a translate included in the $\varepsilon$-neighborhood of $\mathrm{Rot}(f^{-1})$.
Letting $n$ goes to $+\infty$, and then $\varepsilon$ to $0$, we get that 
$\lambda \mathrm{Rot}(g^{-1})$ has a translate included in $\mathrm{Rot}(f^{-1})$, where $\lambda$ denotes as above the quotient of the translation lengths.
We get the reverse inclusion by symmetry of the hypotheses. This completes the proof.
\end{proof}

\subsection{Semi-continuity}\label{subsec:semi-continuity}
We now turn to the more general version.
\begin{theorem}\label{thm:semi-continuity}
Let $f \in \mathrm{Homeo}_0(T^2)$ acting loxodromically on the fine curve graph. Let $(g_k)$ be a sequence in $\mathrm{Homeo}_0(T^2)$ of elements acting loxodromically such that
\begin{enumerate}
\item the sequence $(\mathrm{tl}(g_k))$ is included in $[\eta_0, \eta_1]$ for some $0  < \eta_0 < \eta_1$;
\item $(\xi^-(g_k))_{k \geq 0}$ has no subsequence converging to $\xi^+(f)$;
\item $(\xi^+(g_k))_{k \geq 0}$ converges to $\xi^+(f)$.
\end{enumerate}
Then for every $\varepsilon>0$ and every $k$ large enough, some translate of $\frac{1}{\mathrm{tl}(g_k)}\mathrm{Rot}(g_k)$ is included in the $\varepsilon$-neighborhood of $\frac{1}{\mathrm{tl}(f)}\mathrm{Rot}(f)$.
\end{theorem}

Let us make explicit the semi-continuity property. Let $\mathcal{N}$ be any neighborhood of the diagonal in $\partial_\infty \cd(T^2)^2$, and define $\mathcal{X}_\mathcal{N}$ to be the set of $\xi \in \partial_\infty \cd(T^2)$ such that there exists $f  \in \mathrm{Homeo}_0(T^2)$ acting loxodromically such that $\xi^+(f) = \xi$ and $(\xi^-(f), \xi^+(f)) \not \in \mathcal{N}$.
Then the theorem says that the map $\mathrm{srot}$, defined in section~\ref{subsec:axis-rotation-set}, is semi-continuous when restricted to the set $\mathcal{X}_\mathcal{N}$.

We are currently unable to prove a full continuity result; in other words:
\begin{question}
Under the hypotheses of Theorem~\ref{thm:semi-continuity}, do we have, up to translation,
\[
\lim \frac{1}{\mathrm{tl}(g_k)}\mathrm{Rot}(g_k) = \frac{1}{\mathrm{tl}(f)}\mathrm{Rot}(f) ?
\]
\end{question}

\bigskip

Compared to the proofs of Theorems~\ref{thm:boundary-determines-rotation-set} and~\ref{thm:equivalence}, which essentially involved only the two maps $f$ and $g$, the main difficulty here is that we need to control the quality of the quasigeodesic axes for the $g_k$'s. This will be achieved via Lemma~\ref{lem:K-for-axis}.

\begin{proof}[Proof of theorem~\ref{thm:semi-continuity}]
Let $f$ and $(g_k)$ be as in the statement of the theorem. Since all translation lengths are assumed to be larger than $\eta_0$, up to changing $f$ to $f^N$ for some large $N$, and each $g_k$ to $g_k^N$, which enlarge all involved rotation sets by the same factor $N$, we may assume that $\mathrm{tl}(f)$ and all $\mathrm{tl}(g_k)$'s are larger than the constant $L(\delta)$ from Lemma~\ref{lem:K-for-axis}. 
The lemma provides a constant $K=K(\delta)$, and invariant paths 
\[(\alpha_i)_{i\in \mathbb{Z}} \text{ for } f, \text{ and } (\beta_i^{(k)})_{i\in \mathbb{Z}} \text{ for } g_k \text{ for every } k,
\]
which are all $K$-quasigeodesics. We also observe that by Corollary~\ref{cor:fellow-travel}, every $\beta_i^{(k)}$ is distance at most $B=B(\delta, K, \eta_1)$ from some $g_k^{2m}(\beta_0^{(k)})$.

By hypotheses (2) and (3), the sequence
$((\xi^-(g_k), \xi^+(g_k))_{\alpha_0})_{k \geq 0}$
is bounded: indeed, otherwise there is a subsequence of $(g_k)$ such that the Gromov product tends to $+\infty$, and since, by (3), the sequence  $(\xi^+(g_k))$ tends to $\xi^+(f)$ we would have a subsequence of $(\xi^-(g_k))$ also converging to $\xi^+(f)$, contrary to (2).
Using Morse lemma, we get some $b_0$ such that for every $k$, the quasigeodesic $(\beta_i^{(k)})_{i\in \mathbb{Z}}$ enter the $b_0$-neighborhood of $\alpha_0$. Up to shifting the indices, we may (and will) assume that $d(\alpha_0, \beta_0^k) \leq b'$ for every $k$.

Let $b=b(\delta, K, b_0)$ be given by Lemma~\ref{lem:fellow-travel-tech}. Fix some $n \geq 0$, and let $N$ be such that $\alpha_N = f^{2n} \alpha_0$. Applying the lemma with our $N$, we get some $L$. By the convergence hypothesis (3), for $k$ large enough we have $(\xi^+(g_k), \xi^+(f))_{\alpha_0} > L$. Let us consider such a $k$. Then $(\alpha_{i_0},\beta^k_{i_0})_{\alpha_0} \geq L$ for every large enough $i_0$.
Now Lemma~\ref{lem:fellow-travel-tech} provides some point on the quasigeodesic $(\beta_i^{(k)})$ which is distance no more than $b$ from 
$f^{2n}\alpha_0$. By the above observation from Corollary~\ref{cor:fellow-travel}, up to increasing $b$ by $B$, we may assume that this point is $g_k^{2m}\beta^k_0$
for some $m = m_{k,n}$, and we get
\[
d(f^{2n}\alpha_0, g_k^{2m}\beta^k_0) \leq b.
\]
We now apply Corollary~\ref{cor:midpoints} and get some $M' = M'(\delta, K, \max(b,b'))$ so that $d(f^{n}\alpha_0, g_k^{m}\beta^k_0) \leq M'$. Now Corollary~\ref{cor:metric-WPD} may be applied with $b''=\max(b,b', M')$ and provides some $D$ and some $n_0$, depending on $b, \alpha_0, f$, so that, assuming $n \geq n_0$, we get $\tilde d(f^{-n}, g_k^{-m})\leq D$.
For clarity, let us summarize what we have proved until now.
\begin{claim}\label{claim:semi-continuity}
There is some $n_0$ and $D$  such that for every $n \geq n_0$, there exists $k_n$ such that for every $k\geq k_n$, there is some $m_{k,n}$ such that
\[
\tilde d(f^{-n}, g_k^{-m_{k,n}})\leq D.
\]
\end{claim}


Now let $\varepsilon>0$, and $n_0$ and $D$ as above. Let $n \geq n_0$ and $k \geq k_n$ as in the claim. Since $\widetilde d(f^{-n}, g_k^{-m_{n,k}}) \leq D$,
 we may apply Lemma~\ref{lem:uc-distance}, part i), with $f^{-1}$ instead of $f$ and $g = g_k^{-m_{k,n}}$. Remembering that $\mathrm{Rot}(f^{-1}) = -\mathrm{Rot}(f)$, we get that up to enlarging $n_0$,
 \[
 \frac{m_{n,k}}{n}\mathrm{Rot}(g_k) \subseteq N_\varepsilon(\mathrm{Rot}(f)).
 \]
To get the inclusion asserted by the Theorem, it remains to see that if $n$ is large enough then the fraction on the left-hand side is arbitrarily close to $\frac{\mathrm{tl}(f)}{\mathrm{tl}(g_k)}$, uniformly in $k$. Let us check this last fact. On the one hand, since all our quasigeodesic axes come  from Lemma~\ref{lem:K-for-axis}, the second part of the lemma yields
\begin{eqnarray}
d(\alpha_0, f^{n}\alpha_0 )  &=& n \mathrm{tl}(f) + \Delta \\
d(\beta^{(k)}_0, g_k^{m_{k,n}} \beta^{(k)}_0 )  &=& m_{k,n} \mathrm{tl}(g_k) + \Delta'
\end{eqnarray}
with $|\Delta|, |\Delta'|$ bounded by the Morse constant $M(K)$ from the lemma.
On the other hand, we have
\[
d(\alpha_0, \beta_0^k) \leq b',
\ \ d(f^n \alpha_0, g_k^{m_{k,n}} \beta^{(k)}_0)  \leq M',
\] 
and thus (1) and (2) differ by at most $M'+b'$. This gives the wanted convergence.
\end{proof}

\begin{proof}[Proof of Theorem~\ref{thm:equivalence} from Theorem~\ref{thm:semi-continuity}]
Assume $f \sim g$: there exists a sequence $(g_k)$ of maps conjugate to $g$ such that
$\lim \xi^-(g_k) = \xi^-(f)$, $\lim \xi^+(g_k) = \xi^+(f)$. Note that the $g_k$'s, begin conjugate to $g$, have the same translation length and the same rotation set.
We may apply Theorem~\ref{thm:semi-continuity} to $(g_k)$. Thus we get that $\mathrm{Rot}(g)$ has some translate included in
\[
\frac{\mathrm{tl}(g)}{\mathrm{tl}(f)} \mathrm{Rot}(f)
\]
Exchanging the roles of $f$ and $g$ gives the reverse inclusion. We conclude by noting that if each of both compact subsets $C,C'$ has some translate included in the other one, then some translate of $C'$ is equal to $C$.
\end{proof}

\subsection{Metric characterization of fixed points of loxodromic}\label{subsec:metric-axis}
The following statement contains Corollary~\ref{cor:fixed-point-charac} from the introduction (when applied to $f^{-1}$).
\begin{proposition}
Let $f,g$ be two elements of $\mathrm{Homeo}_0(T^2)$ acting loxodromically on the fine curve graph.
Then the following are equivalent.
\begin{enumerate}
\item $\xi^+(f) = \xi^+(g)$,
\item There exists a sequence $(m_n)$ tending to $+\infty$, such that the sequence 
\[
(f^{-n}g^{m_n} ))_{n \geq 0}
\]
is bounded in $\mathrm{Homeo}_0(T^2)$,
\item There exist sequences $(n_i), (m_i)$ tending to $+\infty$, such that the sequence 
\[
(f^{-n_i}g^{m_i} ))_{i \geq 0}
\]
is bounded in $\mathrm{Homeo}_0(T^2)$.
\end{enumerate}  
\end{proposition}
Note that the formulation of the logical equivalence between (2) and (3) does not involve the fine curve graph, but we are not aware of any argument that does not go through property (1).

\begin{proof}
Let $f,g$ be loxodromic elements. The implication (1) $\Rightarrow$ (2) is exactlty Lemma~\ref{lem:metric-from-axis}.
The implication (2) $\Rightarrow$ (3) is obvious, so it remains to prove that (3) implies (1). 
We first make the following observation. 
\begin{claim}\label{claim:d-fine-from-d-tilde}
Let $h \in \mathrm{Homeo}_0(T^2)$, and $\alpha_0$ denote a horizontal curve in $T^2 = \mathbb{R}^2/\mathbb{Z}^2$. Then
\[
d_{\cd(T^2)}(h \alpha_0, \alpha_0) \leq 2\widetilde d(h, \mathrm{Id})+2.
\]
\end{claim}
\begin{proof}[Proof of the claim]
Let $D = \widetilde d(h, \mathrm{Id})$, and let us fix a lift $\widetilde \alpha_0$ of $\alpha_0$ in the universal cover.
From the hypothesis it follows that $h \alpha_0$ has a lift which is between $\widetilde \alpha_0 - (0,D)$ and $\widetilde \alpha_0 + (0,D)$. This lifts meets at most $2D+1$ lifts of $\alpha_0$. Now the claim follows from~\cite[Lemma 4.5]{Dagger2}.
\end{proof}

Assume (3), and let $(m_i), (n_i), D$ be such that for every $i$, $\widetilde d (f^{-n_i}g^{m_i}, \mathrm{Id}) \leq D$. Let $\alpha_0$ be a horizontal simple closed curve in the torus.
Let $i\geq 0$. By the claim and the fact that $f$ acts by isometry on the fine curve graph, we get that the distance in $\cd(T^2)$ between $f^{n_i} \alpha_0$ and $g^{m_i} \alpha_0$ is bounded by $2D+2$. This implies at once that $\xi^+(f) = \xi^+(g)$.
\end{proof}
Note that the proof also shows, under hypothesis (3), that the quotient $m_i/n_i$ tends to the ratio of the translation lengths.


\section{Stabilizers of points at infinity}\label{sec:tits}
In this section we study the stabilizer of the fixed point at infinity of a loxodromic element, proving Corollaries~\ref{cor:fixed-point-dyn} and~\ref{cor:fixed-point-stab} and Theorem~\ref{thm:tits}.

Given $f \in \mathrm{Homeo}_0(T^2)$, we define the \emph{forward shadowing relation} on the torus defined by 
\[
x \sim_{f^+} y \Leftrightarrow \text{ the sequence } (\|\widetilde f^n (\widetilde x) - \widetilde f^n (\widetilde y) \|_\infty)_{n \geq 0} \text{ is bounded}
\]
where the norm is the sup norm in $\mathbb{R}^2$, $\widetilde f$ denotes any lift of $f$, and  $\widetilde x, \widetilde y$ any lifts of $x,y$ (the property does not depend on the choices of the lifts). In this situation we also says that the \emph{positive orbits of $\widetilde x$ and $\widetilde y$ shadows each other}.
The \emph{backward shadowing relation} is the forward shadowing relation for $f^{-1}$.
We first notice that the forward shadowing equivalence relation is highly non trivial.
\begin{lemma}
Let $f \in \mathrm{Homeo}_0(T^2)$, assume the rotation set of $f$ has interior.
Then there are uncountably many forward shadowing equivalence classes. \end{lemma}
\begin{proof}
Let us fix a lift $\widetilde f$ of $f$.
Assume $x$ and $y$ forward shadows each other, and $x$ has rotation vector $\rho = \mathrm{Rot}_{\widetilde f}(x)$, that is
\[
\lim_{n \to +\infty} \frac{1}{n}\left(\widetilde f^n(\widetilde x) - \widetilde x \right) = \rho.
\]
Then $y$ also has rotation vector $\rho$. Now the lemma follows immediately from 
 of a result of Llibre and MacKay (item (iv) of Theorem~1 in \cite{LM}), which says in particular that every element in the interior of the rotation set is the rotation vector of some point.
 
Note that the actual theorem of Llibre and MacKay is more precise: it says that every compact connected set in the interior of $\mathrm{Rot}(\widetilde f)$ is obtained as the rotation set of some point. This provides even more forward shadowing classes.
\end{proof}

\begin{proof}[Proof of Corollary~\ref{cor:fixed-point-dyn}]
Let $f,g$ be such that $\xi^-(f) = \xi^-(g)$, we want to prove that their forward shadowing classes coincide.
We fix lifts $\widetilde f, \widetilde g$ of $f$ and $g$.
 By Corollary~\ref{cor:fixed-point-charac}, there exists a sequence of positive numbers $(m_n)_{n \geq 0}$ such that the sequence  $(g^{n} f^{-m_n} )_{n \geq 0}$ is bounded in the group $\mathrm{Homeo}_0(T^2)$. This means that we can find for each $n \geq 0$ some lift $G_n$ of $g^n$ such that the sequence of sup norms $(\|G_n - \widetilde f^{m_n}\|_\infty)_{n \geq 0}$ is bounded by some number $D$.
 Let $\widetilde x, \widetilde y$ whose positive $\widetilde f$-orbits shadow each other: the sequence $(\|\widetilde f^m(\widetilde x)-\widetilde f^m(\widetilde y)\|)_{m \geq 0}$ is bounded by some number $D'$. Let us check that the positive $\widetilde g$-orbits of $\widetilde x$ and $\widetilde y$ shadow each other.
For each $n \geq 0$ we have $\widetilde g^n \widetilde x - \widetilde g^n \widetilde y = G_n \widetilde x - G_n \widetilde y$, and we deduce that this vector is bounded by $D+2D'$. 
\end{proof}


\begin{proof}[Proof of Corollary~\ref{cor:fixed-point-stab}]
Let $f, h, \widetilde f, \widetilde h$ be as in the statement. 
We first observe that if the numbers ${\|\widetilde f^{n} - \widetilde f^{n} \widetilde h \| }_\infty$ are bounded independently of $n \geq 0$ then it follows immediately that for any point $\widetilde x$ in the plane, the positive orbit of $\widetilde h  \widetilde x$ shadows that of $\widetilde x$: $h$ preserves each forward shadowing class of $f$. This proves the last sentence of the corollary.

We now prove the equivalence between the two properties.
 First assume the sup norms 
${\|\widetilde f^{n} - \widetilde f^{n} \widetilde h \| }_\infty$
are bounded independently of $n \geq 0$. Since $\widetilde h^{-1}$ is at bounded distance from the identity in the plane, the norms
${\|\widetilde f^{n} - \widetilde h^{-1}\widetilde f^{n} \widetilde h \| }_\infty$
are also bounded. Then the distances $\widetilde d(f^n, h^{-1}f^n h)$ are also bounded, 
and the sequence $(f^n h^{-1}f^{-n} h)_{n \geq 0}$ is bounded in the group $\mathrm{Homeo}_0(T^2)$. By Corollary~\ref{cor:fixed-point-charac} we get 
$\xi^-(f) = \xi^-(h^{-1}fh)$, and since $\xi^-(h^{-1}fh) = h^{-1}(\xi(f))$ we see that $h$ fixes $\xi^-(f)$. This proves that the second condition implies the first one.

Let us prove the reverse implication. We assume that $h(\xi^-(f)) = \xi^-(f)$, and we want to find a bound for${\|\widetilde f^{n} - \widetilde f^{n} \widetilde h \| }_\infty$.
The sup norm is left invariant, and $\widetilde h$ is a bounded distance from the identity. Thus it suffices to prove that the sequence
\[
(\|\widetilde h\widetilde f^{n}\widetilde h^{-1} - \widetilde f^{n}\|_\infty)_{n \geq 0}
\]
is bounded.
 By hypothesis we have $\xi^-(hfh^{-1}) = \xi^-(f)$. By Corollary~\ref{cor:fixed-point-charac}, there exists a sequence of positive numbers $(m_n)_{n \geq 0}$ such that the sequence  $(f^{n} hf^{-m_n}h^{-1} )_{n \geq 0}$ is bounded in the group $\mathrm{Homeo}_0(T^2)$.
In other words, we can find a sequence $(\tau_n)_{n \geq 0}$ in $\mathbb{Z}^2$, a sequence $(\varepsilon_n)_{n\geq 0}$ of continuous maps $\mathbb{R}^2 \to \mathbb{R}^2$ whose sequence of sup norms $(||\varepsilon_n||_\infty)$ is bounded by some constant $D_0$, such that for every 
$n \geq 0$ we have
\[
\widetilde h\widetilde f^{m_n}\widetilde h^{-1} = \widetilde f^{n} + \tau_n + \varepsilon_n \ \  (*).
\]

We will first show that $m_n$ is actually equal to $n$ up to a bounded quantity, and that the sequence $(\tau_n)$ is bounded.
For this, we will make use of the Bounded Deviation theorem of Le Calvez-Tal (\cite{LCT}, Theorem D), which says that every displacement vector under $\widetilde f^n$ is at bounded distance from $n$ times the rotation set of $\widetilde f$. Remember that the rotation set is convex, and a conjugacy invariant, and that we have the
classical upper bound (Lemma~\ref{lem:upper-bound-rotation-set}). Thus we can find $D_1, D_2$ such that the following inequalities hold for every $n \geq 1$:
\[
n \mathrm{Rot}(\widetilde{f})
\subseteq \mathrm{Conv}(\mathrm{Im}(\widetilde f^n - \mathrm{Id}))
\subseteq N_{D_1} (n \mathrm{Rot}(\widetilde f) )
\]
\[
m_n \mathrm{Rot}(\widetilde{f})
\subseteq \mathrm{Conv}(\mathrm{Im}(\widetilde h \widetilde f^{m_n}\widetilde h^{-1} - \mathrm{Id}))
\subseteq N_{D_2} (m_n \mathrm{Rot}(\widetilde f) )
\]
where $N_D(X)$ denotes the $D$-neighborhood of a subset $X$ of the plane (remember that the rotation sets are convex and compact, and note that the $D$-neighborhood of a convex set is convex).
Plugging in equality (*), we get
\[
\mathrm{Im}(\widetilde h \widetilde f^{m_n}\widetilde h^{-1} - \mathrm{Id})
\subseteq N_{D_0}( \mathrm{Im}(\widetilde f^n - \mathrm{Id}) ) + \tau_n
\]
\[
\mathrm{Im}(\widetilde f^n - \mathrm{Id})
\subseteq N_{D_0}( \mathrm{Im}(\widetilde h \widetilde f^{m_n}\widetilde h^{-1} - \mathrm{Id}) ) - \tau_n.
\]
Combining this with the first inequalities yields
\[
m_n \mathrm{Rot}(\widetilde{f})
\subseteq N_{D_0 + D_1} (n \mathrm{Rot}(\widetilde f) ) + \tau_n
\]
\[
n \mathrm{Rot}(\widetilde{f})
\subseteq N_{D_0 + D_2} (m_n \mathrm{Rot}(\widetilde f) ) - \tau_n.
\]
(we have used the general equalities $N_{D_0} (N_{D_1}(Y)) = N_{D_0 + D_1} (Y)$ and 
$\mathrm{Conv}(N_D(Y)) = N_D(\mathrm{Conv}(Y))$).
We first deduce than $m_n$ is a bounded distance from $n$.
Let $\delta$ denotes the diameter of the rotation set of $f$. Since adding $\tau_n$ does not change the diameter, we get
\[
m_n \delta \leq n \delta + 2(D_0 + D_1), \ \ n \delta \leq m_n \delta + 2(D_0 + D_2)
\]
and thus 
\[
|m_n - n | \leq 2\frac{D_0 + D_1 + D_2}{\delta} := D_3.
\]

Now let us prove that $(\tau_n)$  is bounded. 
Since $|m_n -n|$ is bounded, the sequence $(\|\widetilde f^{m_n} - \widetilde f^n\|_\infty)_{n \geq 0}$ is also bounded. This means that in equality (*) above we may now take $m_n=n$, up to increasing the bound on $\varepsilon_n$. The last above inclusion now reads
\[
n \mathrm{Rot}(\widetilde{f}) + \tau_n
\subseteq N_{D_0 + D_2} (n \mathrm{Rot}(\widetilde f) )
\]
which immediately implies that $\|\tau_n\|$ is bounded by $D_0 + D_2$.

Finally, up to increasing the bound on $\varepsilon_n$ again we may take $\tau_n=0$, and equality (*) now reads
\[
\widetilde h\widetilde f^{n}\widetilde h^{-1} = \widetilde f^{n} + \varepsilon_n \ \  (*)
\]
which means that the sequence
\[
(\|\widetilde h\widetilde f^{n}\widetilde h^{-1} - \widetilde f^{n}\|_\infty)_{n \geq 0}
\]
is bounded. This completes the proof.

\end{proof}

We now prove our Tits alternative in $\mathrm{Homeo}_0(T^2)$. This relies on Corollary~\ref{cor:fixed-point-stab} and on the following classical lemma.

\begin{lemma}\label{lem:abstract-tits}
Let $G$ be a group acting on a Gromov-hyperbolic metric space $X$, and $f\in G$ be a loxodromic element. Then one of the following holds:
\begin{enumerate}
\item $G$ fixes $\{\xi^-(f), \xi^+(f)\}$, 
\item $G$ fixes $\xi^-(f)$,
\item $G$ fixes $\xi^+(f)$,
\item $G$ contains an element $h$ such that the endpoints of the loxodromic element $hfh^{-1}$ are disjoint from those of $f$.
\end{enumerate}
\end{lemma}

\begin{proof}[Proof of theorem~\ref{thm:tits}]
Let $G$ be a subgroup of $\mathrm{Homeo}_0(T^2)$ containing a loxodromic element $f$. We explore the four possibilities of Lemma~\ref{lem:abstract-tits}.
If $G$ fixes $\xi^-(f)$, then Corollary~\ref{cor:fixed-point-stab} applies, and we get that $G$ leaves invariant each forward shadowing equivalence class.
Symmetrically, if $G$ fixes $\xi^+(f)$, then $G$ leaves invariant each backward shadowing equivalence class.
If $\{\xi^-(f), \xi^+(f)\}$ is invariant by $G$ then $G$ or an index two subgroup of $G$ fixes both points, and thus preserves each class of both the backward and the forward shadowing relation.
In the remain case, by the ping-pong lemma $G$ contains a free group generated by $f^n$ and $hf^n h^{-1}$, for $h$ given by the lemma and some $n>0$.
\end{proof}

\begin{proof}[Proof of Lemma~\ref{lem:abstract-tits}]
We denote $\xi^- = \xi^-(f), \xi^+ = \xi^+(f)$. We assume none of the three first possibilities hold, and we search for some element $h$ fulfilling the fourth condition.
By hypothesis, the $G$-orbit of $\xi^-$ contains some $x^- \not \in \{\xi^-, \xi^+\}$, and symmetrically the orbit of $\xi^+$ contains some $x^+ \not \in \{\xi^-, \xi^+\}$.
We want to find some element $h$ of $G$ such that $h \xi^\pm \neq \xi ^\pm$.

Let $g\in G$ be such that $g \xi^+ = x^+  \neq \xi^\pm$. If $g \xi^- \not = \xi^\pm$ then we can take $h=g$ and are done. It remains to consider case (1) when $g \xi^- = \xi ^+$ and case (2) when $g \xi^- = \xi ^-$.

First assume we are in case (1). Then  $g^2 \xi^- = x^+ \not \in \{\xi^-, \xi^+\}$.
In the case (1.1) where $g^2 \xi^+ \neq \xi^-$ then also $g^2 \xi^+ \neq \xi^+$ (since otherwise we would have $x^+ = g\xi^+ = g^{-1} \xi^+ = \xi^-$, contradiction). In this case we can take $h=g^2$.
In the case (1.2) when $g^2 \xi^+ = \xi^-$ then let $g'=fg$. Since $f$ has no fixed point except $\xi^\pm$, we get that 
\[
x'^+ :=g' \xi^+ = f x \neq \xi^\pm, \ \ 
g' \xi^- = \xi^+, \ \ 
\]
so we are in case (1) with $g$ replaced by $g'$, and now
$g'^2 \xi^+ \neq \xi^-$ (otherwise $g' \xi^+  =  g'^{-1}\xi^-$, that is, $x'^+ = x^+$ which is impossible since $f$ does not fix $x^+$). So we are in case (1.1) with $g'$ replaced by $g$, and we can take $h = g'^2$.

Let us consider case (2) when $g \xi^- = \xi ^-$.
Then we use the point $x^-$ defined above.
Let $g'\in G$ be such that $g' \xi^- = x^-  \neq \xi^\pm$.
If $g' \xi^+ \neq \xi ^\pm$ then we can take $h=g'$.
If $g' \xi^+ = \xi ^-$ then we are in a case symmetric to the above case (1), and we can take either $h = g'^2$ or $h = (f^{-1}g')^2$. It remains to consider the case when $g' \xi^+ = \xi^+$. Remember that we are in the case when $g$ fixes $\xi^-$.
Let $n>0$ and consider $h = g'f^n g$. We have
$h\xi^- = x^- \neq \xi^\pm$, 
and $h \xi^+ = g' f^n x^+$. This last point is not $\xi^+$ since $\xi^+$ is fixed by $g'$ and by $f$. Furthermore, we can assume that $f^n x^+$ is arbitrarily close to $\xi^+$ by taking $n$ arbitrarily large, and then $h \xi^+$
 is also close to $\xi^+$, in particular it is distinct from $\xi^-$. So again $h$ suits our needs. This completes the proof.
\end{proof}

\section{(Counter)examples}\label{sec:examples}
In this section we describes examples of torus homeomorphisms that illustrate our results. In section~\ref{ssec:parallelograms} we give formulas for homeomorphisms having a given parallelogram as a rotation set. For triangles we are not aware of any formula, thus instead we provide a geometric construction in the spirit of Kwapisz \cite{kwapisz_realise}, 
in section~\ref{ssec:triangle}. In section~\ref{ssec:twists} we define the twist numbers between two points at infinity, and the twist number of a loxodromic element, and prove some elementary properties. We illustrate this notion in section~\ref{ssec:computations} by providing an example of computation for a nice family of examples.
In section~\ref{ssec:schottky} we describe a free subgroup of $\mathrm{homeo}_0(T^2)$ made of loxodromic elements, whose scaled rotation sets behave as badly as possible: some boundary points in $\partial_\infty \cd(T^2)$ have no associated (scaled) rotation sets, and the scaled rotation set associated to the attracting fixed points of loxodromic element does not depend continuously on the element.

\subsection{Parallelograms}\label{ssec:parallelograms}
\begin{lemma}\label{lem:parallelograms}
Every parallelogram with rational vertices is the rotation set of a homeomorphism $f$ some power of which is conjugate to its inverse, and thus such that $\mathrm{scl}(f) = 0$.
\end{lemma}
\begin{proof}
We first explain how to realize the square $[0,1]^2$ by a homeomorphism $h$ which is conjugate to its inverse, from which we will deduce the general case. 
Let $\tilde f (x,y) = (x + \sin^2(\pi y), y )$, $\tilde g (x,y) = (x, y + \sin^2(\pi x) )$, $\tilde h = \tilde g \tilde f$. Then the rotation set of $\tilde h$ is the square $[0, 1]^2$ (see for instance \cite{LRPSW}). Furthermore, denoting $T_u$ the translation by vector $u$, we note that the map $T_{(\frac{1}{2},0)}$ commutes with $\tilde f$ and conjugate $\tilde g$ to $(x, y + 1 - \sin^2(\pi x) )$. Denoting $T = T_{(\frac{1}{2},\frac{1}{2})}$, we get
\[
T \tilde h T^{-1}
= T_{(1,1)}\tilde h^{-1}.
\]
Hence the torus homeomorphism $h$ induced by $\tilde h$ is conjugate to its inverse under the translation $T$.

Now let us deduce the general case.
Any rational parallelogram $P$ is the image of the square $[0,1]^2$ under a rational affine transformation $\Phi$. Then Lemma~\ref{lem:GAQ} provides a torus homeomorphism $h'$ whose rotation set is $P$, with a lift of the form $\tilde h' = T_v A^{-1} \tilde h^p A$, where $\tilde h$ is as before, for some rational vector $v$ and integer matrix $A$. Denoting $T' = A^{-1} T A$, we get $T' \tilde h' T'^{-1} = T_{A^{-1}(p,p)} \tilde h'^{-1}$. If we choose $q$ such that the matrix $pqA^{-1}$ has integer coefficients, then the map $T'$ conjugate $h'^q$ to its inverse. 

Let us now explain the well-known fact that a map $h$ having some power conjugate to its inverse has vanishing stable commutator length. Since $\mathrm{scl}(h^p)= p \mathrm{scl}(h)$, it suffices to show that the relation $h=\Phi h^{-1}\Phi^{-1}$ implies $\mathrm{scl}(h) =0$. From this relation we get $h^n=\Phi h^{-n} \Phi^{-1}$, and multiplying by $h^n$ displays $h^{2n}$ as the commutator of $h^{-n}$ and $\Phi$. In particular $\mathrm{cl}(h^{2n}) = 1$, and thus $\mathrm{scl}(h)=0$. \end{proof}

Motivated by this, we ask
\begin{question}
  Is every symmetric rational polygon realized as the rotation set of
  a torus homeomorphism which has a power conjugate to its inverse?
\end{question}
Up to our knowledge, there is no known example of a rotation set which
is a finite-sided polygon that is not rational.

\subsection{Triangles}\label{ssec:triangle}
We now describe a torus homeomorphism (which can be realised as a $C^\infty$-diffeomorphism) whose rotation set is the triangle $\mathcal{T}$ with vertices $(0,0)$, $(1,0)$, $(0,1)$. From this particular example we get all rational triangles by applying Lemma~\ref{lem:GAQ}.

To get our example we follow the construction by Kwapisz (\cite{kwapisz_realise}), as exposed by Béguin in \cite{beguin},
 except that we do not require $f$ to be Axiom A (we observe that Markov partitions are not needed for computing the rotation set). We include the construction since the more general construction of Kwapisz, which allows him to realize every rational polygon, is much more difficult to follow.

\begin{figure}[h]\label{fig:triangle}
  \centering
  \def\svgwidth{0.9\textwidth}
  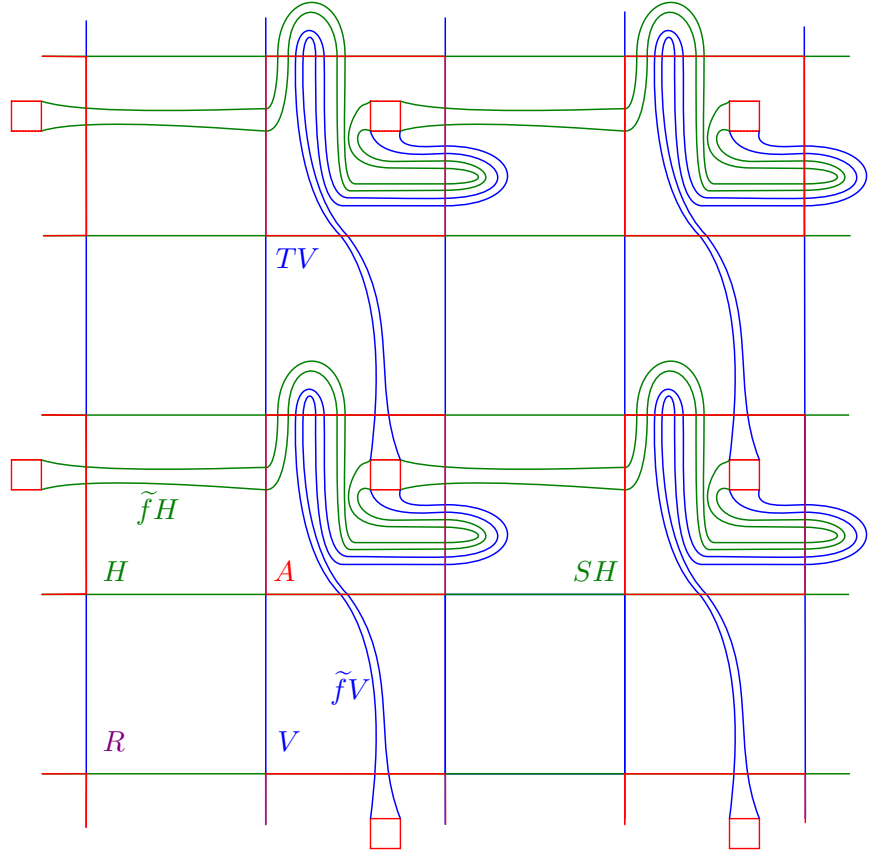
  \caption{Construction of $f$}
\end{figure}
 Let us partition the square $[0,1]^2$ into four squares $\widetilde H, \widetilde A, \widetilde R, \widetilde V$ as on Figure~\ref{fig:triangle}, and $H,A,V,R$ their projection in the torus $\mathbb{R}^2/\mathbb{Z}^2$. We want to define $\widetilde{f} \in \mathrm{Homeo}(\mathbb{R}^2, \mathbb{Z}^2)$ such that, roughly speaking, $A$ will be an attractor, $R$ a repulsor, points in $H$ will go to the right, and points of $V$ will go up. More precisely, denote $S,T$ the unit horizontal and vertical translations respectively. We demand the following condition (see Figure~\ref{fig:triangle}):
\begin{enumerate}
\item $\widetilde{f}(\widetilde A) \subset \inte \widetilde A$,
\item $\widetilde R \subset \inte \widetilde{f}(\widetilde R)$,
\item $\widetilde{f}(\widetilde H)$ meets $S^{-1} \widetilde A, \widetilde H, \widetilde A, T\widetilde V, S\widetilde H$ and is included in their unions,
\item $\widetilde{f}(\widetilde V)$ meets $T^{-1}\widetilde A, \widetilde V, \widetilde A, T\widetilde V, S\widetilde H$ and is included in their unions.
\end{enumerate}
These conditions are easily achieved, as the figure shows.
Note that such an $\widetilde{f}$ has a fixed point in $\widetilde A$ (and in $\widetilde R$), so that $\mathrm{Rot}(\widetilde{f})$ contains $(0,0)$. Denote $f$ the homeomorphism of the torus induced by $\widetilde{f}$. 
By pre-composing $f$ with a homeomorphism supported on $H$ and with a homeomorphism supported on $V$, we can further assume that 
\begin{enumerate}
\item[(3b)] $f$ has a fixed point in $H$ whose rotation vector is $(1,0)$,
\item[(4b)] $f$ has a fixed point in $V$ whose rotation vector is $(0,1)$.
\end{enumerate}
With these conditions, by convexity the rotation set of $\widetilde f$ contains the triangle $\mathcal{T}$. It remains to check that it is also included in $\mathcal{T}$. Let $x \in T^2$, and consider the positive orbit $(f^n(x))_{n \geq 0}$. If it is entirely included in $R$, then it has rotation vector $0$. If it enters $A$ then it is trapped in $A$, and again it has rotation vector $0$. Let us examine the remaining case, namely when the orbit is not entirely in $R$ and is disjoint from $A$. Since $R$ is a repulsor, up to changing $x$ to some iterate we may assume that the orbit is included in $H \cup V$. The following observation is not crucial but makes the argument easier: we have
\[TV = H + \left(\frac{1}{2}, \frac{1}{2}\right), SH = V + \left(\frac{1}{2}, \frac{1}{2}\right).\]
Now if $y$ is a point of $C = H \text{ or } V$ whose image $f(y)$ also belongs to $H \cup V$, then the image of a lift $\widetilde y$ belongs to some translate $\widetilde C$ of $\widetilde H$ or $\widetilde V$, and according to properties (3) and (4),
\[
\widetilde{f}(\widetilde{y}) \in \widetilde C + v, \text{ with } v \in E :=\left\{(0,0), (1,0), (0,1), \left(\frac{1}{2}, \frac{1}{2}\right)\right\}.
\]
Choose a lift $\widetilde x$ of $x$ in $\widetilde C = \widetilde H \text{ or } \widetilde V$.
An immediate induction yields, for every $n\geq 0$, 
\[
\widetilde{f}^n(\widetilde x) \in C + nv
\]
for some $v$ which is a convex combination of elements of $E$, and thus also of the three vertices of the triangle $\mathcal{T}$. We deduce that the $\mathrm{Rot}(\widetilde{f}) \subset \mathcal{T}$, and conclude that $\mathrm{Rot}(\widetilde{f}) = \mathcal{T}$.

\subsection{Twists numbers for boundary points, and for maps}\label{ssec:twists}
In this subsection we work in the fine curve graph $\cd(S)$ of any surface $S$, with the only assumption that this graph is unbounded and hyperbolic. We will use twist numbers to define a new conjugacy invariant for loxodromic maps, and thus in particular for elements of our group $\mathrm{Homeo}_0(T^2)$ with rotation set having non-empty interior. Elements with sufficiently distinct twist numbers are not Bestvina-Fujiwara-equivalent, and thus this also provides a tool for getting positive scl.
In the next subsection we will compute this invariant on relevant examples, and show that it is not a function of the rotation set and the translation length.

Let us denote $r \geq 2$ the constant given by the first Bounded Geodesic Image Theorem, Lemma~\ref{lem:bgit1}, for geodesics: for every curve $\alpha$, for any two points $x,y$ that may be joined by a geodesic that does not enter the $r$-ball around $\alpha$, the diameter of $\pi_\alpha(x) \cup \pi_\alpha(y)$ in the arc graph is at most $2$, where $\pi_\alpha$ denotes the projection in the annulus $A_\alpha$. Furthermore, none of the projections $\pi_\alpha(x), \pi_\alpha(y)$ is empty.
This implies that every arc in $\pi_\alpha(x)$ is distance at most two of every arc in $\pi_\alpha(y)$. Thus the number $\mathrm{tw}_\alpha(x,y)$ is well defined and no more than $2$. Note that this situation happens in particular if the Gromov product $(x \cdot y)_\alpha$ is larger than $r$ (see section~\ref{ssec:hyperbolic-boundary}).

We now fix two boundary points $\xi, \xi' \in \partial \cd(S)$. Let $(x_i), (x'_i)$ be admissible sequences that respectively converge to $\xi, \xi'$. Fix another curve $\alpha$. We can find $i_0$ such that for all $i,j \geq i_0$, $(x_i \cdot x_j)_\alpha \geq r$, and $(x'_i \cdot x'_j)_\alpha \geq r$. Thus

\begin{enumerate}
\item $\forall i,j \geq i_0$, $\mathrm{tw}_\alpha(x_i, x_j) \leq 2$ and $\mathrm{tw}_\alpha(x'_i, x'_j) \leq 2$.
\end{enumerate}
Using the triangle inequality for twist numbers (Lemma~\ref{lem:basic-twists-curves}, iv)), we get

\begin{enumerate}
\item[(2)] the sequence $(\mathrm{tw}_\alpha(x_i, x'_i) )$ is bounded by $\mathrm{tw}_\alpha(x_{i_0}, x'_{i_0}) + 4$.
\end{enumerate}
Therefore the formula
\[
\mathrm{tw}_\alpha(\xi, \xi') = \sup \{ t , \exists (x_i) \to \xi, \exists (x'_i) \to \xi', \forall i, \tw_\alpha(x_i, x'_i) = t \}
\]
defines a natural number.
Furthermore, for every admissible sequences $(x_i), (x'_i)$ converging to $\xi, \xi'$, for every $i$ large enough, $\mathrm{tw}_\alpha(x_i, x'_i)$ differs from $\mathrm{tw}_\alpha(\xi, \xi')$ by at most $4$.

\begin{definition}
The \emph{twist number between $\xi$ and $\xi'$} is
\[
\mathrm{tw}(\xi, \xi') = \sup_{\alpha \in \cd(S)} \tw_\alpha(\xi, \xi').
\]
This is either a natural number or $+\infty$. 
\end{definition}

We have the following useful triangular inequalities.

\begin{lemma}[Triangular inequalities]~
\begin{enumerate}
\item For all $\xi, \xi', \xi'', \alpha$, we have  $\mathrm{tw}_\alpha(\xi, \xi'') \leq \mathrm{tw}_\alpha(\xi, \xi') + \mathrm{tw}_\alpha(\xi', \xi'')$.
\item For all $\xi, \xi', \xi''$,
$\mathrm{tw}(\xi, \xi'') \leq \mathrm{tw}(\xi, \xi') + \mathrm{tw}(\xi', \xi'')$.
\end{enumerate}
\end{lemma}
\begin{proof}
  The first one is proved by taking $x, x''$ close to $\xi, \xi''$ where the sup is (almost) attained, and any $x'$ close to $\xi'$. The second follows.
\end{proof}

Let us briefly explore the infinite case; we want to show that the twist number $\mathrm{tw}(\xi, \xi')$ being infinite is a property of the individual points at infinity $\xi$ and/or $\xi'$ rather than of the pair $\{\xi, \xi'\}$. Fix $\xi, \xi'$. Fix some $\alpha_0$ and some radius $R$. We claim that $\{\mathrm{tw}_\alpha(\xi, \xi'), \alpha \in B_R(\alpha_0)\}$ is bounded, where $B_R(\alpha_0)$ denotes the $R$ -ball around $\alpha_0$. To prove the claim, take small neighborhoods $N,N'$ of $\xi, \xi'$ such that any geodesic segment joining two points in $N$ or two points in $N'$ is outside $B_{R+r}(\alpha_0)$. Fix some curves $x,x'$ in $N,N'$. Up to replacing them by nearby curves, we may assume they have finitely many intersection points. Let $\alpha \in B_R(\alpha_0)$. Then 
$\mathrm{tw}_\alpha(\xi, \xi') \leq \mathrm{tw}_\alpha(x,x') + 4$. On the other hand, 
$\mathrm{tw}_\alpha(x,x')$ is bounded by the cardinal of $x \cap x'$, since the number of intersection points is clearly an upper bound for the distance in the arc graph. Thus we get a bound on  $\mathrm{tw}_\alpha(\xi, \xi')$ for all $\alpha \in B_R(\alpha_0)$.

Now let $(x_i)_{i \in \mathbb{Z}}$ be a $K$-quasigeodesic path joining $\xi$ to $\xi'$ for some $K$ (for the existence of quasigeodesics connecting boundary points, see for instance~\cite{LongTan}, Proposition 2.3). For every $\alpha$ outside the $r(K)$-neighborhood of this path, $\mathrm{tw}_\alpha(\xi, \xi') \leq 2$; here $r(K)$ is the constant from the Bounded Geodesic Image Theorem.
We now assume that $\mathrm{tw}(\xi, \xi') = +\infty$: by definition there is a sequence $(\alpha_i)$ such that $\mathrm{tw}_{\alpha_i}(\xi, \xi')$ tends to $+\infty$.
According to the previous observations, such a sequence lies inside the $r(K)$ neighborhood of our path, and escape every ball. Up to extracting, such a sequence converges either to $\xi$ or $\xi'$.
There are two (overlapping) cases: (i) there are such sequences converging to $\xi$, or (ii)  there are such sequences converging to $\xi'$.
\begin{definition}
We say that $\xi$ has \emph{infinite twist} (relative to $\xi'$) if there exists a sequence $(\alpha_i)$, converging to $\xi$, such that the twist numbers $\mathrm{tw}_{\alpha_i}(\xi, \xi')$ tends to $+\infty$.
\end{definition}
So the preceding discussion proves that if $\mathrm{tw}(\xi, \xi')= +\infty$ then at least one of the two points $\xi, \xi'$ has infinite twist. Furthermore, it turns out that $\xi$ having infinite twist (relative to $\xi'$) has nothing to do with $\xi'$.
\begin{lemma}
If $\xi$ has infinite twist relative to $\xi'$, then it has infinite twist relative to all other boundary points.
\end{lemma}

\begin{proof}
Fix a quasigeodesic sequence $\gamma$ from $\xi'$ to another boundary point, distinct from $\xi$,  denoted $\xi''$. We want to check that if $\xi$ does not have infinite twist relative to $\xi''$, then it does not have infinite twist relative to $\xi'$.
Let $(\alpha_i)$ be a sequence converging to $\xi$.
For every $i$ large enough, the quasigeodesic $\gamma$ is distance more that $r(K)$ of $\alpha_i$, and thus $\mathrm{tw}_{\alpha_i}(\xi', \xi'') \leq 2$. Now by the triangular inequality, 
if the sequence $(\mathrm{tw}_{\alpha_i}(\xi, \xi''))$ is bounded, then so is $(\mathrm{tw}_{\alpha_i}(\xi, \xi'))$.
\end{proof}

\bigskip

Now we want to define the twist number of a map.
Let $f$ acts loxodromically on $\cd(S)$, and denote $\xi^-(f), \xi^+(f)$ its repulsive and attracting boundary fixed points. We can set, for each $\alpha \in \cd(S)$, $\mathrm{tw}_\alpha(f) = \mathrm{tw}_\alpha(\xi^-(f), \xi^+(f))$.
\begin{definition}
The \emph{twist of $f$} is defined by
\[\mathrm{tw}(f) = \mathrm{tw}(\xi^-(f), \xi^+(f)) = \sup_{\alpha \in \cd(S)} \tw_\alpha(\xi^-(f), \xi^+(f)).\]
\end{definition}
Note that, at least on the torus (but this is most probably more general), $\mathrm{tw}(f)$ cannot be infinite, since twist numbers are bounded along axes (Lemma~\ref{claim:d-fine-from-d-tilde}).
 In other words, boundary points with infinite twist cannot be fixed by any loxodromic element.

The Gromov boundary $\partial \cd(S)$ is equipped with a topology (see e.g. \cite{BH} for a general discussion of boundaries, and \cite{Dagger3, LongTan} specifically for the fine curve graph).
\begin{lemma}
The map $f \mapsto (\xi^-(f), \xi^+(f))$
which is defined on the subset of $\mathrm{Homeo}(S)$ of elements acting loxodromically, is continuous.
\end{lemma}
\begin{proof}  
Let $(f_m)$ be a sequence of loxodromic isometries in some Gromov-hyperbolic metric space $X$, and $f$ be another loxodromic isometry. Assume the following "simple" convergence: for every $x\in X$, the sequence $(f_m x)$ converges to $fx$. Then the sequence of boundary fixed points $(\xi^+(f_m))$ converges to $(\xi^+(f))$. This is a classical property. 

Let us go back to our specific context. Let $(f_m)$ be a sequence in $\mathrm{Homeo}(S)$ which converges to $f$ for the $C^0$ distance.
In our context, the space $X= \cd(S)$ is a graph, the distance takes only integer values, and the "simple" convergence does not hold. Instead, we have the following "coarse simple convergence" (cf \cite[Lemma 3.5]{Dagger2}): for every $\alpha$, and every $n$ large enough, we have $d(f_m \alpha, f \alpha) \leq 2$.
With this property, the convergence of $(\xi^+(f_m))$ to $(\xi^+(f))$ may be proved along the same lines as the continuity of the translation length (\cite[Theorem 3.4]{Dagger2}).

Here is a sketch of proof, details are left to the reader. According to \cite[Lemma~3.8]{Dagger2}, up to replacing each map with a fixed power, we may assume the following. There is a curve $\alpha$ such that for $m$ large enough, for every $n$, and for every $i$ with $-n \leq i \leq n$, the point $f_m^i\alpha$ is within $4\delta$ of the geodesic $[f_m^{-n} \alpha, f_m^n \alpha ]$. Fix $i_0$ so that $f^{i_0} \alpha$ is close to $\xi^+(f)$, and $m$ large enough so that the previous property holds, and $d(f_m^{i_0} \alpha, f^{i_0} \alpha) \leq 2$: then the set 
$N = \{x \in \cd(S), (x,f_m^{i_0} \alpha)_\alpha \geq d(\alpha, f_m^{i_0} \alpha) - 100\delta \}$ is a small neighborhood of $\xi^+(f)$.
For every $n \geq i_0$, the geodesic $[f_m^{-n} \alpha, f_m^n \alpha ]$ passes close to $f_m^{i_0} \alpha$, and thus the point $f_m^n \alpha$ belongs to $N$. We conclude that $\xi^+(f_m)$ belongs to the closure of $N$, which is a small neighborhood of $\xi^+(f)$ in the $\cd(S) \cup \partial \cd(S)$.
\end{proof}

\begin{lemma}\label{lem:semi-continuity-twist}
 Let $\xi, \xi'$ be two distinct boundary points. 
 
 Assume $\mathrm{tw}(\xi, \xi')$ is finite. Then if $(\eta, \eta')$ is sufficiently close to $(\xi, \xi')$ then $\mathrm{tw}(\eta, \eta') \geq \mathrm{tw}(\xi, \xi') -4$.

  Assume $\mathrm{tw}(\xi, \xi')$ is infinite. Then $\lim_{(\eta, \eta') \to \xi, \xi')}\mathrm{tw}(\eta, \eta') = + \infty$.
\end{lemma}
\begin{proof}
Assume $\mathrm{tw}(\xi, \xi')$ is finite, and let $\alpha$ so that $\mathrm{tw}(\xi, \xi') = \mathrm{tw}_\alpha(\xi, \xi')$. For $(\eta, \eta')$ sufficiently close to $(\xi, \xi')$ we have $\mathrm{tw}_\alpha(\xi, \eta) \leq 2$ and $\mathrm{tw}_\alpha(\eta', \xi') \leq 2$. By the triangular inequality we get $\mathrm{tw}_\alpha(\eta, \eta') \geq \mathrm{tw}_\alpha(\xi, \xi') -4$, and thus $\mathrm{tw}(\eta, \eta') \geq \mathrm{tw}(\xi, \xi') -4$. The infinite case is analogous.
\end{proof}

Remember that the map $f \mapsto \mathrm{tl}(f)$ is also continuous (\cite[Theorem 3.4]{Dagger2}). In particular being loxodromic is an open property.
By chaining the two lemmas we immediately get the following.
\begin{corollary}[Semi-continuity of twist numbers]\label{cor:semi-continuity-twists}
Given $f \in \mathrm{Homeo}(S)$ acting loxodromically on $\cd(S)$, if $g \in \mathrm{Homeo}(S)$ is sufficiently close to $f$, then $\mathrm{tw}(g) \geq \mathrm{tw}(f) -4$.
\end{corollary}
In the next subsection we will see that there are examples with arbitrarily large twist number. Together with the semi-continuity, this implies the following non density result. Remember that $\mathrm{Homeo}_0(S)$ acts minimally on the boundary, see~\cite{DaggerSurvey,LongTan}.
\begin{corollary}
If $\mathrm{tw}(\xi, \xi') < +\infty$ then the orbit of $(\xi, \xi')$ under the action of $\mathrm{Homeo}(S)$ on $\partial \cd(S) \times \partial \cd(S)$ is not dense.
\end{corollary}

Furthermore, the following consequence may be useful, e.g. to construct non trivial quasi-morphisms.
\begin{corollary}\label{cor:twist-fujiwara}
Assume that $|\mathrm{tw}(f) - \mathrm{tw}(g) | > 4$. Then $f \not \sim g$ and $f \not \sim g^{-1}$ in the sense of Bestvina-Fujiwara. In particular, the four points $\xi^-(f), \xi^+(f), \xi^-(g), \xi^+(g)$ are distinct.
\end{corollary}
\begin{proof}
Up to exchanging $f$ and $g$ we may assume that $\mathrm{tw}(g) < \mathrm{tw}(f) -4$, and thus by Lemma~\ref{lem:semi-continuity-twist}, the couple $(\xi^-(g)), \xi^+(g))$ may not be sent arbitrarily close to $(\xi^-(f)), \xi^+(f))$ by another element $h$. In other words, $f \not \sim g$. Since $\mathrm{tw}(g^{-1}) = \mathrm{tw}(g)$, the second inequivalence follows. The last part of the corollary follows from Lemma~\ref{lem:bf}.
\end{proof}

\subsection{Computation of twist and translation length}\label{ssec:computations}
Let us get back to the torus.
We will illustrate twist numbers by showing how computations can be done on our favorite family of examples. The idea is that first a lower bound is obtained by using explicit annuli. Then we can use a particular quasigeodesic axis and the Bounded Geodesic Image Theorem to show that all annuli in which the twist $\mathrm{tw}_A (f)$ is big have their boundary close to the quasi-axis, thus close to one of our explicit annulus, and then use this information to give an upper bound.
As a by-product, we are able to show that this quasi-axis is actually a geodesic axis, and thus to compute the precise value of the translation length on these examples.

We consider the maps $f$ and $g$ from the proof of Lemma~\ref{lem:parallelograms}. Let us fix some $p, q >0$ and denote $h = h_{p,q} = g^q f^p$.
Let $\alpha, \beta$ be respectively the horizontal and vertical curves that bound the fundamental domain $[0,1]^2$. Roughly speaking, $h$ twists with length $p$ in the complement of $\alpha$, and then twists with length $q$ in the complement of $\beta$.
Note that $f \alpha = \alpha, g \beta = \beta$, $h \alpha = g^q \alpha$ intersects $\beta$ exactly once and transversely, so that the bi-infinite path 
\[
\gamma = (\dots, h ^{-1} \alpha, h^{-1} \beta, \alpha, \beta, h \alpha, h \beta, \dots)
\]
is an invariant quasi-axis for $h$. This already tells us that $\mathrm{tl}(h) \leq 2$, as noted in~\cite{LRPSW}.
\begin{figure}[h!]
  \centering
  \def\svgwidth{0.8\textwidth}
  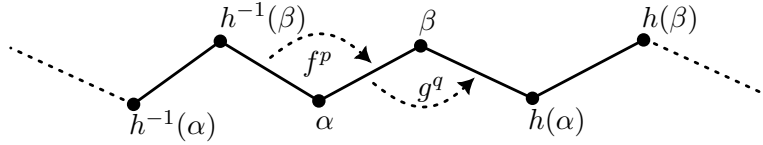
  \caption{Building an axis for $h$. The homeomorphisms $f^p, g^q$ act
    as elliptic isometries which ``rotate'' around $\alpha, \beta$.}
\end{figure}

Let $B$ denotes the constant from the second Bounded Geodesic Image Theorem for geodesics. We may assume that $B \geq 2$.
\begin{proposition}\label{prop:tl}
Assume $p,q > B$. Then the path $\gamma$ is an invariant geodesic for $h$.
\end{proposition}
As an immediate corollary we get the translation length.
\begin{corollary}
 $\mathrm{tl}(h) = 2$.
\end{corollary}

%

We denote $\tilde{\alpha}= \mathbb{R}\times \{0\}, \tilde{\beta} = \{0\} \times \mathbb{R}$ our favorite lifts of $\alpha, \beta$. At this point, as a motivation for the next definitions, we suggest that the reader draws a picture of $h \beta$ (and maybe $h^2 \beta$; compare also Figure~\ref{fig:h-iterates}).
For every $k \geq 1$, let us call an arc in the plane a \emph{good $\beta$-arc of size $k$} if it has one end-point in $(0,0)$ and the other on the top side on the rectangle $[0,1] \times [0,k]$ and is otherwise included in the interior of this rectangle. We define the \emph{good $\alpha$-arcs} by exchanging the $x$ and $y$ coordinates. Note that every good $\beta$-arc contains a good $\beta$-arc of size $1$.
Also note that $f$ fixes $\alpha$ pointwise, $g$ fixes $\beta$ pointwise, and in particular the point $(0,0)$ is fixed by both $f$ and $g$. The reader can easily check, by drawing pictures, the following properties:
\begin{enumerate}
\item $h(\alpha) = g^q \alpha$ contains a good $\beta$-arc of size $q$;
\item the image of a good $\beta$-arc under $f^p$ contains a good $\alpha$-arc of size $p$;
\item the image of a good $\alpha$-arc under $g^q$ contains a good $\beta$-arc of size $q$.
\end{enumerate}
By an induction chaining these  properties, we get that for every $n>0$, $h^n\alpha$ contains a good $\beta$-arc of size $q$. We can be a little more precise.
\begin{lemma}\label{lem:V-curve}
For every $n>0$, $h^n \alpha$ contains an arc included in the rectangle $[0,1] \times [0, q+1]$, connecting the fixed point $(0,0)$ to the right-hand side of the rectangle, and containing a $\beta$-arc of size $q$.
\end{lemma}
\begin{proof}
  \begin{figure}[h!]
    \centering
    \def\svgwidth{0.8\textwidth}
    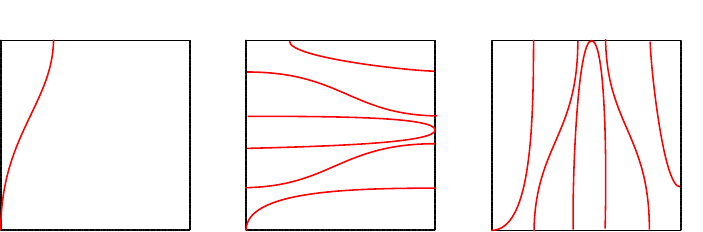
    \caption{The proof of Lemma~\ref{lem:V-curve}.
      $b_0 \subset h^n\alpha$ is the $\beta$-arc of size $1$. The
      middle picture shows the image under $f^p$, with the initial
      segment $a_0$ which is an $\alpha$-arc of size $1$. The right
      picture shows the image of that subarc under $g^q$.}\label{fig:h-iterates}
  \end{figure}
For $n=1$ this is easy. Let us assume $n>2$. We know that $h^{n-1}\alpha$ contains a $\beta$-arc of size $q$, and in particular a $\beta$-arc of size $1$.
The image of this arc under $f^p$ contains an $\alpha$-arc of size $1$, and the image of such an arc under $g^q$ contains the wanted curve.
\end{proof}
The arc provided by the lemma is distance $q+2$ from $\alpha$ in the arc graph $\mathcal{A}^\dagger(A_\beta)$, and no arc disjoint from it is distance $q+3$.
Thus we get the twist number
$\mathrm{tw}_\beta(\alpha, h^n \alpha) = q+2$.

Analogous considerations give, for every $n>0$, $\mathrm{tw}_\beta(\alpha, h^{-n} \alpha) = 2$.
Applying the triangular inequality for twists gives the first part of the following lemma. The second part follows from analogous considerations.
\begin{lemma}\label{lem:explicit-twists}~
\begin{enumerate}
\item For every $m\geq 0, n>0$,
$\mathrm{tw}_\beta(h^{-m}\alpha, h^{n} \alpha) \geq q$.
\item For every $m, n>0$, the curves $h^{-m}\alpha, h^{n} \alpha$ have non-empty projection under the map $\pi_{A_\alpha}$, and  $\mathrm{tw}_\alpha(h^{-m}\alpha, h^{n} \alpha) \geq p$.
\end{enumerate}
\end{lemma}

\begin{proof}[Proof of Proposition~\ref{prop:tl}]
Let $n>0$. Consider a geodesic path $\gamma'$ from $\alpha$ to $h^n \alpha$. 
For every $k = 0, \dots, n-1$, we have, by point (1) of the lemma,
\[
\mathrm{tw}_{h^k \beta}(\alpha, h^{n} \alpha) = \mathrm{tw}_{\beta}(h^{-k}\alpha, h^{n-k} \alpha) \geq q > B
\]
and likewise, for $k = 1, \dots, n-1$, $\mathrm{tw}_{h^k \alpha}(\alpha, h^{n} \alpha) \geq p >B$.
By the Bounded Geodesic Image Theorem, the geodesic $\gamma'$ has to "lose projection" on each of the annuli determines by the following curves, which constitutes a segment of our bi-infinite path $\gamma$:
\[
\beta, h \alpha, h \beta, \dots, h^{n-1} \alpha, h^{n-1} \beta.
\]
This means that we can find a curves $c_0, d_1, \dots , d_{n-1}, c_{n-1}$ which are vertices of the geodesic $\gamma'$ and such that 
$\pi_{h^k \beta}(c_k) = 0$, $\pi_{h^k \alpha}(d_k) = 0$.

\begin{claim}
The $c_k$'s and $d_k$'s are pairwise distinct.
\end{claim}
Let us prove the claim.
The curve $c_k$ does not project in the annulus determined by $h^k \beta$, thus it is homotopic to $h^k \beta$, which is homotopic to $\beta$. Likewise, each $d_j$ is homotopic to $\alpha$. Thus the $c_k$'s are certainly distinct from the $d_j$'s.
Let us check that the $d_j$'s are also pairwise disjoint (the proof for the $c_k$'s is similar). By applying a power of $h$ it suffices to prove the following second claim.
\begin{claim}
Let $k>0$. If $q>2$ then there is no curve $\delta$ such that both $\pi_{A_\alpha}(\delta) = \emptyset$ and 
$\pi_{h^k(A_\alpha)}(\delta) = \emptyset$.
\end{claim}
Assume for a contradiction that such a $\delta$ exists. By assumption $\delta$ does not cross the annulus determined by $\alpha$. Thus there is a lift $\tilde{\delta}$ which is between $\tilde{\alpha}$ and $\tilde{\alpha} + (0, 2)$, which we denote 
\[
\tilde{\alpha} < \tilde{\delta} < \tilde{\alpha} + (0, 2).
\]
Likewise, there is an integer $\tau$ such that 
$\tilde{h}^k\tilde{\alpha} < \tilde{\delta} + (0, \tau) < \tilde{h}^k\tilde{\alpha} + (0, 2)$, and thus
\[
\tilde{h}^k\tilde{\alpha} -(0, \tau) < \tilde{\delta} < \tilde{h}^k\tilde{\alpha} + (0, 2 - \tau).
\]
 We first deduce that
$\tilde{\alpha} < \tilde{h}^k\tilde{\alpha} + (0, 2 - \tau)$ which implies $2-\tau > 0$ since $\tilde{h}^k\tilde{\alpha}$ meets $\tilde{\alpha}$ at the point $(0,0)$.
And also 
$\tilde{h}^k\tilde{\alpha} -(0, \tau) < \tilde{\alpha} + (0, 2)$,
and since $\tilde{h}^k\tilde{\alpha}$ meets $\tilde{\alpha} + (0, q)$,
this implies $2+\tau > q$. We conclude that $q-2 < \tau < 2$, and since all these numbers are integers, $q \leq 2$. This contradicts the assumption on $q$. This completes the proof of both claims.

\bigskip

From the first claim we conclude that the length of $\gamma'$ is at least $2n$. Since this is exactly the length of the path $\gamma$ which has the same end-points, we conclude that $d(\alpha, h^n \alpha ) = 2n$. Thus $\gamma$ is a geodesic.
\end{proof}
In the above proof we seemingly were lucky to find exactly the good number of distinct annuli in which our geodesic $\gamma$ does not project. Actually the proof is a little more robust than this would make you believe: indeed, assume we had only found, say, $2n-10$ such vertices; then we first conclude that $d(x, h^n(x)) \geq 2n-9$, but then $\mathrm{tl}(h) \geq 2$, and from this it follows that $d(x, h^n(x)) = 2n$, since any smaller number would give a smaller translation length.

\bigskip

We finally search for an estimate on the twist number of $h$, showing that it is coarsely equal to $\max(p,q)$. We denote $B$ the constant from the second Bounded Geodesic Image Theorem, as above.
\begin{proposition}
Assume $p,q > B$. Then
\[
\mathrm{tw}(h) \in [\max(p,q), \max(p, q)+4].
\]
\end{proposition}
\begin{proof} 
From Lemma~\ref{lem:explicit-twists} we get that 
\[
\mathrm{tw}_\beta(\xi^-(h), \xi^+(h)) \geq q, \mathrm{tw}_\alpha(\xi^-(h), \xi^+(h)) \geq p
\]
and thus $\mathrm{tw}(h) \geq \max(p,q)$.

It remains to prove the upper bound.
Assume that $\mathrm{tw}(h) > B+4$, otherwise there is nothing to prove. It suffices to prove that under this assumption, $\mathrm{tw}(h) \leq \max(p,q)+4$.
By definition, there is some annulus $A = A_c$ for some curve $c$ such that
$\mathrm{tw}_A(\xi^-(h), \xi^+(h)) > B+4$. By the usual argument we may pick $n$ large enough so that $\mathrm{tw}_A(h^{-n}\alpha,h^n\alpha) \geq \mathrm{tw}_A(\xi^-(h), \xi^+(h)) -4 > B$. By the Bounded Geodesic Image Theorem, the geodesic $\gamma$ must lose projection in $A$: there is some $k$ such that $c$ does not project either in the annulus bounded by $h^k\alpha$ or in the annulus bounded by $h^k\beta$. We treat the first case, the second one is similar. The curve $c' = h^{-k} c$ does not project in $A_\alpha$, and we have $\mathrm{tw}_{c'}(h^{-n-k}\alpha,h^{n-k}\alpha) > B$.
We can have chosen $n$ larger than $|k|+1$, so that $-n-k$ and $n-k$ are still respectively negative and positive.
Since $\pi_\alpha(c') = \emptyset$, there is a lift $\tilde{c'}$ which is between $\tilde{\alpha}$ and $\tilde{\alpha}+(0,2)$. The strip $\tilde{A'}$ bounded by $\tilde{c'}$ and $\tilde{c'} + (0,1)$ is a universal cover of the annulus $A' = A_c$, and it is included in the strip $\tilde{A}$ bounded by $\tilde{\alpha}$ and $\tilde{\alpha} + (3,0)$. Let $A$ denote the annulus obtained by quotienting $\tilde{A}$ by the horizontal unit translation. Consider the torus $\mathbb{R}^2/(\mathbb{Z} \times 3\mathbb{Z})$, which is a degree 3 cover of our initial torus, and denote $\bar h$ a lift of $h$, and $\bar \alpha$ the projection of $\tilde{\alpha}$ in this cover. The annulus cut by $\bar \alpha$ in this cover identifies with $A$.
The twist number $\mathrm{tw}_{A}(\bar h^{-n-k}\bar \alpha,\bar h^{n-k}\bar \alpha)$ 
is, by definition, the Hausdorff distance in the fine arc graph $\mathcal{A}^\dagger(A)$ between the projections of the curves $\bar h^{-n-k}\bar \alpha$ and $\bar h^{n-k}\bar \alpha$.
The inclusion of strips implies that 
$\mathrm{tw}_{A'}(h^{-n-k}\alpha,h^{n-k}\alpha) \leq \mathrm{tw}_{A}(\bar h^{-n-k}\bar \alpha,\bar h^{n-k}\bar \alpha)$.

It remains to check that, given $i,j >0$,
$\mathrm{tw}_{A}(\bar h^{-i}\bar \alpha,\bar h^{j}\bar \alpha)  \leq p+4$. Here we use the hypothesis that $q \geq 3$.
The curve $\tilde{h}^j\tilde{\alpha}$ contains the image under $h$ of a good $\beta$-arc of size $1$, and this contains an arc $a^+$ originated in $(0,0)$ that crosses the strip $\tilde{A}$ and is included in $[0,1] \times [0,3]$. 
The curve $\tilde{h}^{-i}(\tilde{\alpha}+ (0,3))$ contains an arc $a^-$ originated from $(0,3)$ that crosses the strip $\tilde{A}$ and included in $[-p-1, 0] \times [0,3]$. The arc $a^-$ meet at most $p+1$ translates of the arc $a^+$, thus their distance in the arc graph is at most $p+2$. Any arc disjoint from $a^-$ is at distance at most $p+4$ from any arc disjoint from $a^+$. Thus we have $\mathrm{tw}_{A}(h^{-i}\alpha,h^{j}\alpha)  \leq p+4$, as wanted.
\end{proof}

\bigskip

Finally, we consider a variation on the family $h_{p,q}$ to provide an example where the twist number are useful to prove non conjugacy.
Let $p,q,p',q'>0$, and consider $h = h_{p,q,p',q'} = g^{q'}f^{p'}g^q f^p$.

\begin{proposition}\label{prop:4-param}
Assume $p,p',q,q' >B$. Then
\[
\mathrm{Rot}(h) = [0, p+p'] \times [0, q+q'],
\]
\[
\mathrm{tl}(h) = 4,
\]
\[
\mathrm{tw}(h) \in [\max(p,p',q,q'), \max(p,p',q,q') + 4]
\]
\end{proposition}
The proof is very similar to the previous proofs, and is left to the reader. As a corollary, we see that this family contains homeomorphisms that have the same rotation set and the same translation length but whose twist numbers differ by more than 4; in particular they are not conjugate, and by Corollary~\ref{cor:twist-fujiwara} they have pairwise distinct fixed points at infinity.

\bigskip

Let us end this subsection with more free groups. The rotation set of $h_{n,n} = g^n f^n$ is $[0,n]^2$ and its translation length is $2$. Thus all the $h_{n,n}$'s have distinct scaled rotation sets. By Theorem~\ref{thm:boundary-determines-rotation-set} they have distinct boudary fixed points, though they have homothetic rotation sets (and thus Theorem~\ref{thm:free-group} does not directly apply). The ping-pong argument  tells us that up to taking powers, any two of them generate a free group. Likewise, Proposition~\ref{prop:4-param} provides plenty of examples of pairs $h,h'$ of maps sharing the same scaled rotation sets, but with distinct twist numbers. Choose $h, h'$ so that their twists numbers are more than 4 apart. Then by Corollary~\ref{cor:twist-fujiwara} $h$ and $h'$ have no boundary fixed points in common. Again, they have powers generating a free group.

\subsection{Schottky groups}\label{ssec:schottky}
We now collect some examples which
show that some generalizations of our main results which one might
hope for are in fact false.

All of these examples rely on a Schottky group acting on the fine
curve graph. We begin with two homeomorphisms
$f,g \in \mathrm{Homeo}_0(T)$ with distinct rotation sets and
independent axes in the fine curve graph (note that this is implied by
our main result). A standard ping-pong argument (or by using
Lemma~\ref{lem:local-to-global}) one shows that by replacing $f, g$ by
large enough powers we may assume that $f,g$ generate a free group
$F_2 < \mathrm{Homeo}_0(T)$ so that the orbit map
\[ F_2 \to \mathcal{C}^\dagger(T) \]
is a quasi-isometric embedding. For our examples to work in this general setting, we would need to know that the scaled rotation set of $f g^n$ converges to that of $g$ when $n$ tends to $+\infty$ (see question~\ref{ques:projective-cv}).
We now specify $f$ and $g$ so that this kind of convergence holds. We will assume that
\begin{enumerate}
\item There are two disjoint topological closed disks $C(f), D(f)$ of $T^2$, such that $f$ is the identity on $D(f)$, and every rational vector in $\mathrm{Rot}(f)$ is the rotation vector of some periodic point in $C(f)$,
\item Likewise for $g$,
\item $C(f) \subset D(g), C(g) \subset D(f)$.
\end{enumerate}
Examples with these properties may be constructed  the following way. Let $P$ be a rational polygon containing $0$. Kwapisz \cite{kwapisz_realise} constructs an Axiom A diffeomorphisms $f$ whose rotation set is $P$, whose periodic orbits belong to some Cantor set, except for some sinks and sources that have rotation vector $0$, and $0$ is also the rotation vector of some fixed point in the Cantor set. One may blow-up one of the source or sink to get a disk $D(f)$ of fixed points disjoint from the Cantor set. Now one can choose $C(f)$ to be another disc that contains the Cantor set and is disjoint from $D(f)$. The map $g$ is constructed analogously, and then conjugated to get property 3.

These properties have the following consequence. Let $h$ be some element of the Schottky group generated by $f$ and $g$, and respectively denote $n,m$ the sum of powers of $f$ and $g$ in $h$. Iterating $h$ on a point of $C(f)$ amounts to iterating just $f$, and likewise for $g$. Thus we get the lower bound
\[
\mathrm{Rot}(h) \supseteq \frac{1}{n+m}\left(n\mathrm{Rot}(f) + m \mathrm{Rot}(g)\right).
\]
Combining this with the upper bound lemma~\ref{lem:upper-bound-rotation-set}, we get that for every $k$,
\[  
\frac{1}{m}\mathrm{Rot}(f^kg^m) \xrightarrow{m\to\infty} \mathrm{Rot}(g).
\]

We can now construct our examples. 
\begin{itemize}
\item Consider the element $\varphi_{k,m} = f^kg^m$. By choosing $k$
  large, a quasi-axis for $\varphi_{k,m}$ can be made to fellow-travel
  a quasi-axis for $f$ for an arbitrarily long time (since this is
  true in the Cayley graph of the free group, and the fact that it
  quasi-isometrically embeds into the fine curve graph). On the other
  hand, we have that the scaled rotation set of $f^kg^m$ converges to the
  scaled rotation set of $g$ if $k$ is fixed and $m\to\infty$.

\item Consider an infinite path in $F_2$ of the form
  \[ P = f^{n_1}g^{m_1}f^{n_2}g^{m_2}\cdots \]
  Observe that for any homeomorphisms $A, B$ of the plane, by the Misiurewicz-Ziemian theorem we have
  \[ \frac{1}{m}AB^m([0,1]^2) = \frac{1}{m}(ABA^{-1})^m(A[0,1]^2) \xrightarrow{m\to\infty} \mathrm{Rot}(B). \]
  Thus, given $n_i,m_i, i\leq N$ we can choose $n_{i+1},m_{i+1}$ so that
  \[ f^{n_1}g^{m_1}\cdots f^{n_i}g^{m_i}f^{n_{i+1}}[0,1]^2 \]
  is projectively close to the rotation set of $f$, and
  \[ f^{n_1}g^{m_1}\cdots f^{n_i}g^{m_i}f^{n_{i+1}}g^{m_{i+1}}[0,1]^2 \]
  is projectively close to the rotation set of $g$.

  As a consequence, the infinite product of $f,g$ given by $P$ does
  not have a well-defined projective rotation set. Hence, there is
  also no obvious way to associate a (projective) rotation set to an arbitrary
  boundary point of the fine curve graph.
\end{itemize}

\bibliographystyle{alpha}
\bibliography{ref-semicontinuity}

\end{document}

%% file: lem_3_2.pdf_tex
\begingroup%
  \makeatletter%
  \providecommand\color[2][]{%
    \errmessage{(Inkscape) Color is used for the text in Inkscape, but the package 'color.sty' is not loaded}%
    \renewcommand\color[2][]{}%
  }%
  \providecommand\transparent[1]{%
    \errmessage{(Inkscape) Transparency is used (non-zero) for the text in Inkscape, but the package 'transparent.sty' is not loaded}%
    \renewcommand\transparent[1]{}%
  }%
  \providecommand\rotatebox[2]{#2}%
  \newcommand*\fsize{\dimexpr\f@size pt\relax}%
  \newcommand*\lineheight[1]{\fontsize{\fsize}{#1\fsize}\selectfont}%
  \ifx\svgwidth\undefined%
    \setlength{\unitlength}{406.20650644bp}%
    \ifx\svgscale\undefined%
      \relax%
    \else%
      \setlength{\unitlength}{\unitlength * \real{\svgscale}}%
    \fi%
  \else%
    \setlength{\unitlength}{\svgwidth}%
  \fi%
  \global\let\svgwidth\undefined%
  \global\let\svgscale\undefined%
  \makeatother%
  \begin{picture}(1,0.336463)%
    \lineheight{1}%
    \setlength\tabcolsep{0pt}%
    \put(0,0){\includegraphics[width=\unitlength,page=1]{lem_3_2.pdf}}%
    \put(0.00586679,0.29598495){\makebox(0,0)[lt]{\lineheight{1.25}\smash{\begin{tabular}[t]{l}$\alpha$\end{tabular}}}}%
    \put(0.00184629,0.0065091){\makebox(0,0)[lt]{\lineheight{1.25}\smash{\begin{tabular}[t]{l}$\beta$\end{tabular}}}}%
    \put(0.82805835,0.30804647){\makebox(0,0)[lt]{\lineheight{1.25}\smash{\begin{tabular}[t]{l}$\alpha'$\end{tabular}}}}%
    \put(0.83207875,0.01857068){\makebox(0,0)[lt]{\lineheight{1.25}\smash{\begin{tabular}[t]{l}$\beta'$\end{tabular}}}}%
    \put(0.30740408,0.23768774){\makebox(0,0)[lt]{\lineheight{1.25}\smash{\begin{tabular}[t]{l}$\alpha_\ell$\end{tabular}}}}%
    \put(0.41193703,0.0808884){\makebox(0,0)[lt]{\lineheight{1.25}\smash{\begin{tabular}[t]{l}$\beta_\ell$\end{tabular}}}}%
    \put(0.29333235,0.0788782){\makebox(0,0)[lt]{\lineheight{1.25}\smash{\begin{tabular}[t]{l}$\beta_{\overline{\ell}}$\end{tabular}}}}%
    \put(0,0){\includegraphics[width=\unitlength,page=2]{lem_3_2.pdf}}%
  \end{picture}%
\endgroup%

%% file: lem_3_3.pdf_tex
\begingroup%
  \makeatletter%
  \providecommand\color[2][]{%
    \errmessage{(Inkscape) Color is used for the text in Inkscape, but the package 'color.sty' is not loaded}%
    \renewcommand\color[2][]{}%
  }%
  \providecommand\transparent[1]{%
    \errmessage{(Inkscape) Transparency is used (non-zero) for the text in Inkscape, but the package 'transparent.sty' is not loaded}%
    \renewcommand\transparent[1]{}%
  }%
  \providecommand\rotatebox[2]{#2}%
  \newcommand*\fsize{\dimexpr\f@size pt\relax}%
  \newcommand*\lineheight[1]{\fontsize{\fsize}{#1\fsize}\selectfont}%
  \ifx\svgwidth\undefined%
    \setlength{\unitlength}{452.77019831bp}%
    \ifx\svgscale\undefined%
      \relax%
    \else%
      \setlength{\unitlength}{\unitlength * \real{\svgscale}}%
    \fi%
  \else%
    \setlength{\unitlength}{\svgwidth}%
  \fi%
  \global\let\svgwidth\undefined%
  \global\let\svgscale\undefined%
  \makeatother%
  \begin{picture}(1,0.31749679)%
    \lineheight{1}%
    \setlength\tabcolsep{0pt}%
    \put(0,0){\includegraphics[width=\unitlength,page=1]{lem_3_3.pdf}}%
    \put(0.00526344,0.26778801){\makebox(0,0)[lt]{\lineheight{1.25}\smash{\begin{tabular}[t]{l}$\alpha$\end{tabular}}}}%
    \put(0.00165642,0.00808238){\makebox(0,0)[lt]{\lineheight{1.25}\smash{\begin{tabular}[t]{l}$\beta$\end{tabular}}}}%
    \put(0.74289936,0.27860911){\makebox(0,0)[lt]{\lineheight{1.25}\smash{\begin{tabular}[t]{l}$\alpha'$\end{tabular}}}}%
    \put(0.7465063,0.01890352){\makebox(0,0)[lt]{\lineheight{1.25}\smash{\begin{tabular}[t]{l}$\beta'$\end{tabular}}}}%
    \put(0.2757901,0.21548619){\makebox(0,0)[lt]{\lineheight{1.25}\smash{\begin{tabular}[t]{l}$\alpha_{\overline{k}}$\end{tabular}}}}%
    \put(0.56668088,0.13422203){\makebox(0,0)[lt]{\lineheight{1.25}\smash{\begin{tabular}[t]{l}$\beta_{\overline{\ell}}$\end{tabular}}}}%
    \put(0,0){\includegraphics[width=\unitlength,page=2]{lem_3_3.pdf}}%
    \put(0.2553857,0.29200269){\makebox(0,0)[lt]{\lineheight{1.25}\smash{\begin{tabular}[t]{l}$\hat{\alpha}$\end{tabular}}}}%
    \put(0,0){\includegraphics[width=\unitlength,page=3]{lem_3_3.pdf}}%
    \put(0.5716538,0.00634119){\makebox(0,0)[lt]{\lineheight{1.25}\smash{\begin{tabular}[t]{l}$\hat{\beta}$\end{tabular}}}}%
  \end{picture}%
\endgroup%

%% file: cor_3_4_2.pdf_tex
\begingroup%
  \makeatletter%
  \providecommand\color[2][]{%
    \errmessage{(Inkscape) Color is used for the text in Inkscape, but the package 'color.sty' is not loaded}%
    \renewcommand\color[2][]{}%
  }%
  \providecommand\transparent[1]{%
    \errmessage{(Inkscape) Transparency is used (non-zero) for the text in Inkscape, but the package 'transparent.sty' is not loaded}%
    \renewcommand\transparent[1]{}%
  }%
  \providecommand\rotatebox[2]{#2}%
  \newcommand*\fsize{\dimexpr\f@size pt\relax}%
  \newcommand*\lineheight[1]{\fontsize{\fsize}{#1\fsize}\selectfont}%
  \ifx\svgwidth\undefined%
    \setlength{\unitlength}{456.8241954bp}%
    \ifx\svgscale\undefined%
      \relax%
    \else%
      \setlength{\unitlength}{\unitlength * \real{\svgscale}}%
    \fi%
  \else%
    \setlength{\unitlength}{\svgwidth}%
  \fi%
  \global\let\svgwidth\undefined%
  \global\let\svgscale\undefined%
  \makeatother%
  \begin{picture}(1,0.32224698)%
    \lineheight{1}%
    \setlength\tabcolsep{0pt}%
    \put(0,0){\includegraphics[width=\unitlength,page=1]{cor_3_4_2.pdf}}%
    \put(0.00521674,0.27495597){\makebox(0,0)[lt]{\lineheight{1.25}\smash{\begin{tabular}[t]{l}$\alpha$\end{tabular}}}}%
    \put(0.00164173,0.01755504){\makebox(0,0)[lt]{\lineheight{1.25}\smash{\begin{tabular}[t]{l}$\beta$\end{tabular}}}}%
    \put(0.73630666,0.28568104){\makebox(0,0)[lt]{\lineheight{1.25}\smash{\begin{tabular}[t]{l}$f^{2n}\alpha$\end{tabular}}}}%
    \put(0.72021906,0.01934265){\makebox(0,0)[lt]{\lineheight{1.25}\smash{\begin{tabular}[t]{l}$g^{2n}\beta$\end{tabular}}}}%
    \put(0.34752414,0.29729976){\makebox(0,0)[lt]{\lineheight{1.25}\smash{\begin{tabular}[t]{l}$f^{n}\alpha$\end{tabular}}}}%
    \put(0.35109911,0.00683002){\makebox(0,0)[lt]{\lineheight{1.25}\smash{\begin{tabular}[t]{l}$g^{n}\beta$\end{tabular}}}}%
    \put(0,0){\includegraphics[width=\unitlength,page=2]{cor_3_4_2.pdf}}%
  \end{picture}%
\endgroup%

%% file: lem_3_5.pdf_tex
\begingroup%
  \makeatletter%
  \providecommand\color[2][]{%
    \errmessage{(Inkscape) Color is used for the text in Inkscape, but the package 'color.sty' is not loaded}%
    \renewcommand\color[2][]{}%
  }%
  \providecommand\transparent[1]{%
    \errmessage{(Inkscape) Transparency is used (non-zero) for the text in Inkscape, but the package 'transparent.sty' is not loaded}%
    \renewcommand\transparent[1]{}%
  }%
  \providecommand\rotatebox[2]{#2}%
  \newcommand*\fsize{\dimexpr\f@size pt\relax}%
  \newcommand*\lineheight[1]{\fontsize{\fsize}{#1\fsize}\selectfont}%
  \ifx\svgwidth\undefined%
    \setlength{\unitlength}{435.82620936bp}%
    \ifx\svgscale\undefined%
      \relax%
    \else%
      \setlength{\unitlength}{\unitlength * \real{\svgscale}}%
    \fi%
  \else%
    \setlength{\unitlength}{\svgwidth}%
  \fi%
  \global\let\svgwidth\undefined%
  \global\let\svgscale\undefined%
  \makeatother%
  \begin{picture}(1,0.60692391)%
    \lineheight{1}%
    \setlength\tabcolsep{0pt}%
    \put(0,0){\includegraphics[width=\unitlength,page=1]{lem_3_5.pdf}}%
    \put(0.39823178,0.40522691){\makebox(0,0)[lt]{\lineheight{1.25}\smash{\begin{tabular}[t]{l}$g_1$\end{tabular}}}}%
    \put(0.38768807,0.23194401){\makebox(0,0)[lt]{\lineheight{1.25}\smash{\begin{tabular}[t]{l}$g_2$\end{tabular}}}}%
    \put(0.82986413,0.58043865){\makebox(0,0)[lt]{\lineheight{1.25}\smash{\begin{tabular}[t]{l}$g_1(s)$\end{tabular}}}}%
    \put(0.84860036,0.00710871){\makebox(0,0)[lt]{\lineheight{1.25}\smash{\begin{tabular}[t]{l}$g_1(t)$\end{tabular}}}}%
    \put(0,0){\includegraphics[width=\unitlength,page=2]{lem_3_5.pdf}}%
    \put(0.51916352,0.3431309){\makebox(0,0)[lt]{\lineheight{1.25}\smash{\begin{tabular}[t]{l}$\delta$\end{tabular}}}}%
    \put(0,0){\includegraphics[width=\unitlength,page=3]{lem_3_5.pdf}}%
  \end{picture}%
\endgroup%

%% file: lem_3_6.pdf_tex
\begingroup%
  \makeatletter%
  \providecommand\color[2][]{%
    \errmessage{(Inkscape) Color is used for the text in Inkscape, but the package 'color.sty' is not loaded}%
    \renewcommand\color[2][]{}%
  }%
  \providecommand\transparent[1]{%
    \errmessage{(Inkscape) Transparency is used (non-zero) for the text in Inkscape, but the package 'transparent.sty' is not loaded}%
    \renewcommand\transparent[1]{}%
  }%
  \providecommand\rotatebox[2]{#2}%
  \newcommand*\fsize{\dimexpr\f@size pt\relax}%
  \newcommand*\lineheight[1]{\fontsize{\fsize}{#1\fsize}\selectfont}%
  \ifx\svgwidth\undefined%
    \setlength{\unitlength}{306.46241196bp}%
    \ifx\svgscale\undefined%
      \relax%
    \else%
      \setlength{\unitlength}{\unitlength * \real{\svgscale}}%
    \fi%
  \else%
    \setlength{\unitlength}{\svgwidth}%
  \fi%
  \global\let\svgwidth\undefined%
  \global\let\svgscale\undefined%
  \makeatother%
  \begin{picture}(1,0.49709808)%
    \lineheight{1}%
    \setlength\tabcolsep{0pt}%
    \put(0,0){\includegraphics[width=\unitlength,page=1]{lem_3_6.pdf}}%
    \put(0.39795383,0.00862764){\makebox(0,0)[lt]{\lineheight{1.25}\smash{\begin{tabular}[t]{l}$x$\end{tabular}}}}%
    \put(0.85053411,0.44638068){\makebox(0,0)[lt]{\lineheight{1.25}\smash{\begin{tabular}[t]{l}$f(x)$\end{tabular}}}}%
    \put(-0.00370439,0.45943284){\makebox(0,0)[lt]{\lineheight{1.25}\smash{\begin{tabular}[t]{l}$f^{-1}(x)$\end{tabular}}}}%
    \put(0,0){\includegraphics[width=\unitlength,page=2]{lem_3_6.pdf}}%
    \put(0.3418385,0.24502681){\makebox(0,0)[lt]{\lineheight{1.25}\smash{\begin{tabular}[t]{l}$y$\end{tabular}}}}%
    \put(0,0){\includegraphics[width=\unitlength,page=3]{lem_3_6.pdf}}%
    \put(0.56961588,0.28510633){\makebox(0,0)[lt]{\lineheight{1.25}\smash{\begin{tabular}[t]{l}$f(y)$\end{tabular}}}}%
    \put(0,0){\includegraphics[width=\unitlength,page=4]{lem_3_6.pdf}}%
    \put(0.45111176,0.25971052){\makebox(0,0)[lt]{\lineheight{1.25}\smash{\begin{tabular}[t]{l}$z_1$\end{tabular}}}}%
    \put(0,0){\includegraphics[width=\unitlength,page=5]{lem_3_6.pdf}}%
    \put(0.46077542,0.34843976){\makebox(0,0)[lt]{\lineheight{1.25}\smash{\begin{tabular}[t]{l}$z_2$\end{tabular}}}}%
    \put(0,0){\includegraphics[width=\unitlength,page=6]{lem_3_6.pdf}}%
    \put(0.27331454,0.32603422){\makebox(0,0)[lt]{\lineheight{1.25}\smash{\begin{tabular}[t]{l}$g$\end{tabular}}}}%
    \put(0,0){\includegraphics[width=\unitlength,page=7]{lem_3_6.pdf}}%
    \put(0.66487971,0.34502692){\makebox(0,0)[lt]{\lineheight{1.25}\smash{\begin{tabular}[t]{l}$f(g)$\end{tabular}}}}%
    \put(0.35936944,0.16939665){\makebox(0,0)[lt]{\lineheight{1.25}\smash{\begin{tabular}[t]{l}$4\delta$\end{tabular}}}}%
  \end{picture}%
\endgroup%

%% file: marking_move.pdf_tex
\begingroup%
  \makeatletter%
  \providecommand\color[2][]{%
    \errmessage{(Inkscape) Color is used for the text in Inkscape, but the package 'color.sty' is not loaded}%
    \renewcommand\color[2][]{}%
  }%
  \providecommand\transparent[1]{%
    \errmessage{(Inkscape) Transparency is used (non-zero) for the text in Inkscape, but the package 'transparent.sty' is not loaded}%
    \renewcommand\transparent[1]{}%
  }%
  \providecommand\rotatebox[2]{#2}%
  \newcommand*\fsize{\dimexpr\f@size pt\relax}%
  \newcommand*\lineheight[1]{\fontsize{\fsize}{#1\fsize}\selectfont}%
  \ifx\svgwidth\undefined%
    \setlength{\unitlength}{455.14039047bp}%
    \ifx\svgscale\undefined%
      \relax%
    \else%
      \setlength{\unitlength}{\unitlength * \real{\svgscale}}%
    \fi%
  \else%
    \setlength{\unitlength}{\svgwidth}%
  \fi%
  \global\let\svgwidth\undefined%
  \global\let\svgscale\undefined%
  \makeatother%
  \begin{picture}(1,0.30523049)%
    \lineheight{1}%
    \setlength\tabcolsep{0pt}%
    \put(0,0){\includegraphics[width=\unitlength,page=1]{marking_move.pdf}}%
    \put(0.23387972,0.09280723){\makebox(0,0)[lt]{\lineheight{1.25}\smash{\begin{tabular}[t]{l}$\alpha$\end{tabular}}}}%
    \put(0.10221198,0.04949489){\makebox(0,0)[lt]{\lineheight{1.25}\smash{\begin{tabular}[t]{l}$\beta$\end{tabular}}}}%
    \put(0.75588983,0.09164203){\makebox(0,0)[lt]{\lineheight{1.25}\smash{\begin{tabular}[t]{l}$\alpha'$\end{tabular}}}}%
    \put(0.74773345,0.01104325){\makebox(0,0)[lt]{\lineheight{1.25}\smash{\begin{tabular}[t]{l}$\beta'$\end{tabular}}}}%
  \end{picture}%
\endgroup%

%% file: lem_5_2.pdf_tex
\begingroup%
  \makeatletter%
  \providecommand\color[2][]{%
    \errmessage{(Inkscape) Color is used for the text in Inkscape, but the package 'color.sty' is not loaded}%
    \renewcommand\color[2][]{}%
  }%
  \providecommand\transparent[1]{%
    \errmessage{(Inkscape) Transparency is used (non-zero) for the text in Inkscape, but the package 'transparent.sty' is not loaded}%
    \renewcommand\transparent[1]{}%
  }%
  \providecommand\rotatebox[2]{#2}%
  \newcommand*\fsize{\dimexpr\f@size pt\relax}%
  \newcommand*\lineheight[1]{\fontsize{\fsize}{#1\fsize}\selectfont}%
  \ifx\svgwidth\undefined%
    \setlength{\unitlength}{412.21390852bp}%
    \ifx\svgscale\undefined%
      \relax%
    \else%
      \setlength{\unitlength}{\unitlength * \real{\svgscale}}%
    \fi%
  \else%
    \setlength{\unitlength}{\svgwidth}%
  \fi%
  \global\let\svgwidth\undefined%
  \global\let\svgscale\undefined%
  \makeatother%
  \begin{picture}(1,0.19626027)%
    \lineheight{1}%
    \setlength\tabcolsep{0pt}%
    \put(0,0){\includegraphics[width=\unitlength,page=1]{lem_5_2.pdf}}%
    \put(0.00945527,0.00641427){\makebox(0,0)[lt]{\lineheight{1.25}\smash{\begin{tabular}[t]{l}$m_0$\end{tabular}}}}%
    \put(0.09265526,0.12773877){\makebox(0,0)[lt]{\lineheight{1.25}\smash{\begin{tabular}[t]{l}$f_1m_0$\end{tabular}}}}%
    \put(0.18651127,0.01077224){\makebox(0,0)[lt]{\lineheight{1.25}\smash{\begin{tabular}[t]{l}$f_1f_2m_0$\end{tabular}}}}%
    \put(0.61745466,0.16825791){\makebox(0,0)[lt]{\lineheight{1.25}\smash{\begin{tabular}[t]{l}$f_1\cdots f_n m_0$\end{tabular}}}}%
  \end{picture}%
\endgroup%

%% file: lem_5_2_2.pdf_tex
\begingroup%
  \makeatletter%
  \providecommand\color[2][]{%
    \errmessage{(Inkscape) Color is used for the text in Inkscape, but the package 'color.sty' is not loaded}%
    \renewcommand\color[2][]{}%
  }%
  \providecommand\transparent[1]{%
    \errmessage{(Inkscape) Transparency is used (non-zero) for the text in Inkscape, but the package 'transparent.sty' is not loaded}%
    \renewcommand\transparent[1]{}%
  }%
  \providecommand\rotatebox[2]{#2}%
  \newcommand*\fsize{\dimexpr\f@size pt\relax}%
  \newcommand*\lineheight[1]{\fontsize{\fsize}{#1\fsize}\selectfont}%
  \ifx\svgwidth\undefined%
    \setlength{\unitlength}{350.64240299bp}%
    \ifx\svgscale\undefined%
      \relax%
    \else%
      \setlength{\unitlength}{\unitlength * \real{\svgscale}}%
    \fi%
  \else%
    \setlength{\unitlength}{\svgwidth}%
  \fi%
  \global\let\svgwidth\undefined%
  \global\let\svgscale\undefined%
  \makeatother%
  \begin{picture}(1,0.65286362)%
    \lineheight{1}%
    \setlength\tabcolsep{0pt}%
    \put(0,0){\includegraphics[width=\unitlength,page=1]{lem_5_2_2.pdf}}%
    \put(-0.00236731,0.20468214){\color[rgb]{0,0,0}\makebox(0,0)[lt]{\lineheight{1.25}\smash{\begin{tabular}[t]{l}$\alpha_0$\end{tabular}}}}%
    \put(0.20065841,0.03062399){\color[rgb]{0,0,0.99607843}\makebox(0,0)[lt]{\lineheight{1.25}\smash{\begin{tabular}[t]{l}$\beta_0$\end{tabular}}}}%
    \put(0.27710353,0.00672594){\color[rgb]{0.34901961,0,0.65882353}\makebox(0,0)[lt]{\lineheight{1.25}\smash{\begin{tabular}[t]{l}$f_1f_0\beta_0$\end{tabular}}}}%
    \put(0.29650439,0.09768872){\color[rgb]{0.74117647,0,0.42352941}\makebox(0,0)[lt]{\lineheight{1.25}\smash{\begin{tabular}[t]{l}$f_2f_1f_0\beta_0$\end{tabular}}}}%
    \put(0.39689381,0.02505028){\color[rgb]{0,0,0}\makebox(0,0)[lt]{\lineheight{1.25}\smash{\begin{tabular}[t]{l}$\color{red}f_3f_2f_1f_0\beta_0=\beta_1$\end{tabular}}}}%
    \put(0.12106726,0.5888623){\color[rgb]{0,0,0}\makebox(0,0)[lt]{\lineheight{1.25}\smash{\begin{tabular}[t]{l}$U_2$\end{tabular}}}}%
    \put(0.11829412,0.35205208){\color[rgb]{0,0,0}\makebox(0,0)[lt]{\lineheight{1.25}\smash{\begin{tabular}[t]{l}$U_4$\end{tabular}}}}%
    \put(0.10225213,0.27107828){\color[rgb]{0,0,0}\makebox(0,0)[lt]{\lineheight{1.25}\smash{\begin{tabular}[t]{l}$U_2$\end{tabular}}}}%
    \put(0.45059236,0.58045937){\color[rgb]{0,0,0}\makebox(0,0)[lt]{\lineheight{1.25}\smash{\begin{tabular}[t]{l}$U_3$\end{tabular}}}}%
    \put(0.46052313,0.28024511){\color[rgb]{0,0,0}\makebox(0,0)[lt]{\lineheight{1.25}\smash{\begin{tabular}[t]{l}$U_3$\end{tabular}}}}%
  \end{picture}%
\endgroup%

%% file: farey.pdf_tex
\begingroup%
  \makeatletter%
  \providecommand\color[2][]{%
    \errmessage{(Inkscape) Color is used for the text in Inkscape, but the package 'color.sty' is not loaded}%
    \renewcommand\color[2][]{}%
  }%
  \providecommand\transparent[1]{%
    \errmessage{(Inkscape) Transparency is used (non-zero) for the text in Inkscape, but the package 'transparent.sty' is not loaded}%
    \renewcommand\transparent[1]{}%
  }%
  \providecommand\rotatebox[2]{#2}%
  \newcommand*\fsize{\dimexpr\f@size pt\relax}%
  \newcommand*\lineheight[1]{\fontsize{\fsize}{#1\fsize}\selectfont}%
  \ifx\svgwidth\undefined%
    \setlength{\unitlength}{494.39756859bp}%
    \ifx\svgscale\undefined%
      \relax%
    \else%
      \setlength{\unitlength}{\unitlength * \real{\svgscale}}%
    \fi%
  \else%
    \setlength{\unitlength}{\svgwidth}%
  \fi%
  \global\let\svgwidth\undefined%
  \global\let\svgscale\undefined%
  \makeatother%
  \begin{picture}(1,0.36649403)%
    \lineheight{1}%
    \setlength\tabcolsep{0pt}%
    \put(0,0){\includegraphics[width=\unitlength,page=1]{farey.pdf}}%
    \put(0.00330304,0.13256002){\makebox(0,0)[lt]{\lineheight{1.25}\smash{\begin{tabular}[t]{l}$\alpha_1$\end{tabular}}}}%
    \put(0.16929471,0.13256002){\makebox(0,0)[lt]{\lineheight{1.25}\smash{\begin{tabular}[t]{l}$\alpha_2$\end{tabular}}}}%
    \put(0.34602215,0.18045803){\makebox(0,0)[lt]{\lineheight{1.25}\smash{\begin{tabular}[t]{l}$\alpha_3$\end{tabular}}}}%
    \put(0.51118791,0.18128391){\makebox(0,0)[lt]{\lineheight{1.25}\smash{\begin{tabular}[t]{l}$\alpha_4$\end{tabular}}}}%
    \put(0.67305047,0.13256002){\makebox(0,0)[lt]{\lineheight{1.25}\smash{\begin{tabular}[t]{l}$\alpha_5$\end{tabular}}}}%
    \put(0.00412887,0.26882184){\makebox(0,0)[lt]{\lineheight{1.25}\smash{\begin{tabular}[t]{l}$\beta_2$\end{tabular}}}}%
    \put(0.07927932,0.33323653){\makebox(0,0)[lt]{\lineheight{1.25}\smash{\begin{tabular}[t]{l}$\beta_3$\end{tabular}}}}%
    \put(0.20893452,0.34314647){\makebox(0,0)[lt]{\lineheight{1.25}\smash{\begin{tabular}[t]{l}$\beta_4$\end{tabular}}}}%
    \put(0.30885986,0.28286094){\makebox(0,0)[lt]{\lineheight{1.25}\smash{\begin{tabular}[t]{l}$\beta_5$\end{tabular}}}}%
  \end{picture}%
\endgroup%

%% file: twist_def.pdf_tex
\begingroup%
  \makeatletter%
  \providecommand\color[2][]{%
    \errmessage{(Inkscape) Color is used for the text in Inkscape, but the package 'color.sty' is not loaded}%
    \renewcommand\color[2][]{}%
  }%
  \providecommand\transparent[1]{%
    \errmessage{(Inkscape) Transparency is used (non-zero) for the text in Inkscape, but the package 'transparent.sty' is not loaded}%
    \renewcommand\transparent[1]{}%
  }%
  \providecommand\rotatebox[2]{#2}%
  \newcommand*\fsize{\dimexpr\f@size pt\relax}%
  \newcommand*\lineheight[1]{\fontsize{\fsize}{#1\fsize}\selectfont}%
  \ifx\svgwidth\undefined%
    \setlength{\unitlength}{438.68767205bp}%
    \ifx\svgscale\undefined%
      \relax%
    \else%
      \setlength{\unitlength}{\unitlength * \real{\svgscale}}%
    \fi%
  \else%
    \setlength{\unitlength}{\svgwidth}%
  \fi%
  \global\let\svgwidth\undefined%
  \global\let\svgscale\undefined%
  \makeatother%
  \begin{picture}(1,0.37465576)%
    \lineheight{1}%
    \setlength\tabcolsep{0pt}%
    \put(0,0){\includegraphics[width=\unitlength,page=1]{twist_def.pdf}}%
    \put(0.23715273,0.12049317){\color[rgb]{0.50196078,0,0.50196078}\makebox(0,0)[lt]{\lineheight{1.25}\smash{\begin{tabular}[t]{l}$\beta$\end{tabular}}}}%
    \put(0.22869045,0.03687158){\color[rgb]{1,0,0}\makebox(0,0)[lt]{\lineheight{1.25}\smash{\begin{tabular}[t]{l}$\beta'$\end{tabular}}}}%
    \put(0.16624191,0.00000003){\color[rgb]{0.52941176,0.66666667,0.87058824}\makebox(0,0)[lt]{\lineheight{1.25}\smash{\begin{tabular}[t]{l}A\end{tabular}}}}%
    \put(0,0){\includegraphics[width=\unitlength,page=2]{twist_def.pdf}}%
    \put(0.74530621,0.32487496){\color[rgb]{0.50196078,0,0.50196078}\makebox(0,0)[lt]{\lineheight{1.25}\smash{\begin{tabular}[t]{l}$\pi_A(\beta)$\end{tabular}}}}%
    \put(0.71750148,0.04651215){\color[rgb]{1,0,0}\makebox(0,0)[lt]{\lineheight{1.25}\smash{\begin{tabular}[t]{l}$\pi_A(\beta')$\end{tabular}}}}%
  \end{picture}%
\endgroup%

%% file: lem_6_10.pdf_tex
\begingroup%
  \makeatletter%
  \providecommand\color[2][]{%
    \errmessage{(Inkscape) Color is used for the text in Inkscape, but the package 'color.sty' is not loaded}%
    \renewcommand\color[2][]{}%
  }%
  \providecommand\transparent[1]{%
    \errmessage{(Inkscape) Transparency is used (non-zero) for the text in Inkscape, but the package 'transparent.sty' is not loaded}%
    \renewcommand\transparent[1]{}%
  }%
  \providecommand\rotatebox[2]{#2}%
  \newcommand*\fsize{\dimexpr\f@size pt\relax}%
  \newcommand*\lineheight[1]{\fontsize{\fsize}{#1\fsize}\selectfont}%
  \ifx\svgwidth\undefined%
    \setlength{\unitlength}{414.71758024bp}%
    \ifx\svgscale\undefined%
      \relax%
    \else%
      \setlength{\unitlength}{\unitlength * \real{\svgscale}}%
    \fi%
  \else%
    \setlength{\unitlength}{\svgwidth}%
  \fi%
  \global\let\svgwidth\undefined%
  \global\let\svgscale\undefined%
  \makeatother%
  \begin{picture}(1,0.24375209)%
    \lineheight{1}%
    \setlength\tabcolsep{0pt}%
    \put(0,0){\includegraphics[width=\unitlength,page=1]{lem_6_10.pdf}}%
    \put(0.33174551,0.06039738){\color[rgb]{1,0,0}\makebox(0,0)[lt]{\lineheight{1.25}\smash{\begin{tabular}[t]{l}$b_0$\end{tabular}}}}%
    \put(0.40135992,0.16718314){\color[rgb]{0,0.50196078,0.50196078}\makebox(0,0)[lt]{\lineheight{1.25}\smash{\begin{tabular}[t]{l}$\alpha$\end{tabular}}}}%
  \end{picture}%
\endgroup%

%% file: prop_7_3.pdf_tex
\begingroup%
  \makeatletter%
  \providecommand\color[2][]{%
    \errmessage{(Inkscape) Color is used for the text in Inkscape, but the package 'color.sty' is not loaded}%
    \renewcommand\color[2][]{}%
  }%
  \providecommand\transparent[1]{%
    \errmessage{(Inkscape) Transparency is used (non-zero) for the text in Inkscape, but the package 'transparent.sty' is not loaded}%
    \renewcommand\transparent[1]{}%
  }%
  \providecommand\rotatebox[2]{#2}%
  \newcommand*\fsize{\dimexpr\f@size pt\relax}%
  \newcommand*\lineheight[1]{\fontsize{\fsize}{#1\fsize}\selectfont}%
  \ifx\svgwidth\undefined%
    \setlength{\unitlength}{764.21680529bp}%
    \ifx\svgscale\undefined%
      \relax%
    \else%
      \setlength{\unitlength}{\unitlength * \real{\svgscale}}%
    \fi%
  \else%
    \setlength{\unitlength}{\svgwidth}%
  \fi%
  \global\let\svgwidth\undefined%
  \global\let\svgscale\undefined%
  \makeatother%
  \begin{picture}(1,0.18903654)%
    \lineheight{1}%
    \setlength\tabcolsep{0pt}%
    \put(0,0){\includegraphics[width=\unitlength,page=1]{prop_7_3.pdf}}%
    \put(0.05840121,0.09872415){\makebox(0,0)[lt]{\lineheight{1.25}\smash{\begin{tabular}[t]{l}$m_0^-$\end{tabular}}}}%
    \put(0.13082016,0.10194609){\makebox(0,0)[lt]{\lineheight{1.25}\smash{\begin{tabular}[t]{l}$m_0^+$\end{tabular}}}}%
    \put(0.19526762,0.09939753){\makebox(0,0)[lt]{\lineheight{1.25}\smash{\begin{tabular}[t]{l}$m_1^-$\end{tabular}}}}%
    \put(0.2628438,0.09890679){\makebox(0,0)[lt]{\lineheight{1.25}\smash{\begin{tabular}[t]{l}$m_1^+$\end{tabular}}}}%
    \put(0.11072729,0.16851593){\makebox(0,0)[lt]{\lineheight{1.25}\smash{\begin{tabular}[t]{l}$\sigma_0$\end{tabular}}}}%
    \put(0.24442663,0.16829886){\makebox(0,0)[lt]{\lineheight{1.25}\smash{\begin{tabular}[t]{l}$\sigma_1$\end{tabular}}}}%
    \put(0.60143704,0.17307242){\makebox(0,0)[lt]{\lineheight{1.25}\smash{\begin{tabular}[t]{l}$\sigma_{n-1}$\end{tabular}}}}%
    \put(0.72123745,0.1741239){\makebox(0,0)[lt]{\lineheight{1.25}\smash{\begin{tabular}[t]{l}$\sigma_n$\end{tabular}}}}%
    \put(0.09797437,0.00665056){\makebox(0,0)[lt]{\lineheight{1.25}\smash{\begin{tabular}[t]{l}$(3)$\end{tabular}}}}%
    \put(0.16889073,0.00288477){\makebox(0,0)[lt]{\lineheight{1.25}\smash{\begin{tabular}[t]{l}$(2)$\end{tabular}}}}%
    \put(0.23062709,0.00805875){\makebox(0,0)[lt]{\lineheight{1.25}\smash{\begin{tabular}[t]{l}$(4)$\end{tabular}}}}%
    \put(0.31547102,0.09814911){\makebox(0,0)[lt]{\lineheight{1.25}\smash{\begin{tabular}[t]{l}$m_2^-$\end{tabular}}}}%
    \put(0.3877806,0.09571056){\makebox(0,0)[lt]{\lineheight{1.25}\smash{\begin{tabular}[t]{l}$m_2^+$\end{tabular}}}}%
    \put(0.35449715,0.01008632){\makebox(0,0)[lt]{\lineheight{1.25}\smash{\begin{tabular}[t]{l}$(5)$\end{tabular}}}}%
    \put(0.29349934,0.00526331){\makebox(0,0)[lt]{\lineheight{1.25}\smash{\begin{tabular}[t]{l}$(2)$\end{tabular}}}}%
    \put(0.3622762,0.17057823){\makebox(0,0)[lt]{\lineheight{1.25}\smash{\begin{tabular}[t]{l}$\sigma_2$\end{tabular}}}}%
    \put(0.05693713,0.04217708){\makebox(0,0)[lt]{\lineheight{1.25}\smash{\begin{tabular}[t]{l}$\tau_0^-$\end{tabular}}}}%
    \put(0.14664994,0.04892165){\makebox(0,0)[lt]{\lineheight{1.25}\smash{\begin{tabular}[t]{l}$\tau_0^+$\end{tabular}}}}%
    \put(0.31962598,0.0555227){\makebox(0,0)[lt]{\lineheight{1.25}\smash{\begin{tabular}[t]{l}$\tau_2^-$\end{tabular}}}}%
    \put(0.19449004,0.05221165){\makebox(0,0)[lt]{\lineheight{1.25}\smash{\begin{tabular}[t]{l}$\tau_1^-$\end{tabular}}}}%
    \put(0.40151038,0.05123171){\makebox(0,0)[lt]{\lineheight{1.25}\smash{\begin{tabular}[t]{l}$\tau_2^+$\end{tabular}}}}%
    \put(0.27749924,0.05194125){\makebox(0,0)[lt]{\lineheight{1.25}\smash{\begin{tabular}[t]{l}$\tau_1^+$\end{tabular}}}}%
    \put(0.66968129,0.05787095){\makebox(0,0)[lt]{\lineheight{1.25}\smash{\begin{tabular}[t]{l}$\tau_n^-$\end{tabular}}}}%
    \put(0.75542613,0.05685428){\makebox(0,0)[lt]{\lineheight{1.25}\smash{\begin{tabular}[t]{l}$\tau_n^+$\end{tabular}}}}%
    \put(0.4727123,0.1725211){\makebox(0,0)[lt]{\lineheight{1.25}\smash{\begin{tabular}[t]{l}$\dots$\end{tabular}}}}%
    \put(0.59090232,0.01534132){\makebox(0,0)[lt]{\lineheight{1.25}\smash{\begin{tabular}[t]{l}$(6)$\end{tabular}}}}%
    \put(0.71088293,0.01230285){\makebox(0,0)[lt]{\lineheight{1.25}\smash{\begin{tabular}[t]{l}$(7)$\end{tabular}}}}%
    \put(0,0){\includegraphics[width=\unitlength,page=2]{prop_7_3.pdf}}%
    \put(0.0508733,0.16781493){\makebox(0,0)[lt]{\lineheight{1.25}\smash{\begin{tabular}[t]{l}$\color{red}\alpha=$\end{tabular}}}}%
    \put(-0.00162927,0.04206317){\makebox(0,0)[lt]{\lineheight{1.25}\smash{\begin{tabular}[t]{l}$\color{red}\beta=$\end{tabular}}}}%
    \put(0.76362709,0.1730124){\makebox(0,0)[lt]{\lineheight{1.25}\smash{\begin{tabular}[t]{l}$\color{red}=\alpha'$\end{tabular}}}}%
    \put(0.79184597,0.05677663){\makebox(0,0)[lt]{\lineheight{1.25}\smash{\begin{tabular}[t]{l}$\color{red}=\beta'$\end{tabular}}}}%
  \end{picture}%
\endgroup%

%% file: prop_7_8.pdf_tex
\begingroup%
  \makeatletter%
  \providecommand\color[2][]{%
    \errmessage{(Inkscape) Color is used for the text in Inkscape, but the package 'color.sty' is not loaded}%
    \renewcommand\color[2][]{}%
  }%
  \providecommand\transparent[1]{%
    \errmessage{(Inkscape) Transparency is used (non-zero) for the text in Inkscape, but the package 'transparent.sty' is not loaded}%
    \renewcommand\transparent[1]{}%
  }%
  \providecommand\rotatebox[2]{#2}%
  \newcommand*\fsize{\dimexpr\f@size pt\relax}%
  \newcommand*\lineheight[1]{\fontsize{\fsize}{#1\fsize}\selectfont}%
  \ifx\svgwidth\undefined%
    \setlength{\unitlength}{412.17113104bp}%
    \ifx\svgscale\undefined%
      \relax%
    \else%
      \setlength{\unitlength}{\unitlength * \real{\svgscale}}%
    \fi%
  \else%
    \setlength{\unitlength}{\svgwidth}%
  \fi%
  \global\let\svgwidth\undefined%
  \global\let\svgscale\undefined%
  \makeatother%
  \begin{picture}(1,0.28558635)%
    \lineheight{1}%
    \setlength\tabcolsep{0pt}%
    \put(0,0){\includegraphics[width=\unitlength,page=1]{prop_7_8.pdf}}%
    \put(0.28296088,0.22575447){\color[rgb]{1,0,0}\makebox(0,0)[lt]{\lineheight{1.25}\smash{\begin{tabular}[t]{l}$\tau$\end{tabular}}}}%
    \put(0.65542524,0.18128586){\color[rgb]{0,0.50196078,0.50196078}\makebox(0,0)[lt]{\lineheight{1.25}\smash{\begin{tabular}[t]{l}$\sigma'$\end{tabular}}}}%
    \put(0.07923157,0.00751668){\color[rgb]{0,0,0}\makebox(0,0)[lt]{\lineheight{1.25}\smash{\begin{tabular}[t]{l}$\sigma$\end{tabular}}}}%
    \put(0,0){\includegraphics[width=\unitlength,page=2]{prop_7_8.pdf}}%
    \put(0.35881507,0.15389493){\makebox(0,0)[lt]{\lineheight{1.25}\smash{\begin{tabular}[t]{l}$x$\end{tabular}}}}%
    \put(0,0){\includegraphics[width=\unitlength,page=3]{prop_7_8.pdf}}%
    \put(0.4969538,0.14664756){\makebox(0,0)[lt]{\lineheight{1.25}\smash{\begin{tabular}[t]{l}$D$\end{tabular}}}}%
    \put(0,0){\includegraphics[width=\unitlength,page=4]{prop_7_8.pdf}}%
    \put(0.17599895,0.06312651){\makebox(0,0)[lt]{\lineheight{1.25}\smash{\begin{tabular}[t]{l}$c'$\end{tabular}}}}%
    \put(0.36091639,0.2242288){\makebox(0,0)[lt]{\lineheight{1.25}\smash{\begin{tabular}[t]{l}$c$\end{tabular}}}}%
  \end{picture}%
\endgroup%

%% file: prop_7_9.pdf_tex
\begingroup%
  \makeatletter%
  \providecommand\color[2][]{%
    \errmessage{(Inkscape) Color is used for the text in Inkscape, but the package 'color.sty' is not loaded}%
    \renewcommand\color[2][]{}%
  }%
  \providecommand\transparent[1]{%
    \errmessage{(Inkscape) Transparency is used (non-zero) for the text in Inkscape, but the package 'transparent.sty' is not loaded}%
    \renewcommand\transparent[1]{}%
  }%
  \providecommand\rotatebox[2]{#2}%
  \newcommand*\fsize{\dimexpr\f@size pt\relax}%
  \newcommand*\lineheight[1]{\fontsize{\fsize}{#1\fsize}\selectfont}%
  \ifx\svgwidth\undefined%
    \setlength{\unitlength}{362.74759848bp}%
    \ifx\svgscale\undefined%
      \relax%
    \else%
      \setlength{\unitlength}{\unitlength * \real{\svgscale}}%
    \fi%
  \else%
    \setlength{\unitlength}{\svgwidth}%
  \fi%
  \global\let\svgwidth\undefined%
  \global\let\svgscale\undefined%
  \makeatother%
  \begin{picture}(1,0.23156597)%
    \lineheight{1}%
    \setlength\tabcolsep{0pt}%
    \put(0,0){\includegraphics[width=\unitlength,page=1]{prop_7_9.pdf}}%
    \put(0.79669724,0.00854083){\makebox(0,0)[lt]{\lineheight{1.25}\smash{\begin{tabular}[t]{l}$\sigma$\end{tabular}}}}%
    \put(0.79669724,0.08201296){\makebox(0,0)[lt]{\lineheight{1.25}\smash{\begin{tabular}[t]{l}$\sigma'$\end{tabular}}}}%
    \put(0.28961895,0.14128284){\makebox(0,0)[lt]{\lineheight{1.25}\smash{\begin{tabular}[t]{l}$\tau$\end{tabular}}}}%
    \put(0,0){\includegraphics[width=\unitlength,page=2]{prop_7_9.pdf}}%
    \put(0.19498466,0.04630562){\color[rgb]{1,0,0}\makebox(0,0)[lt]{\lineheight{1.25}\smash{\begin{tabular}[t]{l}$\gamma$\end{tabular}}}}%
  \end{picture}%
\endgroup%

%% file: lem_9_2.pdf_tex
\begingroup%
  \makeatletter%
  \providecommand\color[2][]{%
    \errmessage{(Inkscape) Color is used for the text in Inkscape, but the package 'color.sty' is not loaded}%
    \renewcommand\color[2][]{}%
  }%
  \providecommand\transparent[1]{%
    \errmessage{(Inkscape) Transparency is used (non-zero) for the text in Inkscape, but the package 'transparent.sty' is not loaded}%
    \renewcommand\transparent[1]{}%
  }%
  \providecommand\rotatebox[2]{#2}%
  \newcommand*\fsize{\dimexpr\f@size pt\relax}%
  \newcommand*\lineheight[1]{\fontsize{\fsize}{#1\fsize}\selectfont}%
  \ifx\svgwidth\undefined%
    \setlength{\unitlength}{279.85932117bp}%
    \ifx\svgscale\undefined%
      \relax%
    \else%
      \setlength{\unitlength}{\unitlength * \real{\svgscale}}%
    \fi%
  \else%
    \setlength{\unitlength}{\svgwidth}%
  \fi%
  \global\let\svgwidth\undefined%
  \global\let\svgscale\undefined%
  \makeatother%
  \begin{picture}(1,0.36224006)%
    \lineheight{1}%
    \setlength\tabcolsep{0pt}%
    \put(0,0){\includegraphics[width=\unitlength,page=1]{lem_9_2.pdf}}%
    \put(0.19334004,0.01078247){\makebox(0,0)[lt]{\lineheight{1.25}\smash{\begin{tabular}[t]{l}$m_0$\end{tabular}}}}%
    \put(0,0){\includegraphics[width=\unitlength,page=2]{lem_9_2.pdf}}%
    \put(0.66998255,0.01078247){\makebox(0,0)[lt]{\lineheight{1.25}\smash{\begin{tabular}[t]{l}$f^n(m_0)$\end{tabular}}}}%
    \put(0,0){\includegraphics[width=\unitlength,page=3]{lem_9_2.pdf}}%
    \put(0.21056809,0.29408808){\makebox(0,0)[lt]{\lineheight{1.25}\smash{\begin{tabular}[t]{l}$h(m_0)$\end{tabular}}}}%
    \put(0,0){\includegraphics[width=\unitlength,page=4]{lem_9_2.pdf}}%
    \put(0.58192812,0.29791652){\makebox(0,0)[lt]{\lineheight{1.25}\smash{\begin{tabular}[t]{l}$hf^n(m_0)$\end{tabular}}}}%
    \put(0,0){\includegraphics[width=\unitlength,page=5]{lem_9_2.pdf}}%
    \put(0.11677095,0.12946456){\makebox(0,0)[lt]{\lineheight{1.25}\smash{\begin{tabular}[t]{l}$\partial A$\end{tabular}}}}%
    \put(0,0){\includegraphics[width=\unitlength,page=6]{lem_9_2.pdf}}%
  \end{picture}%
\endgroup%

%% file: theorem4.pdf_tex
\begingroup%
  \makeatletter%
  \providecommand\color[2][]{%
    \errmessage{(Inkscape) Color is used for the text in Inkscape, but the package 'color.sty' is not loaded}%
    \renewcommand\color[2][]{}%
  }%
  \providecommand\transparent[1]{%
    \errmessage{(Inkscape) Transparency is used (non-zero) for the text in Inkscape, but the package 'transparent.sty' is not loaded}%
    \renewcommand\transparent[1]{}%
  }%
  \providecommand\rotatebox[2]{#2}%
  \newcommand*\fsize{\dimexpr\f@size pt\relax}%
  \newcommand*\lineheight[1]{\fontsize{\fsize}{#1\fsize}\selectfont}%
  \ifx\svgwidth\undefined%
    \setlength{\unitlength}{551.17071677bp}%
    \ifx\svgscale\undefined%
      \relax%
    \else%
      \setlength{\unitlength}{\unitlength * \real{\svgscale}}%
    \fi%
  \else%
    \setlength{\unitlength}{\svgwidth}%
  \fi%
  \global\let\svgwidth\undefined%
  \global\let\svgscale\undefined%
  \makeatother%
  \begin{picture}(1,0.22864443)%
    \lineheight{1}%
    \setlength\tabcolsep{0pt}%
    \put(0,0){\includegraphics[width=\unitlength,page=1]{theorem4.pdf}}%
    \put(0.28398297,0.02668745){\makebox(0,0)[lt]{\lineheight{1.25}\smash{\begin{tabular}[t]{l}$\alpha_0$\end{tabular}}}}%
    \put(0.66511612,0.0339029){\makebox(0,0)[lt]{\lineheight{1.25}\smash{\begin{tabular}[t]{l}$f^{2n}(\alpha_0)$\end{tabular}}}}%
    \put(0,0){\includegraphics[width=\unitlength,page=2]{theorem4.pdf}}%
    \put(0.93978562,0.01211244){\makebox(0,0)[lt]{\lineheight{1.25}\smash{\begin{tabular}[t]{l}$A$\end{tabular}}}}%
    \put(0,0){\includegraphics[width=\unitlength,page=3]{theorem4.pdf}}%
    \put(0.47131011,0.02342685){\makebox(0,0)[lt]{\lineheight{1.25}\smash{\begin{tabular}[t]{l}$f^{n}(\alpha_0)$\end{tabular}}}}%
    \put(0,0){\includegraphics[width=\unitlength,page=4]{theorem4.pdf}}%
    \put(0.29131617,0.14820901){\makebox(0,0)[lt]{\lineheight{1.25}\smash{\begin{tabular}[t]{l}$h\beta_0$\end{tabular}}}}%
    \put(0.67559209,0.15311289){\makebox(0,0)[lt]{\lineheight{1.25}\smash{\begin{tabular}[t]{l}$hg^{2n}(\beta_0)$\end{tabular}}}}%
    \put(0,0){\includegraphics[width=\unitlength,page=5]{theorem4.pdf}}%
    \put(0.93140485,0.18391878){\makebox(0,0)[lt]{\lineheight{1.25}\smash{\begin{tabular}[t]{l}$hB$\end{tabular}}}}%
    \put(0,0){\includegraphics[width=\unitlength,page=6]{theorem4.pdf}}%
    \put(0.4775957,0.14263683){\makebox(0,0)[lt]{\lineheight{1.25}\smash{\begin{tabular}[t]{l}$hg^{n}(\beta_0)$\end{tabular}}}}%
  \end{picture}%
\endgroup%

%% file: triangle_example2.pdf_tex
\begingroup%
  \makeatletter%
  \providecommand\color[2][]{%
    \errmessage{(Inkscape) Color is used for the text in Inkscape, but the package 'color.sty' is not loaded}%
    \renewcommand\color[2][]{}%
  }%
  \providecommand\transparent[1]{%
    \errmessage{(Inkscape) Transparency is used (non-zero) for the text in Inkscape, but the package 'transparent.sty' is not loaded}%
    \renewcommand\transparent[1]{}%
  }%
  \providecommand\rotatebox[2]{#2}%
  \newcommand*\fsize{\dimexpr\f@size pt\relax}%
  \newcommand*\lineheight[1]{\fontsize{\fsize}{#1\fsize}\selectfont}%
  \ifx\svgwidth\undefined%
    \setlength{\unitlength}{429.4733917bp}%
    \ifx\svgscale\undefined%
      \relax%
    \else%
      \setlength{\unitlength}{\unitlength * \real{\svgscale}}%
    \fi%
  \else%
    \setlength{\unitlength}{\svgwidth}%
  \fi%
  \global\let\svgwidth\undefined%
  \global\let\svgscale\undefined%
  \makeatother%
  \begin{picture}(1,0.99028342)%
    \lineheight{1}%
    \setlength\tabcolsep{0pt}%
    \put(0,0){\includegraphics[width=\unitlength,page=1]{triangle_example2.pdf}}%
    \put(0.10763348,0.11457571){\color[rgb]{0.50196078,0,0.50196078}\makebox(0,0)[lt]{\lineheight{1.25}\smash{\begin{tabular}[t]{l}$R$\end{tabular}}}}%
    \put(0.10763348,0.31358315){\color[rgb]{0,0.50196078,0}\makebox(0,0)[lt]{\lineheight{1.25}\smash{\begin{tabular}[t]{l}$H$\end{tabular}}}}%
    \put(0.31153961,0.11457571){\color[rgb]{0,0,1}\makebox(0,0)[lt]{\lineheight{1.25}\smash{\begin{tabular}[t]{l}$V$\end{tabular}}}}%
    \put(0.3069036,0.31358315){\color[rgb]{1,0,0}\makebox(0,0)[lt]{\lineheight{1.25}\smash{\begin{tabular}[t]{l}$A$\end{tabular}}}}%
    \put(0.6564054,0.31358315){\color[rgb]{0,0.50196078,0}\makebox(0,0)[lt]{\lineheight{1.25}\smash{\begin{tabular}[t]{l}$SH$\end{tabular}}}}%
    \put(0.30935582,0.67907268){\color[rgb]{0,0,1}\makebox(0,0)[lt]{\lineheight{1.25}\smash{\begin{tabular}[t]{l}$TV$\end{tabular}}}}%
    \put(0.37282248,0.17511579){\color[rgb]{0,0,1}\makebox(0,0)[lt]{\lineheight{1.25}\smash{\begin{tabular}[t]{l}$\widetilde{f}V$\end{tabular}}}}%
    \put(0.14608224,0.38544662){\color[rgb]{0,0.50196078,0}\makebox(0,0)[lt]{\lineheight{1.25}\smash{\begin{tabular}[t]{l}$\widetilde{f}H$\end{tabular}}}}%
    \put(0,0){\includegraphics[width=\unitlength,page=2]{triangle_example2.pdf}}%
  \end{picture}%
\endgroup%

%% file: h-axis.pdf_tex
\begingroup%
  \makeatletter%
  \providecommand\color[2][]{%
    \errmessage{(Inkscape) Color is used for the text in Inkscape, but the package 'color.sty' is not loaded}%
    \renewcommand\color[2][]{}%
  }%
  \providecommand\transparent[1]{%
    \errmessage{(Inkscape) Transparency is used (non-zero) for the text in Inkscape, but the package 'transparent.sty' is not loaded}%
    \renewcommand\transparent[1]{}%
  }%
  \providecommand\rotatebox[2]{#2}%
  \newcommand*\fsize{\dimexpr\f@size pt\relax}%
  \newcommand*\lineheight[1]{\fontsize{\fsize}{#1\fsize}\selectfont}%
  \ifx\svgwidth\undefined%
    \setlength{\unitlength}{812.70159864bp}%
    \ifx\svgscale\undefined%
      \relax%
    \else%
      \setlength{\unitlength}{\unitlength * \real{\svgscale}}%
    \fi%
  \else%
    \setlength{\unitlength}{\svgwidth}%
  \fi%
  \global\let\svgwidth\undefined%
  \global\let\svgscale\undefined%
  \makeatother%
  \begin{picture}(1,0.16050073)%
    \lineheight{1}%
    \setlength\tabcolsep{0pt}%
    \put(0,0){\includegraphics[width=\unitlength,page=1]{h-axis.pdf}}%
    \put(0.40210416,0.01134326){\color[rgb]{0,0,0}\makebox(0,0)[lt]{\lineheight{1.25}\smash{\begin{tabular}[t]{l}$\alpha$\end{tabular}}}}%
    \put(0.53788022,0.1408853){\color[rgb]{0,0,0}\makebox(0,0)[lt]{\lineheight{1.25}\smash{\begin{tabular}[t]{l}$\beta$\end{tabular}}}}%
    \put(0.68127634,0.01134326){\color[rgb]{0,0,0}\makebox(0,0)[lt]{\lineheight{1.25}\smash{\begin{tabular}[t]{l}$h(\alpha)$\end{tabular}}}}%
    \put(0.83056318,0.14490353){\color[rgb]{0,0,0}\makebox(0,0)[lt]{\lineheight{1.25}\smash{\begin{tabular}[t]{l}$h(\beta)$\end{tabular}}}}%
    \put(0.15744216,0.00396058){\color[rgb]{0,0,0}\makebox(0,0)[lt]{\lineheight{1.25}\smash{\begin{tabular}[t]{l}$h^{-1}(\alpha)$\end{tabular}}}}%
    \put(0.27655393,0.14535067){\color[rgb]{0,0,0}\makebox(0,0)[lt]{\lineheight{1.25}\smash{\begin{tabular}[t]{l}$h^{-1}(\beta)$\end{tabular}}}}%
    \put(0,0){\includegraphics[width=\unitlength,page=2]{h-axis.pdf}}%
    \put(0.53623903,0.05429706){\color[rgb]{0,0,0}\makebox(0,0)[lt]{\lineheight{1.25}\smash{\begin{tabular}[t]{l}$g^q$\end{tabular}}}}%
    \put(0,0){\includegraphics[width=\unitlength,page=3]{h-axis.pdf}}%
    \put(0.38513677,0.08677117){\color[rgb]{0,0,0}\makebox(0,0)[lt]{\lineheight{1.25}\smash{\begin{tabular}[t]{l}$f^p$\end{tabular}}}}%
  \end{picture}%
\endgroup%

%% file: lem_11_14.pdf_tex
\begingroup%
  \makeatletter%
  \providecommand\color[2][]{%
    \errmessage{(Inkscape) Color is used for the text in Inkscape, but the package 'color.sty' is not loaded}%
    \renewcommand\color[2][]{}%
  }%
  \providecommand\transparent[1]{%
    \errmessage{(Inkscape) Transparency is used (non-zero) for the text in Inkscape, but the package 'transparent.sty' is not loaded}%
    \renewcommand\transparent[1]{}%
  }%
  \providecommand\rotatebox[2]{#2}%
  \newcommand*\fsize{\dimexpr\f@size pt\relax}%
  \newcommand*\lineheight[1]{\fontsize{\fsize}{#1\fsize}\selectfont}%
  \ifx\svgwidth\undefined%
    \setlength{\unitlength}{336.55978225bp}%
    \ifx\svgscale\undefined%
      \relax%
    \else%
      \setlength{\unitlength}{\unitlength * \real{\svgscale}}%
    \fi%
  \else%
    \setlength{\unitlength}{\svgwidth}%
  \fi%
  \global\let\svgwidth\undefined%
  \global\let\svgscale\undefined%
  \makeatother%
  \begin{picture}(1,0.32850693)%
    \lineheight{1}%
    \setlength\tabcolsep{0pt}%
    \put(0,0){\includegraphics[width=\unitlength,page=1]{lem_11_14.pdf}}%
    \put(0.03993412,0.11224168){\makebox(0,0)[lt]{\lineheight{1.25}\smash{\begin{tabular}[t]{l}$b_0$\end{tabular}}}}%
    \put(0.41621716,0.29421001){\makebox(0,0)[lt]{\lineheight{1.25}\smash{\begin{tabular}[t]{l}$f^p(b_0)$\end{tabular}}}}%
    \put(0.4317724,0.02056305){\makebox(0,0)[lt]{\lineheight{1.25}\smash{\begin{tabular}[t]{l}$a_0$\end{tabular}}}}%
    \put(0.73049897,0.29421001){\makebox(0,0)[lt]{\lineheight{1.25}\smash{\begin{tabular}[t]{l}$g^q(a_0)$\end{tabular}}}}%
  \end{picture}%
\endgroup%